\documentclass[11pt, a4paper]{amsart}

\usepackage{amsmath,amssymb,amscd,amsthm,indentfirst}
\usepackage{mathtools}
\usepackage{amsfonts}
\usepackage[mathscr]{eucal}
\usepackage{appendix}
\usepackage{float}
\usepackage{rotating}

\usepackage{booktabs}
\usepackage{verbatim}
\usepackage{enumerate}

\usepackage[small,nohug,heads=LaTeX]{diagrams}
\diagramstyle[labelstyle=\scriptstyle]
\usepackage{graphicx}
\usepackage{pstricks}
\usepackage{microtype}

\newtheorem{thm}{Theorem}[section]
\newtheorem{cor}[thm]{Corollary}
\newtheorem{lem}[thm]{Lemma}
\newtheorem{prop}[thm]{Proposition}
\theoremstyle{definition}
\newtheorem{defn}[thm]{Definition}
\newtheorem{conv}[thm]{Convention}
\theoremstyle{remark}
\newtheorem{rem}[thm]{Remark}

\numberwithin{equation}{section}

\newcommand{\bA}{\mathbb{A}}
\newcommand{\bB}{\mathbb{B}}
\newcommand{\bC}{\mathbb{C}}

\newcommand{\bF}{\mathbb{F}}

\newcommand{\bI}{\mathbb{I}}

\newcommand{\bP}{\mathbb{P}}
\newcommand{\bQ}{\mathbb{Q}}
\newcommand{\bR}{\mathbb{R}}

\newcommand{\bZ}{\mathbb{Z}}

\newcommand{\gM}{\bold{M}}

\newcommand\lra{\longrightarrow}
\newcommand\lla{\longleftarrow}

\newcommand\Diff{\mathrm{Diff}}
\newcommand\Emb{\mathrm{Emb}}
\newcommand\Bun{\mathrm{Bun}}

\newcommand\colim{\mathrm{colim}\,}
\newcommand\hocolim{\mathrm{hocolim}\,}

\newcommand{\hcoker}{/\!\!/}

\newarrow {Onto} ----{>>}
\newarrow {Equals} ===={=}

\newcommand{\CircNum}[1]{\ooalign{\hfil\raise .00ex\hbox{\scriptsize #1}\hfil\crcr\mathhexbox20D}}

\newcommand{\X}{B}

\newcommand{\map}{\mathrm{map}}

\newcommand{\Link}{\mathrm{Link}}

\newcommand{\sticks}{\text{\rotatebox[origin=c]{90}{$\shortparallel$}}}
\newcommand{\cross}{\shortparallel}

\mathchardef\ordinarycolon\mathcode`\:
\mathcode`\:=\string"8000
\begingroup \catcode`\:=\active
  \gdef:{\mathrel{\mathop\ordinarycolon}}
\endgroup

\usepackage[all]{xy}
\SelectTips{cm}{10}
\CompileMatrices

\title[Resolutions of moduli spaces]{Resolutions of moduli spaces\\and homological stability}
\author{Oscar Randal-Williams}
\email{o.randal-williams@dpmms.cam.ac.uk}
\address{Centre for Mathematical Sciences, Wilberforce Road, Cambridge CB3 0WB, UK}

\date{\today}

\begin{document}

\begin{abstract}
We describe partial semi-simplicial resolutions of moduli spaces of surfaces with tangential structure. This allows us to prove a homological stability theorem for these moduli spaces, which often improves the known stability ranges and give explicit stability ranges in many new cases. In each of these cases the stable homology can be identified using the methods of Galatius, Madsen, Tillmann and Weiss.
\end{abstract}
\maketitle

\section{Introduction and statement of results}

Let $\Sigma_{g, b}$ be a fixed oriented surface of genus $g$ with $b$ boundary components, and
$$\Gamma_{g,b} := \pi_0(\Diff^+(\Sigma_{g,b}, \partial \Sigma_{g, b}))$$
denote its \textit{mapping class group}: the group of isotopy classes of orientation-preserving diffeomorphisms. Over the last thirty years there has been an intense interest in the homological aspects of this family of groups, stemming principally from the rational homology equivalence $B\Gamma_g \simeq_\bQ \mathbf{M}_g$ from the classifying space of $\Gamma_g$ to the moduli space of Riemann surfaces of genus $g$.

The fundamental contribution in this direction is due to Harer \cite{H}, in which, inspired by the formal similarity between the family of mapping class groups and the classical groups, he shows that these groups exhibit \textit{homological stability}: each of the natural maps between the $\Gamma_{g,b}$ induced by inclusions of surfaces gives a homology isomorphism in a range of degrees which tends to infinity with $g$. This stability range was later improved by Ivanov \cite{Ivanov} and Boldsen \cite{Boldsen}, and extended to also deal with certain coefficient modules.

More recently, there has been a great deal of interest in generalising and extending these results. Wahl \cite{Wahl} has extended the techniques of Harer and Ivanov to prove homological stability for mapping class groups of non-orientable surfaces. Cohen and Madsen \cite{CM, CM2} have defined certain moduli spaces $\mathcal{S}_{g,b}(X)$ of ``surfaces $\Sigma_{g,b}$ in a background space $X$" (which specialise to $B\Gamma_{g,b}$ when $X$ is a point) and used techniques of Ivanov to prove homological stability for these. The purpose of this paper is to generalise the above results to \textit{moduli spaces of surfaces with tangential structure}, which we now define.

\subsection{Moduli spaces of surfaces}

A \emph{tangential structure} is a map $\theta : \X \to BO(2)$ from a path-connected space $\X$, which classifies the bundle $\theta^*\gamma_2 \to \X$ pulled back from the universal bundle $\gamma_2 \to BO(2)$ via $\theta$. A \emph{$\theta$-structure} on a surface $F$ is a bundle map $TF \to \theta^*\gamma_2$, i.e.\ a continuous map between total spaces which is a linear isomorphism on each fibre, and we denote by $\Bun(TF, \theta^*\gamma_2)$ the space of all $\theta$-structures on $F$, endowed with the compact-open topology. 

If $\ell_{\partial F} : \epsilon^1 \oplus T(\partial F) \to \theta^*\gamma_2$ is a bundle map, which we call a \textit{boundary condition}, and $c : (-1,0] \times \partial F \hookrightarrow F$ is a collar, then we define $\Bun_\partial(TF, \theta^*\gamma_2;\ell_{\partial F})$ to be the space of bundle maps $\ell :TF \to \theta^*\gamma_2$ such that $Dc\vert_{\{0\} \times \partial F} \circ \ell\vert_{\partial F} = \ell_{\partial F}$. Let $\Diff_\partial(F)$ denote the group of diffeomorphisms of $F$ which restrict to the identity diffeomorphism on a neighbourhood of the boundary.

\begin{defn}\label{defn:BorelConstModel}
The \textit{moduli space of surfaces of topological type $F$ with $\theta$-structure and boundary condition $\ell_{\partial F}$} is the homotopy quotient (or Borel construction),
$$\mathcal{M}^\theta(F;\ell_{\partial F}) := \Bun_\partial(TF, \theta^*\gamma_2;\ell_{\partial F}) \hcoker \Diff_\partial(F).$$
This will not necessarily be path connected. If we do not wish to introduce notation for a boundary condition, we may write $\mathcal{M}^\theta(F)$ to denote $\mathcal{M}^\theta(F;\ell_{\partial F})$ with a fixed but unspecified boundary condition $\ell_{\partial F}$.
\end{defn}

In Section \ref{sec:ModuliSurfaces} we will give a precise definition of the topology on these spaces, and a particular model for the homotopy quotient. If we define
$$\mathcal{E}^\theta(F;\ell_{\partial F}) := (\Bun_\partial(TF, \theta^*\gamma_2;\ell_{\partial F}) \times F) \hcoker {\Diff_\partial(F)},$$
where the group acts diagonally, then projection to the first factor gives a smooth $F$-bundle
$$F \lra \mathcal{E}^\theta(F;\ell_{\partial F}) \overset{\pi}\lra \mathcal{M}^\theta(F;\ell_{\partial F})$$
equipped with a bundle map $T_\pi \mathcal{E}^\theta(F;\ell_{\partial F}) \to \theta^*\gamma_2$ from the vertical tangent bundle, satisfying appropriate boundary conditions. The bundle $\pi$ is universal among fibre bundles enjoying these properties. 

Examples of tangential structure we have in mind are: no structure at all, given by the identity map $BO(2) \to BO(2)$; orientations, given by the double cover $BSO(2) \to BO(2)$; Spin structures, given by the bundle $B\mathrm{Spin}(2) \to BO(2)$; and any of these together with a map to a background space $X$, e.g.\ the map $BSO(2) \times X \to BO(2)$ given by projection to the first factor and then the double cover. These examples have been studied in the literature, but in a companion paper \cite{RWFramedPinMCG} we use the main theorems of this paper to investigate tangential structures for which homological stability was not previously known, such as framings, $r$-Spin structures, and $\mathrm{Pin}^\pm$ structures.

\subsection{Stabilisation}

If $F$ is a collared surface, and $K$ is a cobordism from $\partial F$ to a 1-manifold $\partial F'$, which is collared at both ends, then there is a canonical smooth structure on $F' := F \cup_{\partial F} K$, making it into a collared surface. Given a $\theta$-structure $\ell_K$ on $K$, we obtain induced boundary conditions $\ell_{\partial F}$  and $\ell_{\partial F'}$, and a \emph{stabilisation map}
\begin{equation}\label{eq:StabilisationMap}
(K, \ell_K)_* : \mathcal{M}^\theta(F;\ell_{\partial F}) \lra \mathcal{M}^\theta(F';\ell_{\partial F'}).
\end{equation}

Qualitatively speaking, we say that the homology group $H_k(\mathcal{M}^\theta)$ \emph{stabilises for orientable surfaces} if each map (\ref{eq:StabilisationMap}) with $F$ and $F'$ both orientable induces an isomorphism on $k$th homology when the genus of $F$ is large enough. Similarly, we say $H_k(\mathcal{M}^\theta)$ \emph{stabilises for non-orientable surfaces} if each map (\ref{eq:StabilisationMap}) with $F$ and $F'$ both non-orientable induces an isomorphism on $k$th homology when the genus of $F$ is large enough (genus in this case is to be interpreted as the maximal number of $\bR\bP^2$ connect-summands). The following is the qualitative version of our main theorem.

\begin{thm}\label{thm:QualStabThm}
Fix a tangential structure $\theta : \X \to BO(2)$.
\begin{enumerate}[(i)]
\item\label{it:thm:QualStabThm:1} Suppose that the bundle $\theta^*\gamma_{2} \to \X$ is orientable, and the zeroth homology groups $H_0(\mathcal{M}^\theta)$ stabilise for orientable surfaces. Then the homology groups $H_i(\mathcal{M}^\theta)$ stabilise for every $i \geq 0$ for orientable surfaces.

\item Suppose that the zeroth homology groups $H_0(\mathcal{M}^\theta)$ stabilise for non-orientable surfaces. Then the homology groups $H_i(\mathcal{M}^\theta)$ stabilise for every $i \geq 0$ for non-orientable surfaces.
\end{enumerate}
\end{thm}

In practice one wants concrete estimates, saying that all stabilisation maps out of $\mathcal{M}^\theta(F;\ell_{\partial F})$ are isomorphisms in degrees $* \leq k$ if the genus of $F$ is at least $f(k)$ for some explictly given function $f$. In Theorems \ref{thm:StabOrientableSurfaces} and \ref{thm:StabNonorientableSurfaces}  we give a quantitative version of this theorem, which provides such a function $f$, and in fact gives more refined information concerning the stability range for certain basic stabilising cobordisms $K$. The statement of this quantitative theorem is complicated, and involves developing some theory first. We will not state it here, but in the following two sections we give some corollaries of this quantitative statement for the most basic tangential structures.

\begin{rem}
The assumption that $\theta^*\gamma_2$ be orientable when considering orientable surfaces is essential if the result is stated in this generality: see Remark \ref{rem:Orient2} for a discussion. It is also technically convenient for our proof: see Remark \ref{rem:Orient1}.
\end{rem}

\subsection{Quantitative results for orientable surfaces}\label{sec:QuantOrientable}

Recall that we use the notation $\Sigma_{g,b}$ for a fixed model connected orientable surface of genus $g$ and with $b$ boundary components (i.e.\ the surface obtained from $\#^g S^1 \times S^1$ by removing the interiors of $b$ disjoint closed discs). As shorthand we write $\mathcal{M}^\theta(\Sigma_{g,b})$ for any space $\mathcal{M}^\theta(F;\ell_{\partial F})$ where $F$ is diffeomorphic to $\Sigma_{g,b}$.

Using this notation, there are several basic forms of the stabilisation maps introduced in the last section, which we write as
\begin{eqnarray*}
	\alpha & : & \mathcal{M}^\theta(\Sigma_{g, b}) \lra \mathcal{M}^\theta(\Sigma_{g+1, b-1})\\
	\beta & : & \mathcal{M}^\theta(\Sigma_{g, b}) \lra \mathcal{M}^\theta(\Sigma_{g, b+1}) \\
	\gamma & : & \mathcal{M}^\theta(\Sigma_{g, b}) \lra \mathcal{M}^\theta(\Sigma_{g, b-1})
\end{eqnarray*}
The notation $\alpha$ denotes any stabilisation map $(K, \ell_K)_*$ where $K$ is a cobordism given by a pair of pants with legs on the incoming boundary (and perhaps some disjoint trivial cobordisms). Similarly, the notation $\beta$ denotes any stabilisation map given by a pair of pants with legs on the outgoing boundary (and perhaps some disjoint trivial cobordisms). Finally, $\gamma$ denotes any stabilisation map given by a disc with boundary on the incoming boundary (and perhaps some disjoint trivial cobordisms). We write these as $\alpha(g)$, $\beta(g)$ and $\gamma(g)$ when we want to record the genus of the smaller surface.

We emphasise that the notation $\alpha$, $\beta$ or $\gamma$ does not specify the stabilisation map $(K, \ell_K)_*$: in each case there is a combinatorial choice of which boundary components the pairs of pants or disc is attached to, as well as a choice of which $\theta$-structure the cobordism is given.

\vspace{2ex}

Consider the tangential structure $\theta : BSO(2) \to BO(2)$. Choose a $\theta$-structure $\ell_{g,b}$ on $\Sigma_{g,b}$, and let $\ell_{\partial \Sigma_{g,b}}$ be the $\theta$-structure induced on the boundary. By elementary obstruction theory the space $\Bun_\partial(T\Sigma_{g,b}, \theta^*\gamma_2 ; \ell_{\partial \Sigma_{g,b}})$ is then contractible if $b > 0$, or has two contractible components if $b=0$, given by the two possible orientations of $\Sigma_{g,0}$. Thus in either case we obtain a homotopy equivalence
$$\mathcal{M}^\theta(\Sigma_{g, b}, \ell_{\partial \Sigma_{g,b}}) \simeq B\Diff_\partial^+(\Sigma_{g,b})$$
to the classifying space of the group of orientation-preserving diffeomorphisms of $\Sigma_{g,b}$. By a theorem of Earle--Eells \cite[p.\ 24]{EE}, the quotient map
$$\Diff_\partial^+(\Sigma_{g,b}) \lra \pi_0(\Diff_\partial^+(\Sigma_{g,b})) =: \Gamma_{g,b}$$
to the mapping class group is a weak homotopy equivalence (for $b>0$ or $g \geq 2$), so in total we have a homotopy equivalence $\mathcal{M}^+(\Sigma_{g, b}, \ell_{\partial \Sigma_{g,b}}) \simeq B\Gamma_{g,b}$. (In turn, $B\Gamma_{g}$ is rationally equivalent to Riemann's moduli space $\gM_g$.) Under these equivalences, the stabilisation maps correspond to the homomorphisms between mapping class groups induced by inclusions of subsurfaces.

The results of Theorem \ref{thm:StabOrientableSurfaces} in this case show that
\begin{enumerate}[(i)]
	\item Any $\alpha(g)_* : H_*(\Gamma_{g, b}) \to H_*(\Gamma_{g+1, b-1})$ is an epimorphism for $3* \leq 2g+1$ and an isomorphism for $3* \leq 2g-2$.
	\item Any $\beta(g)_* : H_*(\Gamma_{g, b}) \to H_*(\Gamma_{g, b+1})$ is an isomorphism for $3* \leq 2g$ and a monomorphism in all degrees.
	\item Any $\gamma(g)_* : H_*(\Gamma_{g, b+1}) \to H_*(\Gamma_{g, b})$ is an isomorphism for $3* \leq 2g$. For $b > 0$ it is an epimorphism in all degrees; for $b=0$ it is an epimorphism for $3* \leq 2g + 3$.
\end{enumerate}
This coincides with the stability range recently obtained by Boldsen \cite{Boldsen}, except that our range for closing the last boundary component is slightly better. 

\vspace{2ex}

Cohen and Madsen \cite{CM} introduced certain \textit{moduli spaces of surfaces with maps to a background space $X$}, denoted $\mathcal{S}_{g,b}(X)$, and studied their homology stability when $X$ is simply connected. In our notation these are simply the spaces $\mathcal{M}^\theta(\Sigma_{g, b})$ for the tangential structure $\theta : BSO(2) \times X \to BO(2)$, with boundary condition that $\partial \Sigma_{g, b}$ is mapped constantly to a basepoint $x_0 \in X$. In Section \ref{sec:DerivingBGSpaceSurfacesRange} we show that, when $X$ is simply connected, these moduli spaces exhibit homology stability in the same ranges of degrees as given above for $B\Gamma_{g,b}$. This slightly improves the stability range recently obtained by Boldsen \cite{Boldsen} in the case of surfaces with non-empty boundary, and also holds for closed surfaces (his methods are unable to prove stability for closing the last boundary).

\subsection{Quantitative results for non-orientable surfaces}\label{sec:QuantNonorientable}

As in the oriented case, we let $S_{g,b}$ denote a fixed model connected non-orientable surface of genus $g$ and with $b$ boundary components (i.e.\ the surface obtained from $\#^g \bR\bP^2$ by removing the interiors of $b$ disjoint closed discs). We again write $\mathcal{M}^\theta(S_{g,b})$ for any space $\mathcal{M}^\theta(F;\ell_{\partial F})$ where $F$ is diffeomorphic to $S_{g,b}$.

As well as analogues of the stabilisation maps $\alpha$, $\beta$ and $\gamma$, there is one further type of basic stabilisation map for non-orientable surfaces,
\begin{eqnarray*}
\mu & : & \mathcal{M}^\theta(S_{n, b}) \lra \mathcal{M}^\theta(S_{n+1, b}),
\end{eqnarray*}
which denotes any stabilisation map given by a projective plane with two open discs removed (and perhaps some disjoint trivial cobordisms).

Consider the trivial tangential structure $\theta : BO(2) \to BO(2)$. By the universal property of the bundle $\gamma_2 \to BO(2)$, the spaces $\Bun_\partial(TF, \gamma_2 ; \ell_{\partial F})$ are always contractible, and so there is a homotopy equivalence
$$\mathcal{M}^\theta(F;\ell_{\partial F}) \simeq B\Diff_\partial(F)$$
for any surface $F$. By a nonorientable version of the theorem of Earle--Eells, proved in full generality by Gramain \cite[Th{\'e}or{\`e}me 1]{Gramain}, the quotient map
$$\Diff_\partial(S_{g,b}) \lra \pi_0(\Diff_\partial(S_{g,b})) =: \mathcal{N}_{g,b}$$
to the mapping class group is a weak homotopy equivalence (for $b>0$ or $g \geq 3$), so we have an equivalence $\mathcal{M}^\theta(S_{g,b};\ell_{\partial S_{g,b}}) \simeq B\mathcal{N}_{g,b}$ for any choice of boundary condition $\ell_{\partial S_{g,b}}$.

The results of Theorem \ref{thm:StabNonorientableSurfaces} in this case show that
\begin{enumerate}[(i)]
	\item Any $\alpha(g)_* : H_*(\mathcal{N}_{g, b}) \to H_*(\mathcal{N}_{g+2, b-1})$ is an isomorphism for $3* \leq g-3$.
	\item Any $\beta(g)_* : H_*(\mathcal{N}_{g, b}) \to H_*(\mathcal{N}_{g, b+1})$ is an isomorphism for $3* \leq g-3$ and a monomorphism in all degrees.
	\item Any $\gamma(g)_* : H_*(\mathcal{N}_{g, b+1}) \to H_*(\mathcal{N}_{g, b})$ is an isomorphism for $3* \leq g-3$. For $b > 0$ it is an epimorphism is all degrees; for $b=0$ it is an epimorphism for $3* \leq g$.
	\item Any $\mu(g)_* : H_*(\mathcal{N}_{g, b}) \to H_*(\mathcal{N}_{g+1, b})$ is an epimorphism for $3* \leq g$ and an isomorphism for $3* \leq g-3$.
\end{enumerate}
This stability range for non-orientable surfaces improves on the previously best known range, due to Wahl \cite{Wahl} which was of slope $1/4$, whereas ours is of slope $1/3$.

\vspace{2ex}

As in the oriented case we can also consider the tangential structure $\theta : BO(2) \times X \to BO(2)$, so the spaces $\mathcal{M}^\theta(S_{n, b})$ are moduli spaces of non-orientable surfaces equipped with a map to a background space $X$. These have not yet appeared in the literature, but fit into our general framework. In Section \ref{sec:DerivingBGSpaceSurfacesRange} we show that, when $X$ is simply connected, these spaces have the same stability range as that given above for $B\mathcal{N}_{g,b}$.

\subsection{Stable homology}

By \cite{MW, GMTW, GR-W}, once one has a homological stability theorem for a tangential structure $\theta$ on surfaces, the stable homology coincides with that of the infinite loop space of the Madsen--Tillmann spectrum
$$\mathbf{MT\theta} := \mathbf{Th}(-\theta^*\gamma_2 \to \X),$$
given by the Thom spectrum of the negative of the bundle classified by the map $\theta$. Calculations of the rational cohomology of these infinite loop spaces are quite elementary, and can be found in \cite{MW} for oriented surfaces, \cite{Wahl} for non-orientable surfaces, and \cite{CM} for oriented surfaces with maps to a simply-connected background space. Calculations of $\bF_p$-homology are far more subtle, and can be found in \cite{galatius-2004} for oriented surfaces, \cite{R-W} for non-orientable surfaces, and \cite{galatius-2005} for Spin surfaces. 

\subsection{Further tangential structures}

In this paper we have focused on the applications of the homological stability theory we have developed to the most usually considered tangential structures, giving the moduli spaces of oriented and non-orientable surfaces, and the same with maps to a simply connected background space. However, the purpose of developing the theory in such generality is for its application to new tangential structures, so let us briefly mention these applications.

In a companion paper \cite{RWFramedPinMCG} we verify the hypotheses of Theorems \ref{thm:StabOrientableSurfaces} and \ref{thm:StabNonorientableSurfaces} for the tangential structures given by framings, Spin structures (and more generally $r$-Spin structures), and $\mathrm{Pin}^\pm$ structures. We then give computational applications of these stability results, for example: the framed mapping class group has trivial stable rational homology, and its stable abelianisation is $\bZ/24$; the $\mathrm{Pin}^+$ mapping class group has stable abelianisation $\bZ/2$, and the $\mathrm{Pin}^-$ mapping class group has stable abelianisation $(\bZ/2)^3$.

In \cite{RWPicardrSpin} we apply homology stability for the moduli spaces $\gM_g^{1/r}$ of $r$-Spin Riemann surfaces to compute their orbifold Picard groups, and to identify presentations for these groups in terms of canonically constructed line bundles. In \cite{ERW10}, Ebert and the author use homology stability for the spaces $\mathcal{S}_{g,b}(\bC\bP^\infty)$ (and in particular stability in integral homology for closing the last boundary, for which there is no other known proof) to study the cohomology of the universal Picard variety over $\gM_g$, and hence to study Kawazumi's extended mapping class group.

\subsection{Remarks on the proof}
Proofs of homological stability for families of groups now have a well established strategy: one finds highly-connected simplicial complexes on which the group acts, and such that the stabiliser subgroups of each simplex are ``smaller" groups in the family. Once such a complex has been found, there are two basic spectral sequence arguments one can use, an ``absolute" and a ``relative" one, to prove homological stability by induction.

In studying mapping class groups, the natural simplicial complex to use is the complex of isotopy classes of disjoint non-separating arcs, as was originally used by Harer \cite{H} and more recently by Boldsen \cite{Boldsen}. Unfortunately the mapping class group does not act transitively on the simplices of this complex; this complicates the spectral sequence, and one must carefully study the domain and range of differentials. We prefer to use a smaller complex, on which the mapping class group does act transitively, which considerably simplifies the spectral sequence argument. If we take arcs in an orientable surface with their ends on the same boundary component this is the complex $B_0(\Sigma)$ used by Ivanov \cite{Ivanov}, but for arcs with ends on different boundaries or for non-orientable surfaces it is a new complex. For the reader interested only in the mapping class groups of oriented surfaces, this approach has been explained by Wahl in her survey \cite{WahlSurvey}.

The second difference in our approach is that we deal exclusively with diffeomorphism groups, not mapping class groups, and so instead of the arc complex we consider a semi-simplicial space made from spaces of embeddings of arcs in a surface, on which the diffeomorphism group of the surface acts. We also find that working with diffeomorphism groups allows for arguments that one simply cannot make with mapping class groups: in Section \ref{sec:LastBoundary} we show how for any $d$-manifold $M$ we can express $B\Diff(M)$ as a homotopy colimit of the spaces $\{B\Diff_\partial(M \setminus \sqcup^n D^d)\}_{n \geq 1}$, and we use this in the case of surfaces to prove homology stability for closing the last boundary component.

\subsection{Outline}
In Section \ref{sec:SimplicialSpaces} we describe some standard notions concerning semi-simplicial spaces and the spectral sequence coming from their skeletal filtration. In Section \ref{sec:ModuliSurfaces} and \ref{sec:FlexModel} we give a careful definition of the moduli spaces $\mathcal{M}^\theta(F;\ell_{\partial F})$, and introduce several different models for these spaces, and for the stabilisation maps between them. In Section \ref{sec:Resolutions} we define semi-simplicial ``resolutions" of our moduli spaces, and establish their properties.

In Section \ref{sec:kTriviality} we introduce the notion of $k$-triviality of a tangential structure $\theta$, and show that it may be verified by knowing only the sets $\pi_0(\mathcal{M}^\theta(F;\ell_{\partial F}))$ and the stabilisation maps between them. For a tangential structure $\theta$, enjoying the property of $k$-triviality will be the key requirement in proving that moduli spaces of $\theta$-surfaces have stability. In Section \ref{sec:StabilityTheorems} we give the statements of our main results, Theorems  \ref{thm:StabOrientableSurfaces} and \ref{thm:StabNonorientableSurfaces}. We then show how to deduce the results of Sections \ref{sec:QuantOrientable} and \ref{sec:QuantNonorientable} from these theorems.

In Section \ref{sec:ZeroInHomology} we give the main technical application of the notion of $k$-triviality, which is showing that certain long compositions of relative stabilisation maps are zero in homology. In Sections \ref{sec:PfOrientable} and \ref{sec:PfNonorientable} we give the proofs of the theorems of Section \ref{sec:StabilityTheorems}, which is a spectral sequence argument hinging on the notion of $k$-triviality. In Section \ref{sec:LastBoundary} we discuss stability for closing the last boundary component, and we give a new argument that uses in an important way that we are working with diffeomorphism groups and not mapping class groups. In Section \ref{sec:StableHomology} we explain how, once homological stability has been established, the stable homology may be calculated.

In Appendix \ref{app:CxArcs} we deduce the connectivities of certain complexes of arcs in surfaces, starting from results of Harer \cite{H} and Wahl \cite{Wahl}, which are necessary to prove the results of Section \ref{sec:Resolutions}. This is included as an appendix as it may be of interest independent of the body of the article.

\subsection{Acknowledgements}

I am grateful to Nathalie Wahl for both her interest in this project and her suggestions, which have greatly improved it. I am also grateful to the anonymous referee for their many helpful and incisive comments, which have greatly helped to improve the readability, and indeed the veracity, of this paper. This paper has been revised several times since first appearing in 2009, and while the essential approach has remained constant, I have incorporated several expository and technical ideas which I have learnt since then. I have learnt these principally through collaborations with S{\o}ren Galatius and with Federico Cantero, as well as discussions with Nathalie Wahl, and I would like to thank them all.

During the preparation of this article the author has been supported by an EPSRC Studentship, ERC Advanced Grant No.\ 228082, and the Danish National Research Foundation via the Centre for Symmetry and Deformation, and the Herchel Smith Fund.

\section{Semi-simplicial spaces and resolutions}\label{sec:SimplicialSpaces}

Let $\Delta^{op}$ denote the opposite of the simplicial category $\Delta$ having objects the finite ordered sets $[n] = \{0,1,\ldots,n\}$ and morphisms the weakly monotone maps. A \textit{simplicial object} in a category $\mathcal{C}$ is a functor $X_\bullet : \Delta^{op} \to \mathcal{C}$. Let $\Delta_{inj} \subset \Delta$ be the subcategory having all objects but only the strictly monotone maps, called the \textit{semi-simplicial category}. Call a functor $X_\bullet : \Delta^{op}_{inj} \to \mathcal{C}$ a \textit{semi-simplicial object} in $\mathcal{C}$, and write $X_n := X_\bullet([n])$. A (semi-)simplicial map $f: X_\bullet \to Y_\bullet$ is a natural transformation of functors: in particular, it has components $f_n :X_n \to Y_n$.

The geometric realisation of a semi-simplicial space $X_\bullet$ is
$$\vert X_\bullet\vert  = \coprod_{n \geq 0} X_n \times \Delta^n / \sim$$
where the equivalence relation is $(d_i(x), y) \sim (x, d^i(y))$, for $d^i : \Delta^n \to \Delta^{n+1}$ the inclusion of the $i$th face. 

If $X_\bullet$ is a semi-simplicial pointed space, its realisation \textit{as a pointed space} is
$$\vert X_\bullet\vert_* = \bigvee_{n \geq 0} X_n \rtimes \Delta^n / \sim$$
where $d_i(x) \rtimes y \sim x \rtimes d^i(y)$. (Recall that the \textit{half smash product} of a space $Y$ and a pointed space $(C,c_0)$ is the pointed space $C \rtimes Y := C \times Y / \{c_0\} \times Y$.) If $X_\bullet^+$ denotes levelwise addition of a disjoint basepoint, then there is a homeomorphism $\vert X_\bullet^+\vert_* \cong \vert X_\bullet\vert_+$. More generally, if the $X_s$ are all well-pointed there is a homeomorphism $\vert X_\bullet\vert_* \cong \vert X_\bullet\vert  / \vert *_\bullet\vert$, where $*_\bullet$ is the semi-simplicial space with a single point in each degree, and in particular $\vert X_\bullet\vert_* \simeq \vert X_\bullet\vert$.

\vspace{2ex}

The skeletal filtration of $\vert X_\bullet\vert$ gives a strongly convergent first quadrant spectral sequence
\begin{equation}\tag{sSS}\label{SSRestrictedSimplicialSpace}
    E^1_{s,t} = h_t(X_s) \Longrightarrow h_{s+t}(\vert X_\bullet\vert)
\end{equation}
for any connective generalised homology theory $h_*$. The $d^1$ differential is given by the alternating sum of the face maps, $d^1 = \sum (-1)^i (d_i)_*$. There is also a pointed analogue, which takes the form
\begin{equation}\tag{PsSS}\label{SSPointedSimplicialSpace}
    E^1_{s,t} = h_t(X_s, *) \Longrightarrow h_{s+t}(\vert X_\bullet\vert_*, *)
\end{equation}
as long as each $X_s$ is well-pointed.

\subsection{Relative semi-simplicial spaces}

Let $f_\bullet: X_\bullet \to Y_\bullet$ be a map of semi-simplicial spaces. Then the levelwise homotopy cofibres form a semi-simplicial pointed space $C_{f_\bullet}$ which is well-pointed in each degree, and
$$\vert C_{f_\bullet}\vert_* \cong C_{\vert f_\bullet\vert}$$
as homotopy colimits commute. In particular, the spectral sequence (\ref{SSPointedSimplicialSpace}) for this semi-simplicial pointed space takes the form
\begin{equation}\tag{RsSS}\label{SSRelativeRestrictedSimplicialSpace}
    E^1_{s,t} = {h}_{t}(C_{f_s}, *) \cong h_t(Y_s, X_s) \Longrightarrow {h}_{s+t}(C_{\vert f_\bullet\vert}, *) \cong h_{s+t}(\vert Y_\bullet\vert, \vert X_\bullet\vert).
\end{equation}

\subsection{Augmented semi-simplicial spaces}\label{sec:AugSSSpaces}

An \textit{augmentation} of a (semi-)simplicial space $X_\bullet$ is a space $X_{-1}$ and a map $\epsilon : X_0 \to X_{-1}$ such that $\epsilon d_0 = \epsilon d_1 : X_1 \to X_{-1}$. An augmentation induces a map $\vert \epsilon\vert  : \vert X_\bullet\vert \to X_{-1}$. In this case there is a spectral sequence defined for $s \geq -1$,
\begin{equation}\tag{AsSS}\label{SSAugmentedRestrictedSimplicialSpace}
E^1_{s,t} = h_t(X_s) \Longrightarrow h_{s+t+1}(C_{\vert \epsilon\vert}, *) \cong h_{s+t+1}(X_{-1}, \vert X_\bullet\vert).
\end{equation}
for any connective generalised homology theory $h_*$. The $d^1$ differentials are as above for $s>0$, and $d^1 : E^1_{0,t} \to E^1_{-1, t}$ is given by $\epsilon_*$.

There is also a relative version of this construction. Let $f: (\epsilon_X : X_\bullet \to X_{-1}) \to (\epsilon_Y : Y_\bullet \to Y_{-1})$ be a map of augmented semi-simplicial spaces. There is a spectral sequence defined for $s \geq -1$,
\begin{equation}\tag{RAsSS}\label{SSRelativeAugmentedRestrictedSimplicialSpace}
E^1_{s,t} = h_t(X_s, Y_s) \Longrightarrow h_{s+t+1}(C_{\vert \epsilon_X\vert }, C_{\vert \epsilon_Y\vert}).
\end{equation}

\subsection{Resolutions}

For our purposes, a \textit{resolution} of a space $X$ is an augmented semi-simplicial space $X_\bullet \to X$ such that the map $\vert X_\bullet\vert \to X$ is a weak homotopy equivalence. An \textit{$n$-resolution} of a space $X$ is an augmented semi-simplicial space $X_\bullet \to X$ such that the map $\vert X_\bullet\vert \to X$ is $n$-connected.

\subsection{The fibre of the augmentation map}\label{sec:AugFib}

Let $G$ be a topological group, $B$ be a $G$-space, and $K_\bullet$ be a simplicial $G$-space. Then $(K_\bullet \times B) \hcoker G$ is a simplicial space augmented over $B\hcoker G$, and we will occasionally need to understand the homotopy fibre of the augmentation map $\vert (K_\bullet \times B) \hcoker G\vert \to B \hcoker G$.

\begin{lem}\label{lem:AugTriv}
Suppose we work in the category of compactly generated weak Hausdorff spaces. Then $\vert (K_\bullet \times B) \hcoker G\vert \to B \hcoker G$ is a locally trivial fibre bundle with fibre $\vert K_\bullet\vert$.
\end{lem}

Because of the assumptions of this lemma, we shall always work in the category of compactly generated weak Hausdorff spaces, without further mention.

\begin{proof}
The natural maps
$$\vert K_\bullet \times B\vert \times_G EG \longleftarrow \vert (K_\bullet \times B) \times EG\vert / G \lra \vert (K_\bullet \times B) \times_G EG\vert$$
are both homeomorphisms. The first as $\vert (K_\bullet \times B) \times EG\vert \to \vert K_\bullet \times B\vert \times EG$ is, as we are working in CGWH spaces, and the second as an inverse may be constructed by induction on skeleta. Finally, the natural map $\vert K_\bullet \times B\vert \to \vert K_\bullet \vert \times B$ is also a homeomorphism, as we are working in CGWH spaces.
\end{proof}

\section{Moduli spaces of surfaces}\label{sec:ModuliSurfaces}

In this section we shall give an alternative definition of moduli spaces of surfaces with $\theta$-structure, which for many purposes is more convenient than that of Definition \ref{defn:BorelConstModel} (though we will show in Section \ref{sec:BorelConst} that these models are homotopy equivalent). We will also describe the analogue of the stabilisation maps \eqref{eq:StabilisationMap} in this model.

We first recall some definitions regarding $\theta$-structures, which we already gave briefly in the introduction.

\begin{defn}
A \emph{tangential structure} is a map $\theta : \X \to BO(2)$ from a path-connected space. A $\theta$-structure on a manifold $M$ of dimension $d \leq 2$ is a bundle map $\ell : \epsilon^{2-d} \oplus TM \to \theta^*\gamma_2$, i.e.\ a map which is a fibrewise linear isomorphism. We call such a pair $(M, \ell)$ a \emph{$\theta$-manifold of dimension $d$}.

For a surface $F$, let $\Bun^\theta(F)$ denote the space of $\theta$-structures on $F$, equipped with the compact-open topology. If $\ell_0$ is a $\theta$-structure on the 1-manifold $\partial F$, and $F$ is equipped with a collar $c : (-1,0] \times \partial F \hookrightarrow F$, then we let $\Bun^\theta_\partial(F;\ell_0) \subset \Bun^\theta(F)$ denote the subspace of those bundle maps $\ell$ such that the composition
$$\epsilon^1 \oplus T(\partial F) \lra T((-1,0] \times \partial F) \overset{Dc}\lra TF \overset{\ell}\lra \theta^*\gamma_2$$
is $\ell_0$, where the first map is the canonical isomorphism onto $T((-1,0] \times \partial F)\vert_{\{0\} \times \partial F}$
\end{defn}

Let $\Psi_\theta(\bR^N)$ denote the set of pairs $(X, \ell_X)$ where $X \subset \bR^N$ is a topologically closed subset which is a smooth submanifold of dimension 2, and $\ell_X : TX \to \theta^*\gamma_2$ is a $\theta$-structure on $X$. In \cite[\S 2]{GR-W} Galatius and the author have described a topology on this set which is ``compact-open" in flavour, and it will be convenient to use the colimit topology on
$$\Psi_\theta = \Psi_\theta(\bR \times \bR^\infty) := \underset{n \to \infty} \colim \Psi_\theta(\bR \times \bR^n)$$
to describe various models for moduli spaces of surfaces with $\theta$-structure that we will require.

\begin{conv}\mbox{}
\begin{enumerate}[(i)]
\item For clarity we will omit the $\theta$-structure from the notation when referring to $\theta$-manifolds, and write $\ell_M$ for the $\theta$-structure on a  $\theta$-manifold $M$ if we explicitly need to refer to it.

\item In $\bR \times \bR^\infty$ we shall write $e_i$ for the $i$th basis vector, starting with $e_0$ for the vector $(1,0,0,\ldots)$.
\end{enumerate}
\end{conv}

\begin{defn}
Let $P \subset \bR^\infty$ be a compact closed smooth $\theta$-manifold of dimension 1. Let $\mathcal{N}^\theta(P)$ denote the set of pairs of a compact surface $X \subset (-\infty,0] \times \bR^\infty$ and a $\theta$-structure $\ell_X : TX \to \theta^*\gamma_2$, such that $(X, \ell_X)$ agrees with $(-\infty,0] \times P$ as a $\theta$-manifold near $\{0\} \times \bR^\infty$. To such a $\theta$-manifold $X$ we can associate the \emph{extended} $\theta$-manifold
$$X^e := X \cup ([0,\infty) \times P) \subset \bR \times \bR^\infty,$$
which is an element of $\Psi_\theta$, and this construction defines an injective function $X \mapsto X^e : \mathcal{N}^\theta(P) \to \Psi_\theta$; we give $\mathcal{N}^\theta(P)$ the subspace topology.
\end{defn}

\subsection{Connectedness, orientation type, and genus}

Let us say that a connected orientable surface $F$ has \emph{genus $g$} if it is diffeomorphic to the surface obtained by removing a collection of disjoint open discs from $\#^g S^1 \times S^1$. Similarly, let us say that a connected non-orientable surface $F$ has \emph{genus $g$} if it is diffeomorphic to the surface obtained by removing a collection of disjoint open discs from $\#^g \bR\bP^2$.
 
\begin{defn}\label{defn:ConnOrientGenus}\mbox{}
\begin{enumerate}[(i)]
\item Let $\mathcal{M}^\theta(P) \subset \mathcal{N}^\theta(P)$ denote the subspace consisting of those surfaces $X$ which are path connected.

\item For $g \geq 1$, let $\mathcal{M}^\theta(g,-;P) \subset \mathcal{M}^\theta(P)$ denote the subspace of those surfaces $X$ which are non-orientable and have genus $g$.

\item\label{it:defn:ConnOrientGenus:3} For $g \geq 0$, if $\theta^*\gamma_2$ is orientable then let $\mathcal{M}^\theta(g,+;{P}) \subset \mathcal{M}^\theta(P)$ denote the subspace of those surfaces $X$ which are orientable and have genus $g$.
\end{enumerate}
\end{defn}

Note that all of these spaces are unions of path components of $\mathcal{N}^\theta(P)$. The reader will see that the notation $\mathcal{M}^\theta(-)$ is used in two ways. In Definition \ref{defn:ConnOrientGenus}, for $P \subset \bR^\infty$ a compact closed 1-dimensional $\theta$-manifold, $\mathcal{M}^\theta(P)$ denotes the space of connected surfaces with $\theta$-structure having boundary equal to $P$. In Definition \ref{defn:BorelConstModel}, for $F$ a abstract collared surface and $\ell_{\partial F}$ a $\theta$-structure on $\partial F$, $\mathcal{M}^\theta(F;\ell_{\partial F})$ denotes the space of surfaces diffeomorphic to $F$ equipped with $\theta$-structure extending $\ell_{\partial F}$. This last description is not immediate from Definition \ref{defn:BorelConstModel}, and in Section \ref{sec:BorelConst} we explain the connection. It will always be clear which of these two notions the notation $\mathcal{M}^\theta(-)$ represents.

\subsection{Relation to the Borel construction model}\label{sec:BorelConst}

Let $F$ be a surface and $c : (-1,0] \times \partial F \hookrightarrow F$ be a collar, and let $\ell_{\partial F}$ be a boundary condition. Choose an embedding $e_0 : \partial F \hookrightarrow \bR^\infty$, and let
$$\Emb_\partial(F, (-\infty,0] \times \bR^\infty ; e_0)$$
denote the set of embeddings $e : F \hookrightarrow (-\infty,0] \times \bR^\infty$ such that there exists an $\epsilon>0$ such that $(e \circ c)(t, x) = (t, e_0(x))$ for all $\vert t \vert < \epsilon$. We give this space the $C^\infty$-topology. Similarly, we let $\Diff_\partial(F)$ denote the set of those diffeomorphisms of $F$ which are the identity on $c((-\epsilon,0] \times \partial F)$ for some $\epsilon>0$, again with the $C^\infty$-topology.

The action of $\Diff_\partial(F)$ on $\Emb_\partial(F, (-\infty,0] \times \bR^\infty ; e_0)$ by precomposition exhibits $\Emb_\partial(F, (-\infty,0] \times \bR^\infty ; e_0)$ as a principal $\Diff_\partial(F)$-space, by the main result of Binz--Fischer \cite{BinzFischer}. Furthermore, it is well-known that such spaces of embeddings into infinite-dimensional Euclidean space are weakly contractible, and so this space of embeddings is a model for the universal principal $\Diff_\partial(F)$-space. Hence, one model for the Borel construction in Definition \ref{defn:BorelConstModel} is
$$\mathcal{M}^\theta(F;\ell_{\partial F}) := (\Emb_\partial(F, (-\infty,0] \times \bR^\infty ; e_0) \times \Bun_\partial(TF, \theta^*\gamma_2 ; \ell_{\partial F}))/\Diff_\partial(F).$$
If we consider $\partial F$ as being a submanifold of $\bR^\infty$ via the embedding $e_0$, then there is a continuous map
\begin{align*}
\Emb_\partial(F, (-\infty,0] \times \bR^\infty ; e_0) \times \Bun_\partial(TF, \theta^*\gamma_2 ; \ell_{\partial F}) &\lra \mathcal{N}^\theta(\partial F, \ell_{\partial F})\\
(e, \ell) &\longmapsto (e(F), \ell \circ (De)^{-1}),
\end{align*}
which is constant on $\Diff_\partial(F)$-orbits, and so induces a continuous map
$$\mathcal{M}^\theta(F;\ell_{\partial F}) \lra \mathcal{N}^\theta(\partial F, \ell_{\partial F}).$$
By the definition in \cite[\S 2]{GR-W} of the topology on $\Psi_\theta$, it follows that this map is a homeomorphism onto a collection of path components.

Thus, an alternative description of $\mathcal{N}^\theta(P)$ is
$$\mathcal{N}^\theta(P) \cong \coprod_{[F]} \mathcal{M}^\theta(F;\ell_P)$$
where the disjoint union is taken over all surfaces with boundary identified with $P$, one in each relative diffeomorphism class. The subspace $\mathcal{M}^\theta(P) \subset \mathcal{N}^\theta(P)$ is given by a similar formula, where the disjoint union is taken over all connected surfaces with boundary identified with $P$, one in each relative diffeomorphism class.

\begin{rem}\label{rem:Orient1}
If $X$ and $Y$ are connected non-orientable surfaces with boundary, and $\phi : \partial X \to \partial Y$ is a diffeomorphism between their boundaries, then $\phi$ extends to a diffeomorphism $\hat{\phi} : X \to Y$ if and only if $X$ and $Y$ are diffeomorphic. The same is not true for orientable surfaces: if we choose an orientation of $X$, restrict it to $\partial X$, and hence obtain an orientation of $\partial Y$ using $\phi$, this orientation need not extend to $Y$. The diffeomorphism $\phi$ extends, however, if and only if this orientation extends and $X$ and $Y$ are diffeomorphic.

This accounts for the requirement in Definition \ref{defn:ConnOrientGenus} (\ref{it:defn:ConnOrientGenus:3}) that $\theta^*\gamma_2$ be orientable. Under this hypothesis we may as well choose an orientation of $\theta^*\gamma_2$, which then induces a canonical orientation of any $\theta$-manifold. In particular the boundary condition $\ell_P$ induces an orientation of $P$, and any $X \in \mathcal{N}^\theta(P)$ has an orientation compatible with that of $P$. Thus any two $X, Y \in \mathcal{N}^\theta(P)$ which are connected and have the same genus are diffeomorphic relative to $P$.
\end{rem}

In particular, the spaces in Definition \ref{defn:ConnOrientGenus} may also be described as
$$\mathcal{M}^\theta(g,-;P) \cong \mathcal{M}^\theta(S_{g,b} ; \ell_{P})$$
when the 1-dimensional $\theta$-manifold $P$ consists of $b$ circles, and an identification $P \cong \partial S_{g,b}$ is chosen, and in the oriented case
$$\mathcal{M}^\theta(g,+;P) \cong \mathcal{M}^\theta(\Sigma_{g,b} ; \ell_{P})$$
when the 1-dimensional $\theta$-manifold $P$ consists of $b$ circles, and an identification $P \cong \partial \Sigma_{g,b}$ is chosen.

From our point of view, the models $\mathcal{M}^\theta(g,\pm;P)$ of Definition \ref{defn:ConnOrientGenus} are more convenient than the Borel construction model, because they do not rely on a choice of model surface of type $S_{g,b}$ and $\Sigma_{g,b}$. In particular, it is easier to construct stabilisation maps between them, which we shall now do.

\subsection{Stabilisation maps and the cobordism category}\label{sec:StabMapsCobCat}

If $P$ and $P' \subset \bR^\infty$ are two compact closed 1-dimensional $\theta$-manifolds, and $K \subset [0,k] \times \bR^\infty$ is a 2-dimensional $\theta$-manifold which agrees with $[0,k] \times P$ near $\{0\} \times \bR^\infty$ and with $[0,k] \times P'$ near $\{k\} \times \bR^\infty$, then there is a continuous map
\begin{equation}\label{eq:StabMap}
\begin{aligned}
K_* : \mathcal{N}^\theta(P) & \lra \mathcal{N}^\theta(P')\\
X & \longmapsto (X \cup K) - k \cdot e_0.
\end{aligned}
\end{equation}

This may be used to construct a functor defined on the cobordism category of Galatius--Madsen--Tillmann--Weiss \cite{GMTW}. Let us recall its definition (which we have slightly modified to suit our needs).

\begin{defn}
Let $\mathcal{C}_\theta$ be the category enriched in topological spaces with
\begin{enumerate}[(i)]
\item objects given by 1-dimensional closed $\theta$-submanifolds $P \subset \bR^\infty$,

\item non-identity morphisms from $P$ to $P'$ given by pairs $(t,W)$ where $t \in (0,\infty)$ and $W \subset [0,t] \times \bR^\infty$ is a $\theta$-surface which agrees with $[0,t] \times P$ near $\{0\} \times \bR^\infty$ and with $[0,t] \times P'$ near $\{t\} \times \bR^\infty$.
\end{enumerate}
There is an injective function
\begin{align*}
\mathcal{C}_\theta(P,P') &\lra \bR \times \Psi_\theta\\
(t,W) & \longmapsto (t, ((-\infty,0] \times P) \cup W \cup ([t,\infty) \times P'))
\end{align*}
and we give $\mathcal{C}_\theta(P,P')$ the subspace topology. Composition in this category is given by the formula
$$(t', W') \circ (t, W) := (t+t', W \cup (W' + t \cdot e_0)).$$
\end{defn}

The construction $P \mapsto \mathcal{N}^\theta(P)$ defines a functor $\mathcal{N}^\theta : \mathcal{C}_\theta \to \mathbf{Top}$, where a cobordism $(t,W) : P \leadsto P'$ induces the map
\begin{align*}
W_* : \mathcal{N}^\theta(P) &\lra \mathcal{N}^\theta(P')\\
X &\longmapsto (X \cup W) - t \cdot e_0 .
\end{align*}

\begin{rem}\label{rem:QuasiIso}
The category $\mathcal{C}_\theta$ has no isomorphisms except for identity maps. However, the cobordisms $[0,t] \times P : P \leadsto P$ induce endomorphisms of $\mathcal{N}^\theta(P)$ which are homotopic to the identity: they merely add an external collar to a nullbordism, which can be shrunk down. As such, it is natural to think of these morphisms as being honorary identity maps, and so it is natural to think of their path-component in $\mathcal{C}_\theta(P, P)$ as consisting of isomorphisms, because they induce homotopy equivalences on $\mathcal{N}^\theta(-)$. 

Thus, given morphisms $(t,W) :P \leadsto P'$ and $(t',W') : P' \leadsto P$ such that $W' \circ W$ and $W \circ W'$ are both isotopic to cylindrical cobordisms, we shall say that $W$ is a \emph{quasi-isomorphism} in $\mathcal{C}_\theta$, with \emph{quasi-inverse} $W'$.
\end{rem}

If $W$ is path-connected relative to $P$, then the gluing construction sends connected surfaces to connected surfaces, and so induces a map
$$W_* : \mathcal{M}^\theta(P) \lra \mathcal{M}^\theta(P').$$
Restricting further to the subspace of connected non-orientable surfaces of genus $g$, gluing on the cobordism $W$ has a definite effect on the genus of the surface, and this effect depends on the genus of the components of $W$ and the combinatorics of which path components of $P$ lie in which path components of $W$. Thus it induces a map
$$W_* : \mathcal{M}^\theta(g,-;P) \lra \mathcal{M}^\theta(g',-;P')$$
for some $g'$ wich can be computed from $g$ and $W$. Similarly, in the orientable case (where, recall, we suppose that $\theta^*\gamma_2$ is orientable) there is an induced map
$$W_* : \mathcal{M}^\theta(g,+;{P}) \lra \mathcal{M}^\theta(g',+;{P}')$$
for some $g'$ which can be computed from $g$ and $W$. There are four basic cases which we shall consider.

\begin{enumerate}[(i)]
\item Suppose that $W$ is an orientable cobordism which has a single 1-handle relative to $P$, which is attached to two distinct path components of $P$. Then the map induced by ${W}$ has the effect of gluing on a pair of pants along the legs, which increases the genus by 1 in the orientable case, and by 2 in the non-orientable case. We call such stabilisation maps \emph{maps of type $\alpha$}.

\item Suppose that $W$ is an orientable cobordism which has a single 1-handle relative to $P$, both ends of which are attached to a single path component of $P$. Then the map induced by ${W}$ has the effect of gluing on a pair of pants along the waist, which does not increase the genus. We call such stabilisation maps \emph{maps of type $\beta$}.

\item Suppose that $W$ is a cobordism which has a single 2-handle relative to $P$, which is necessarily attached along an entire path component of $P$. Then the map induced by ${W}$ has the effect of gluing on a disc, which does not increase the genus. We call such stabilisation maps \emph{maps of type $\gamma$}.

\item Suppose that $W$ is a non-orientable cobordism which has a single 1-handle relative to $P$ (both ends of which are then necessarily attached to a single path component of $P$, along coherently oriented intervals). Then the map induced by ${W}$ has the effect of gluing on a projective plane with two discs removed, which increases the (non-orientable) genus by 1. We call such stabilisation maps \emph{maps of type $\mu$}.
\end{enumerate}

\subsection{Path components}\label{sec:PathComponents}

If we work with the Borel construction model, it is immediate that the set of path components of $\mathcal{M}^\theta(F;\ell_{\partial F})$ is identified with the orbit set
$$\pi_0(\mathcal{M}^\theta(F;\ell_{\partial F})) = \pi_0(\Bun_\partial(TF, \theta^*\gamma_2;\ell_{\partial F})) / \Gamma(F),$$
where $\Gamma(F) := \pi_0(\Diff_\partial(F))$.

\section{A flexible model for moduli spaces of surfaces}\label{sec:FlexModel}

In this section we shall introduce a generalisation of the spaces $\mathcal{N}^\theta(P)$, which are designed so that we have stabilisation maps analogous to \eqref{eq:StabMap} but for cobordisms starting from a codimension 0 subset $Q \subset P$. To do this carefully, we must control the $\theta$-structure near the boundary of $Q$, and to do this we make the following definition.

\begin{defn}[Standard $\theta$-structure]
Choose once and for all a map $\ell_{std} : \bR^2 \to \theta^*\gamma_2$ which is a linear isomorphism to a single fibre. If $P$ is a 1-manifold with a nowhere vanishing vector field $\xi : \epsilon^1 \overset{\sim}\to TP$ then the \emph{standard $\theta$-structure on $P$} is
$$\epsilon^1 \oplus TP \overset{\mathrm{Id} \oplus \xi^{-1}}\lra \epsilon^1 \oplus \epsilon^1 \overset{proj}\lra \bR^2 \overset{\ell_{std}}\lra \theta^*\gamma_2.$$
\end{defn}

In the following, the manifold $\bR \times \{\pm \tfrac{1}{2}\} \times \{0\}^{\infty-2} \subset \bR^\infty$ will play a distinguished role, and we shall always take it to have the $\theta$-structure induced by the vector field given by the first coordinate direction on $\bR \times \{\tfrac{1}{2}\}$ and by minus the first coordinate direction on $\bR \times \{-\tfrac{1}{2}\}$.

Let us write $\bI := [-1,1] \subset \bR$, and
$$\bA := \bR^\infty \setminus \mathrm{int}(\bI^\infty)$$
for the ``annular" region obtained by removing the cube of radius 1 from $\bR^\infty$.  

\begin{defn}
An \emph{inner boundary condition} $Q \subset \bI^\infty$ is a 1-dimensional compact $\theta$-manifold which agrees with $\bI^\infty \cap(\bR \times \{\pm \tfrac{1}{2}\} \times \{0\}^{\infty-2})$ as a $\theta$-manifold near $\partial \bI^\infty$, and whose boundary is precisely the four points $\{(\pm 1, \pm \tfrac{1}{2})\}$.

An \emph{outer boundary condition} $L \subset \bA$ is a 1-dimensional compact $\theta$-manifold which agrees with $\bA \cap(\bR \times \{\pm \tfrac{1}{2}\} \times \{0\}^{\infty-2})$ as a $\theta$-manifold near $\partial \bA$, and whose boundary is precisely the four points $\{(\pm 1, \pm \tfrac{1}{2})\}$.
\end{defn}

\begin{defn}
Given an inner boundary condition $Q$, an outer boundary condition $L$, and a $t \in \bR$, let $\mathcal{N}^\theta_L(t,Q)$ be the set of those $\theta$-surfaces inside
$$\mathcal{U}_t := ((-\infty,0] \times \bA) \cup ((-\infty,t] \times \bI^\infty)$$
which are compact, and agree with $((-\infty,0] \times L) \cup ((-\infty,t] \times Q) \subset \mathcal{U}_t$ near $\partial \mathcal{U}_t$. There is an injective function
\begin{align*}
\mathcal{N}^\theta_L(t,Q) & \lra \Psi_\theta\\
X & \longmapsto X \cup ([0,\infty) \times L) \cup ([t,\infty) \times Q) 
\end{align*}
and we topologise $\mathcal{N}^\theta_L(t,Q)$ as a subspace.
\end{defn}

\begin{figure}[h]
\centering
\includegraphics[bb = 0 0 343 102]{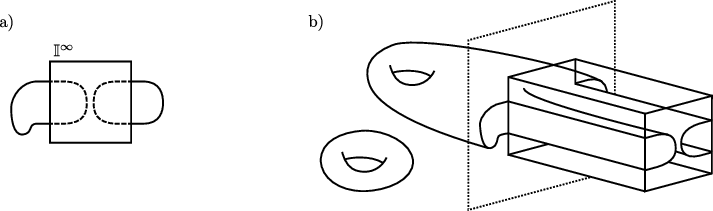}
\caption{a) an example of an inner boundary condition $Q$ (dashed) and an outer boundary condition $L$ (solid), in $\bR^2 \subset \bR^\infty$; b) an example of an element of $\mathcal{N}^\theta_L(t,Q)$ with $t > 0$, in $\bR^3 \subset \bR \times \bR^\infty$.}
\label{fig:FlexModel}
\end{figure}

The spaces $\mathcal{N}^\theta_L(t,Q)$ are functorial in two ways. Firstly, if $W \subset [t,t'] \times \bI^\infty$ is a $\theta$-surface which agrees with $[t,t'] \times (\bR \times \{\pm \tfrac{1}{2}\} \times \{0\}^{\infty-2})$ near $[t,t'] \times \partial \bI^\infty$, with $[t,t'] \times Q$ near $\{t\} \times \bI^\infty$, and with $[t,t'] \times Q'$ near $\{t'\} \times \bI^\infty$, then there is an induced map
\begin{align*}
W_* : \mathcal{N}^\theta_L(t,Q) &\lra \mathcal{N}^\theta_L(t',Q')\\
X & \longmapsto X \cup W.
\end{align*}
We call $W$ an \emph{inner cobordism}, or \emph{cobordism of inner boundary conditions}, from $Q$ to $Q'$.

Secondly, if $R \subset [-s,0] \times \bA$ is a $\theta$-surface which agrees with $[-s,0] \times (\bR \times \{\pm \tfrac{1}{2}\} \times \{0\}^{\infty-2})$ near $[-s,0] \times \partial \bA$, with $[-s,0] \times L'$ near $\{-s\} \times \bA$, and with $[-s,0] \times L$ near $\{0\} \times \bA$ then there is an induced map
\begin{align*}
R^* : \mathcal{N}^\theta_{L'}(t+s,Q) &\lra \mathcal{N}^\theta_L(t,Q)\\
X & \longmapsto (X-s\cdot e_0) \cup R.
\end{align*}
We call $R$ an \emph{outer cobordism}, or \emph{cobordism of outer boundary conditions}, from $L'$ to $L$.

When $L$ and $Q$ are a pair of an outer and an inner boundary condition, we write $L_Q \subset \bR^\infty$ for the 1-dimensional closed $\theta$-manifold $L \cup Q$. There is then an identification $\mathcal{N}^\theta(L_Q) = \mathcal{N}^\theta_L(0,Q)$, as well as inclusions
\begin{equation*}
\begin{aligned}
\iota : \mathcal{N}_L^\theta(0,Q) &\lra \mathcal{N}_L^\theta(t,Q) \\
X &\longmapsto X \cup ([0,t] \times Q)
\end{aligned}
\quad \quad \quad
\begin{aligned}
\iota : \mathcal{N}_L^\theta(-t,Q) &\lra \mathcal{N}_L^\theta(0,Q)\\
X & \longmapsto X \cup ([-t,0] \times Q)
\end{aligned}
\end{equation*}
for each $t \geq 0$.

\begin{lem}\label{lem:ShrinkingHty}
These inclusions are homotopy equivalences.
\end{lem}
\begin{proof}
Let us consider the first case, and define a map in the reverse direction by the formula
$$r:X \mapsto X \cup ([0,t] \times L) - t\cdot e_0: \mathcal{N}_L^\theta(t,Q) \lra \mathcal{N}_L^\theta(0,Q).$$
The composition $r \circ \iota$ is then simply given by 
$$X \longmapsto X \cup ([0,t] \times (L \cup Q)) - t \cdot e_0,$$
which is homotopic to the identity via $(s, X) \mapsto X \cup ([0,s\cdot t] \times (L \cup Q)) - s\cdot t \cdot e_0$ for $s \in [0,1]$. The composition $\iota \circ r$ can be treated similarly.
\end{proof}

An essential feature of the two functorialities described above is that they commute with each other: for an inner cobordism $W$ and an outer cobordism $R$ the square
\begin{equation*}
\xymatrix{
\mathcal{N}^\theta_{L'}(t+s,Q) \ar[d]^-{R^*} \ar[rr]^-{(W+s\cdot e_0)_*}& & \mathcal{N}^\theta_{L'}(t'+s,Q') \ar[d]^-{R^*}\\
\mathcal{N}^\theta_L(t,Q) \ar[rr]^-{W_*}&& \mathcal{N}^\theta_L(t',Q')
}
\end{equation*}
strictly commutes (not just up to homotopy). This is the principal reason for introducing this more flexible model, as this property will be essential later. Apart from this, these two forms of stabilisation should not be considered as very different. For example, it is easy to see that we may always replace gluing on an inner cobordism by gluing on an outer cobordism, up to applying the homotopy equivalences of Lemma \ref{lem:ShrinkingHty} and gluing on quasi-isomorphisms (cf.\ Remark \ref{rem:QuasiIso}). Similarly, in the following section we shall show (in Lemma \ref{lem:EltStab}) that any cobordism having a single 1-handle may, up to homotopy, be factored as a sequence of quasi-isomorphisms and an inner cobordism of a very particular form.

\subsection{Elementary stabilisation maps}

Maps of type $\alpha$, $\beta$ and $\mu$, as defined at the end of Section \ref{sec:StabMapsCobCat}, are all given by cobordisms which have a single relative 1-handle. In this section we introduce versions of these maps in the flexible model given in the last section. It is these versions of the stabilisation maps that we will typically work with.

\begin{defn}\label{defn:EltStab}
An \emph{elementary stabilisation map} is a cobordism ${W} : (t,{Q}) \leadsto (t',{Q}')$ of inner boundary conditions such that
\begin{enumerate}[(i)]
\item ${Q}$ consists of a pair of oriented intervals, which join $(-1, \tfrac{1}{2})$ with $(1, \tfrac{1}{2})$ and $(-1, -\tfrac{1}{2})$ with $(1, -\tfrac{1}{2})$, as shown in Figure \ref{fig:EltStab} a), and 

\item $W$ has a single 1-handle relative to $Q$ which is attached as shown in Figure \ref{fig:EltStab} b).
\end{enumerate}
Under these conditions, the outgoing boundary condition ${Q}'$ will again consist of a pair of oriented intervals, which join $(-1, \tfrac{1}{2})$ with $(-1, -\tfrac{1}{2})$ and $(1, -\tfrac{1}{2})$ with $(1, \tfrac{1}{2})$, as shown in Figure \ref{fig:EltStab} c).

\begin{figure}[h]
\centering
\includegraphics[bb = 0 0 259 60]{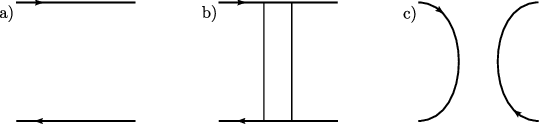}
\caption{}
\label{fig:EltStab}
\end{figure}

If we wish to emphasise this structure, we write ${Q}_{\sticks}$ and ${Q}'_{\cross}$, where the subscript records the combinatorics of how the intervals connect up the four points $(\pm 1, \pm\tfrac{1}{2})$.
\end{defn}

\begin{lem}\label{lem:EltStab}
If $(t,M) : P \leadsto P' \in \mathcal{C}_\theta$ is a cobordism having a single 1-handle relative to $P$, there there are quasi-isomorphisms (cf.\ Remark \ref{rem:QuasiIso})
$$(1,M_0) : A \leadsto P \quad\text{ and }\quad (1,M_1) : P' \leadsto B$$
in $\mathcal{C}_\theta$ such that
\begin{enumerate}[(i)]
\item $A$ and $B$ agree outside $\mathrm{int}(\bI^\infty)$; call this common submanifold $L$,

\item there is a path in $\mathcal{C}_\theta(A,B)$ from $M_1 \circ M \circ M_0$ to $([0, 2+t] \times L) \cup W$ for some elementary stabilisation map ${W} : (0,{Q}_{\sticks}) \leadsto (2+t,{Q}'_{\cross})$.
\end{enumerate}
\end{lem}
\begin{proof}
It will be important to distinguish $\theta$-manifolds from manifolds without a given $\theta$-structure, so we revert to denoting $\theta$-manifolds by $(X, \ell_X)$ for this proof.

Let $\phi : \{\pm1\} \times \bI \hookrightarrow P$ be an attaching map for the 1-handle of $M$ relative to $P$. We may isotope $(P, \ell_P)$ so that the handle attachment map is now
$$\phi' : (s, x) \mapsto (\tfrac{s}{2}, x, 0, \ldots): \{\pm1\} \times \bI \lra \bI^\infty \subset \bR^\infty$$
and $P$ intersects $\bI^\infty$ only in the set $\{\pm\tfrac{1}{2}\} \times \bI$. (Thus, the handle will be attached as in Figure \ref{fig:EltStab} b).) Let $(A, \ell_A)$ be this new object of $\mathcal{C}_\theta$, and $(1, (M_0, \ell_{M_0})) : (A, \ell_A) \leadsto (P, \ell_P)$ be the quasi-isomorphism given by the isotopy. Let $(L, \ell_L)$ be the $\theta$-manifold obtained by intersecting $(A, \ell_A)$ with $\bA$.

Let ${W} \subset [0, 1] \times \bI^\infty$ be an elementary stabilisation map (without $\theta$-structure!) given as the trace of the surgery $\phi'$ on ${Q} = \{\pm\tfrac{1}{2}\} \times \bI$, and write ${Q}'$ for the outgoing boundary. (If we unbend the corners, $W$ is just a disc.) Gluing this cobordism to $[0,1] \times L$ gives a cobordism
$$([0,1] \times L) \cup W : (0, L_Q) \leadsto (1,L_{Q'})$$
without $\theta$-structure. The incoming boundary $L_Q=A$ has a $\theta$-structure, $\ell_A$. We may change $(t,(M, \ell_M)) \circ (1,(M_0, \ell_{M_0})) \in \mathcal{C}_\theta(A, P')$ by an isotopy to obtain a $\theta$-cobordism $(1+t,(X, \ell_X)) : A \leadsto P'$ such that $X$ contains $([0, 1] \times L) \cup W$ as a subset. Furthermore, we may suppose that restricted to $[0,1] \times L$, $\ell_X$ agrees with the $\theta$-structure induced by $\ell_L$. Let $\ell_W := \ell_X\vert_{W}$, and let $Q'$ have the induced $\theta$-structure.

We have expressed $(1+t, (X, \ell_X))$ as a factorisation $(t,(Y, \ell_Y)) \circ (1,([0,1] \times L, \ell_L) \cup (W, \ell_W))$. But as $M$ and $W$ both only have a single relative 1-handle, it follows that $Y$ is a trivial cobordism, an in particular $(t,(Y, \ell_Y))$ is a quasi-isomorphism. If we let $(s,(M_1, \ell_{M_1}))$ be a quasi-inverse to $(t, (Y, \ell_Y))$, we find that the $\theta$-cobordisms
$$(s+t, [0,s+t] \times (L_{Q'})) \circ (1,([0,1] \times L) \cup W) \quad \text{   and   }\quad (s,M_1) \circ (t,M) \circ (1,M_0)$$
are in the same path component of $\mathcal{C}_\theta(A,L_{Q'})$, as required.
\end{proof}

By the above lemma, to prove homological stability for all maps of type $\alpha$, $\beta$, and $\mu$ it will be enough to do so only for those cobordisms of the form $([0,1] \times L) \cup W$ for some outer boundary condition $L$ and some elementary stabilisation map ${W}$. 

Let us revisit three of the stabilisation maps from Section \ref{sec:StabMapsCobCat}, $\alpha$, $\beta$, and $\mu$, from the point of view of the elementary stabilisation maps of Definition \ref{defn:EltStab} and the basic result Lemma \ref{lem:EltStab}. Firstly, following Definition \ref{defn:ConnOrientGenus}, we may define
$$\mathcal{M}^\theta_L(t,Q) \subset \mathcal{N}^\theta_L(t,Q)$$
to be the subspace of those $X$ which are path connected, and
$$\mathcal{M}^\theta_L(g,\pm;t,Q) \subset \mathcal{M}^\theta_L(t,Q)$$
to be the subspace of those $X$ which are of orientability type $\pm$ and of genus $g$. 

If ${W} : {Q}_{\sticks} \leadsto {Q}'_{\cross}$ is an elementary stabilisation map and $L$ is an outer boundary condition, there is an induced map
$$W_* : \mathcal{N}^\theta_L(t,Q) \lra \mathcal{N}^\theta_L(t+1,Q').$$
This restricts to maps on the subspaces of connected non-orientable or orientable surfaces of fixed genus; we record the various possibilities below.

\begin{enumerate}[(i)]
\item Suppose that the two intervals ${Q} \subset {L}_{{Q}}$ lie in different path components. Then the map induced by ${W}$ has the effect of gluing on a pair of pants along the legs, which gives
$$W_* : \mathcal{M}^\theta_{{L}}(g,+; t, {Q}) \lra \mathcal{M}^\theta_{{L}}(g+1,+; t+1, {Q}')$$
in the orientable case, or
$$W_* : \mathcal{M}^\theta_L(g,-; t,Q) \lra \mathcal{M}_L^\theta(g+2,-; t+1,Q')$$
in the non-orientable case, a map of type $\alpha$.

\item Suppose that the two intervals ${Q} \subset {L}_{{Q}}$ lie in the same path component and are coherently oriented (this is automatic in the case $\theta^*\gamma_2$ is orientable). Then the map induced by ${W}$ has the effect of gluing on a pair of pants along the waist, which gives
$$W_* : \mathcal{M}^\theta_{{L}}(g,+; t,{Q}) \lra \mathcal{M}^\theta_{{L}}(g,+; t+1,{Q}')$$
in the orientable case, or
$$W_* : \mathcal{M}^\theta_L(g,-; t,{Q}) \lra \mathcal{M}^\theta_L(g,-; t+1,{Q}')$$
in the non-orientable case, a map of type $\beta$.

\item Suppose that the two intervals ${Q} \subset L_Q$ lie in the same path component and are oppositely oriented. Then the map induced by $W$ has the effect of gluing on a projective plane with two discs removed, which gives
$$W_* : \mathcal{M}^\theta_L(g,-; t,{Q}) \lra \mathcal{M}_L^\theta(g+1,-; t+1,{Q}'),$$
a map of type $\mu$.
\end{enumerate}

Maps of type $\gamma$, i.e.\ those that glue on a disc, will not be treated using this model. We will deal with them in Section \ref{sec:LastBoundary} by different means.

\section{Resolutions}\label{sec:Resolutions}

In the previous two sections we have defined moduli spaces $\mathcal{N}^\theta(P)$ of $\theta$-surfaces with boundary $P$, and a more flexible model $\mathcal{N}_L^\theta(t, Q)$ for this space when the $\theta$-manifold $P$ arises as $L_Q$ for an outer boundary condition $L$ and an inner boundary condition $Q$. We have defined subspaces $\mathcal{M}^\theta(g,\pm;P) \subset \mathcal{N}^\theta(P)$ consisting of connected surfaces with orientability type $\pm$ and genus $g$, and similar subspaces in the flexible model.

In this section we shall construct, after choosing some auxiliary data (namely an embedding $b : \{\pm 1 \} \times \bR \hookrightarrow L$, and a choice of sign $\tau \in \{+, -\}$) an augmented semi-simplicial space
$$\epsilon : \mathcal{N}^\theta_L(t,Q;b, \tau)_\bullet \lra \mathcal{N}^\theta_L(t,Q)$$
whose $p$-simplices consist of a surface $X \in \mathcal{N}^\theta_L(t,Q)$ along with $(p+1)$ thickened arcs embedded in $X$, starting on $b(\{- 1 \} \times \bR)$ and ending on $b(\{1 \} \times \bR)$, satisfying certain conditions. We shall then show that when restricted to $\mathcal{M}^\theta_L(g,\pm;t,Q) \subset \mathcal{N}^\theta_L(t,Q)$ the fibres of the map $\vert \epsilon \vert$ have a connectivity which increases linearly with $g$, so the restriction of $\epsilon$ provides an increasingly good semi-simplicial resolution of the moduli spaces $\mathcal{M}^\theta_L(g,\pm;t,Q)$.

\begin{defn}\label{defn:ArcCx}
For an outer boundary condition $L$, an inner boundary condition $Q$ and an embedding $b : \{\pm 1 \} \times \bR \hookrightarrow L$, let $\mathcal{N}^\theta_L(t,Q;b, \tau)_0$ denote the set of those pairs $(X;a)$ where $X \in \mathcal{N}^\theta_L(t,Q)$ and $a : \bI \times \bI \hookrightarrow X$ is an embedding, such that
\begin{enumerate}[(i)]
\item $a(\{-1\} \times \bI) \subset b(\{-1\} \times \bR)$ and $a(\{1\} \times \bI) \subset b(\{1\} \times \bR)$, both preserving orientation,

\item the embedding $a$ is collared near $\{\pm 1\} \times \bI$, i.e.\ we have
\begin{align*}
a(1-s, t) &= a(1,t)-s\\
a(-1+s, t) &= a(-1,t) -s
\end{align*}
for all small enough $s \geq 0$,

\item the complement $X \setminus a(\bI \times \bI)$ is connected, and has the same orientability type as $X$ (i.e.\ if $X$ is non-orientable, $X \setminus a(\bI \times \bI)$ is required to be too).
\end{enumerate}
Topologise $\mathcal{N}^\theta_L(t, Q;b, \tau)_0$ as a subspace of $\mathcal{N}^\theta_L(t,Q) \times \Emb(\bI \times \bI, \bR \times \bR^\infty)$.

For $p \geq 0$, let $\mathcal{N}^\theta_L(t, Q;b, \tau)_p$ be the subspace of the $(p+1)$-fold fibre product of $\mathcal{N}^\theta_L(t, Q;b,\tau)_0$ over $\mathcal{N}^\theta_L(t, Q)$ consisting of tuples $(X; a_0, \ldots, a_p)$ such that
\begin{enumerate}[(i)]
\setcounter{enumi}{3}
\item the embeddings $a_i$ have disjoint images,

\item\label{it:ConnOrient} the complement $X \setminus \cup_i a_i(\bI \times \bI)$ is connected, and has the same orientability type as $X$,

\item\label{it:Ordering} the end points of the arcs $a_i(\bI \times \{0\})$ are ordered as
$$a_0(-1,0) < a_1(-1,0) < \cdots < a_p(-1,0)$$
with respect to the standard order on $b(\{-1\} \times \bR)$, and as
$$\begin{cases}
a_0(1,0) < a_1(1,0) < \cdots < a_p(1,0) & \text{ if $\tau=+$},\\
a_0(1,0) > a_1(1,0) > \cdots > a_p(1,0) & \text{ if $\tau=-$},
\end{cases}$$
with respect to the standard order on $b(\{1\} \times \bR)$.
\end{enumerate}
There are face maps $d_i : \mathcal{N}^\theta_L(t, Q;b,\tau)_p \to \mathcal{N}^\theta_L(t, Q;b,\tau)_{p-1}$ given by forgetting $a_i$, giving an augmented semi-simplicial space $\mathcal{N}^\theta_L(t, Q;b,\tau)_\bullet \to \mathcal{N}^\theta_L(t, Q)$.
\end{defn}

If $W \subset [t,t'] \times \bI^\infty$ is an inner cobordism from $(t,Q)$ to $(t',Q')$ then we define
\begin{align*}
W_* : \mathcal{N}^\theta_L(t, Q;b,\tau)_0 & \lra  \mathcal{N}^\theta_L(t', Q';b,\tau)_0\\
(X; a) &\longmapsto (X \cup W ;a)
\end{align*}
The analogous formula defines a map on higher simplices, so the construction
$$(t, Q) \longmapsto \{\epsilon : \mathcal{N}^\theta_L(t, Q;b,\tau)_\bullet \to \mathcal{N}^\theta_L(t, Q)\}$$
is functorial for inner cobordisms.

\begin{prop}\label{prop:AugIsQfib}
The map $\vert\epsilon\vert : \vert \mathcal{N}^\theta_L(t, Q;b,\tau)_\bullet \vert \to \mathcal{N}^\theta_L(t, Q)$ is a locally trivial fibre bundle.
\end{prop}

\begin{proof}
The fibre of the augmented semi-simplicial space $\epsilon :  \mathcal{N}^\theta_L(t,Q;b,\tau)_\bullet  \to \mathcal{N}^\theta_L(t,Q)$ over $X$ is the semi-simplicial space $A(X)_\bullet$ with $p$-simplices the space of $(p+1)$-tuples of embeddings $e : \bI \times \bI \hookrightarrow X$ satisfying the conditions described in Definition \ref{defn:ArcCx}. As such, the group $\Diff_\partial(X)$ acts on $A(X)_\bullet$ levelwise. If we let $\Emb_\partial(X, \mathcal{U}_{t} ; inc)$ be the space of embedding of $X$ into $\mathcal{U}_{t}$ which are equal to the identity embedding near the boundary, equipped with the $C^\infty$-topology, then we can recover that part of $\mathcal{N}^\theta_L(t,Q;b,\tau)_\bullet$ lying over the path component of $X$ (up to homeomorphism) as the semi-simplicial space
$$[p] \longmapsto \left(A(X)_p \times \Emb_\partial(X, \mathcal{U}_{t} ; inc) \times \Bun^\theta_\partial(X;\ell_L \cup \ell_Q)\right)/\,\Diff_\partial(X),$$
using the techniques of Section \ref{sec:BorelConst}. Hence by Lemma \ref{lem:AugTriv} the augmentation map is a locally trivial fibre bundle, with fibre $\vert A(X)_\bullet\vert$ over $X$.
\end{proof}

There are three basic situations that we shall refer to thoughout, which concern the combinatorial situations the 1-manifolds $Q$, $L$, and $b(\{\pm1\} \times \bR)$ can form. The following figures show three cases of how these data can be arranged, and in the rest of our discussion we shall refer to these figures to mean any tuple of data $(Q, L, b)$ which has the combinatorial form shown in these figures. The figures only show those parts of $L$ which touch $\partial \bA$: they should be interpreted as saying that disjoint components of $L$ may be freely added if required.

\begin{figure}[H]
\centering
\includegraphics[bb = 0 0 157 59]{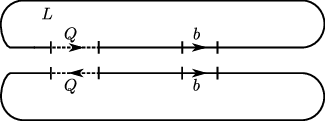}
\caption{The oriented intervals ${Q} \subset L_Q$ lie in different components. One of each of the intervals $b$ lies in each of these components, and one is oriented coherently with ${Q}$, and the other is oriented oppositely to ${Q}$.}\label{fig:BoundaryRes}
\end{figure}

\begin{figure}[H]
\centering
\includegraphics[bb = 0 0 159 33]{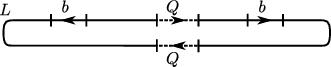}
\caption{The oriented intervals ${Q} \subset L_Q$ lie in a single component and are coherently oriented. The intervals $b$ lie in the same component, have opposite orientations, and separate ${Q}$.}\label{fig:HandleRes}
\end{figure}

\begin{figure}[H]
\centering
\includegraphics[bb = 0 0 156 61]{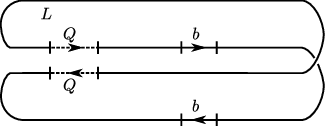}
\caption{The oriented intervals ${Q} \subset L_Q$ lie in a single component and are oppositely oriented. The intervals $b$ lie in the same component, are coherently oriented, and do not separate ${Q}$.}\label{fig:ProjRes}
\end{figure}

\begin{thm}\label{thm:ResConnectivity1}
Let $X \in \mathcal{N}^\theta_L(t, Q)$, and let $F_X$ denote the homotopy fibre of the map $\vert\epsilon\vert$ over $X$.
\begin{enumerate}[(i)]
\item If $X$ is connected, orientable, and has genus $g$, the data $(Q,L,b)$ is as in Figure \ref{fig:BoundaryRes}, and $\tau=+$, then $F_X$ is $(g-2)$-connected.

\item If $X$ is connected, orientable, and has genus $g$, the data $(Q,L,b)$ is as in Figure \ref{fig:HandleRes}, and $\tau=+$, then $F_X$ is $(g-2)$-connected.

\item If $X$ is connected, non-orientable, and has genus $g$, the data $(Q,L,b)$ is as in Figure \ref{fig:ProjRes}, and $\tau=-$, then $F_X$ is $(\lfloor \tfrac{g-1}{3}\rfloor-1)$-connected.
\end{enumerate}
\end{thm}
\begin{proof}
By Proposition \ref{prop:AugIsQfib}, and its proof, the fibre over $X$ is the geometric realisation of the semi-simplicial space $A(X)_\bullet$. Let us write $\pi_0A(X)_\bullet$ for the semi-simplicial set obtained as the levelwise sets of path components. We first claim that 
$$\vert A(X)_\bullet \vert \lra \vert \pi_0A(X)_\bullet \vert$$
is a weak homotopy equivalence. We shall show this by showing that in fact it is a levelwise weak homotopy equivalence, i.e.\ that each $A(X)_p$ has contractible path components.

This relies on a theorem of Gramain \cite[Th\'{e}or\`{e}me 5]{Gramain} which we rephrase here: let $F$ be a compact surface with boundary, and $x_0$, $x_1$ be distinct points on $\partial F$. Let $P(([0,1], 0, 1), (F, x_0, x_1))$ denote the space of smooth embeddings $f:[0,1] \to F$ sending $0,1$ to $x_0, x_1$ respectively and being disjoint from the boundary otherwise, equipped with the $C^\infty$ topology. Gramain's theorem is that this space has contractible components.

First define a semi-simplicial space $A'(X)_\bullet$ having $0$-simplices given by collared embeddings $f : \bI \hookrightarrow X$  starting at $b(\{ -1\} \times \bR)$ and ending at $b(\{ 1\} \times \bR)$, having path connected complement of the same orientability type as $X$. A $(p+1)$-tuple $(f_0, \ldots, f_p)$ spans a $p$-simplex if the embeddings $f_i$ are disjoint, the complement $X \setminus \cup_i f_i(\bI)$ is path connected and of the same orientability type as $X$, and the endpoints of the arcs satisfy the ordering criterion of Definition \ref{defn:ArcCx} (\ref{it:Ordering}). There is a semi-simplicial map $A(X)_\bullet \to A'(X)_\bullet$ given on 0-simplices by $e \mapsto e\vert_{\bI \times \{0\}}$ and by the analogous formula on higher simplices, and this is a levelwise weak homotopy equivalence as an arc has a contractible space of thickenings.

To show $A'(X)_p$ has contractible path components, we proceed by induction on $p$. The map
$$A'(X)_0 \lra (b(\{-1\} \times \bR)) \times (b(\{1\} \times \bR))$$
given by $f \mapsto (f(-1), f(1))$ is a fibration over a contractible space, and its fibre over $(x_0, x_1)$ is homeomorphic to Gramain's space of embeddings of an arc in $X$ with fixed endpoints $x_0$ and $x_1$, so has contractible path components.

For $p > 0$, the face map $d_0 : A'(X)_p \to A'(X)_{p-1}$ is a locally trivial fibre bundle by the main result of \cite{Palais} (see \cite{Lima} for a short proof), and the fibre over $(f_1, \ldots, f_p)$ is the space of embedded arcs in $X \setminus \cup_{i=1}^p f_i(\bI)$ with endpoints in certain intervals in the boundary. This is nothing but the space $A'(X \setminus \cup_{i=1}^p f_i(\bI))_0$ for a particular choice of intervals in the boundary, so has contractible path components by the argument above. By inductive hypothesis $A'(X)_{p-1}$ has contractible path components, so $A'(X)_{p}$ does too. This finishes the proof that $\vert A(X)_\bullet \vert \to \vert \pi_0A(X)_\bullet \vert$ is a weak homotopy equivalence. 

To finish the proof of the theorem, we must show that $\vert \pi_0A(X)_\bullet \vert$ has a certain connectivity when $X$, $(Q,L,b)$ and $\tau$ are as in the statement of the theorem. For the remainder of the proof, we will refer to definitions and results in Appendix \ref{app:CxArcs}, which the reader will need to consult. In cases (i) and (ii), there is a map
$$\vert \pi_0A(X)_\bullet \vert \lra B_0(X)$$
to the simplicial complex $B_0(X)$ defined in Section \ref{sec:CxArcsOrientable}, given by taking a simplex $(e_0, \ldots, e_p)$ to the collection of arcs $e_i\vert_{\bI \times \{0\}}$, then isotoping their endpoints in the intervals $b(\{\pm 1\} \times \bR)$ so that their endpoints lie at the centres $b(\{\pm 1\} \times \{0\})$. A path component of $A(X)_p$ is determined by the isotopy classes of the arcs $e_i\vert_{\bI \times \{0\}}$, so this map is a homeomorphism. In Theorem \ref{thm:OrientableComplex} we show that $B_0(X)$ is $(g-2)$-connected when $X$ is connected, orientable, and of genus $g$. (In case (i), where the two marked intervals are on the same boundary component, Ivanov has already shown this connectivity result.)

In case (iii) the same formula as above defines a map
$$\vert \pi_0A(X)_\bullet \vert \lra C_0(X)$$
to the simplicial complex $C_0(X)$ defined in Section \ref{sec:CxArcsNonorientable}, and this map is again a homeomorphism. In Theorem \ref{thm:MobiusComplex}, we show that $C_0(X)$ is $(\lfloor \frac{g-1}{3} \rfloor -1)$-connected when $X$ is connected, non-orientable, and of genus $g$.
\end{proof}

There are two remaining cases we would like the above result for, the analogues of (i) and (ii) for non-orientable surfaces. Unfortunately we do not know how to show that the associated simplicial complexes are highly connected, but using an idea of Wahl we are able to show that they become contractible after stabilising by projective planes. Let us show how this idea may be used.

\begin{defn}\label{defn:StabCob}
Let $L$ be an outer boundary condition, and $b : \{ \pm 1\} \times \bR \hookrightarrow L$ be given. A cobordism of outer boundary conditions $K : (-1,L) \leadsto (0,L')$ is called a \emph{stabilising cobordism for $(L,b)$} if
\begin{enumerate}[(i)]
\item $K$ contains $[-1,0] \times b(\{\pm 1\} \times \bR)$ as a $\theta$-submanifold,

\item $K$ is diffeomorphic relative to its boundary to the manifold formed from $[-1,0] \times L$ by taking the connected sum with $\bR\bP^2$ inside a component which touches the image of the map $b$.
\end{enumerate}
\end{defn}

Given $L$ and $b$, stabilising cobordisms for this data exist: form the ambient connect-sum of $[-1,0] \times L$  with a disjoint embedded copy of $\bR\bP^2$, where the connect-sum is formed disjointly from $[-1,0] \times b(\{\pm 1\} \times \bR) \subset [-1,0] \times L$ and is formed in a component of $[-1,0] \times L$ which touches the image of $b$, and then choose a $\theta$-structure $\ell_K$ on $K$ which agrees with $\ell_L$ on $(\{-1\} \times L) \cup ([-1,0] \times b(\{\pm 1\} \times \bR))$; this induces a new $\theta$-structure $\ell'_L$ on $\{0\} \times L \subset K$, and we call the new outer boundary condition so obtained $L'$. (Note that it is possible to choose such a $\ell_K$: the map $\theta_* : \pi_1(B) \to \pi_1(BO(2))$ must be surjective---or else $\theta^*\gamma_2$ is orientable and so no non-orientable surface admits a $\theta$-structure, making this discussion unnecessary---and so $\theta$-structures can always be extended over a 1-handle with nullhomotopic attaching map, by elementary obstruction theory.)

If we write $L_0 = L$, $L_1 = L'$, and $K_0 = K$, there is a map
\begin{align*}
{K^*_0}: \mathcal{N}^\theta_{L_0}(t,Q) &\lra \mathcal{N}^\theta_{L_1}(t-1,Q)\\
X &\longmapsto (X -e_0) \cup K_0
\end{align*}
which lifts to a map of augmented semi-simplicial spaces
$$(K^*_0)_\bullet : \mathcal{N}^\theta_{L_0}(t,Q;b,\tau)_\bullet \lra \mathcal{N}^\theta_{L_1}(t-1,Q;b,\tau)_\bullet,$$
by extending arcs $a : \bI \times \bI \hookrightarrow X$ cylindrically by $[0,1] \times a(\{\pm 1\} \times \bI)$, and then reparametrising. By iterating this construction (that is, choosing a stabilising cobordism $K_1 : L_1 \leadsto L_2$ for the data $(L_1, b)$, and so on) we obtain a direct system of augmented semi-simplicial spaces, and we can consider
$$\epsilon' : \underset{n \to \infty} \hocolim \mathcal{N}^\theta_{L_{n}}(t-n,Q;b,\tau)_\bullet \lra \underset{n \to \infty} \hocolim \mathcal{N}^\theta_{L_{n}}(t-n,Q),$$
the augmented semi-simplicial space obtained by taking the levelwise homotopy colimits.

\begin{thm}\label{thm:ResConnectivity2}
The homotopy fibre of $\vert\epsilon'\vert$ over a point $X \in \mathcal{N}^\theta_L(t,Q)$ is contractible if either
\begin{enumerate}[(i)]

\item $X$ is connected and non-orientable, the data $(Q,L,b)$ is as in Figure \ref{fig:HandleRes}, and $\tau=+$, or

\item $X$ is connected and non-orientable, the data $(Q,L,b)$ is as in Figure \ref{fig:BoundaryRes}, and $\tau=+$.
\end{enumerate}
\end{thm}
\begin{proof}
The first part of the proof of Theorem \ref{thm:ResConnectivity1} still applies. In case (i), if we let
$$b_- := b(\{-1\} \times \{0\}) \quad\quad\quad b_+ := b(\{1\} \times \{0\})$$
then the usual map gives a homeomorphism
$$\vert \pi_0A(X)_\bullet \vert \lra D_0(X, b_-, b_+)$$
to a simplicial complex which we define in Section \ref{sec:CxArcsNonorientable}. The stabilisation map $(K^*)_\bullet$ described above induces a map
$$D_0(X, b_-, b_+) \lra D_0(Y, b_-, b_+)$$
where $Y \cong X \# \bR\bP^2$. By Theorem \ref{thm:NonorientableHandleComplex} the homotopy colimit of countably many iterations of this construction is contractible, which establishes the theorem in this case. In case (ii) the argument is the same, using the simplicial complex $E_0(X, b_-, b_+)$ and Theorem \ref{thm:NonorientableBoundaryComplex}.
\end{proof}

Recall that we have defined subspaces
$$\mathcal{M}_L^\theta(g,\pm; t,Q)  \subset  \mathcal{N}_L^\theta(t,Q)$$
consisting of those surfaces which are connected and have fixed orientation type and genus. Pulling back $\epsilon : \mathcal{N}^\theta_L(t,Q;b,\tau)_\bullet \to \mathcal{N}^\theta_L(t,Q)$ to these subspaces defines certain augmented semi-simplicial spaces which we shall give their own names and notation, as they will be the principal objects we consider in the remainder of the paper.

\begin{defn}\mbox{}
\begin{enumerate}[(i)]
\item If $(Q, L, b)$ are as in Figure \ref{fig:BoundaryRes} and $\tau=+$, then we write
$$\mathcal{B}^\theta_L(g,\pm; t,{Q};b)_\bullet \lra \mathcal{M}^\theta_L(g,\pm; t,Q)$$
for the pulled back augmented semi-simplicial space, and call it the \emph{boundary resolution}.

\item If $(Q, L, b)$ are as in Figure \ref{fig:HandleRes} and $\tau=+$, then we write
$$\mathcal{H}^\theta_L(g,\pm; t,{Q};b)_\bullet \lra \mathcal{M}_L^\theta(g,\pm; t,Q)$$
for the pulled back augmented semi-simplicial space, and call it the \emph{handle resolution}.

\item If $(Q, L, b)$ are as in Figure \ref{fig:ProjRes} and $\tau=-$, then we write 
$$\mathcal{P}^\theta_L(g,-; t,{Q};b)_\bullet \lra \mathcal{M}^\theta_L(g,-; t,Q)$$
for the pulled back augmented semi-simplicial space, and call it the \emph{projective plane resolution}.
\end{enumerate}
\end{defn}

\begin{rem}
Note that we define (iii) only for non-orientable surfaces: the pullback to $\mathcal{M}^\theta_{{L}}(g,+; t,{Q})$ in this case is empty, as if there is an arc in a surface with endpoints on the same boundary whose endpoints are coherently oriented, then the surface must contain a M{\"o}bius band.
\end{rem}

An important feature of these semi-simplicial resolutions is that elementary stabilisation maps ${W} : (0,{Q}) \leadsto (1,{Q}')$ mix them. The following proposition records the effect of starting with data $(Q, L, b)$ as in Figures \ref{fig:BoundaryRes}, \ref{fig:HandleRes} or \ref{fig:ProjRes} and gluing on an elementary stabilisation map ${W} : (0,{Q}) \leadsto (1,{Q}')$. The proof of this proposition is immediate, by checking the combinatorics in each case.

\begin{prop}\label{prop:StabMapsOnRes}
In the case of orientable surfaces:
\begin{enumerate}[(i)]
\item If $(Q,L,b)$ is as in Figure \ref{fig:BoundaryRes}, then $W$ induces a map of type $\alpha$ which is covered by a map of resolutions $\mathcal{B}^\theta_{{L}}(g,+; 0,{Q};b)_\bullet \to \mathcal{H}^\theta_{{L}}(g+1,+; 1,{Q}';b)_\bullet$.

\item If $(Q,L,b)$ is as in Figure \ref{fig:HandleRes}, then $W$ induces a map of type $\beta$ which is covered by a map of resolutions $\mathcal{H}^\theta_{{L}}(g,+; 0,{Q};b)_\bullet \to \mathcal{B}^\theta_{{L}}(g,+; 1,{Q}';b)_\bullet$.
\end{enumerate}
In the case of non-orientable surfaces:
\begin{enumerate}[(i)]
\setcounter{enumi}{2}
\item If $(Q,L,b)$ is as in Figure \ref{fig:BoundaryRes}, then $W$ induces a map of type $\alpha$ which is covered by a map of resolutions $\mathcal{B}^\theta_{{L}}(g,-; 0,{Q};b)_\bullet \to \mathcal{H}^\theta_{{L}}(g+2,-; 1,{Q}';b)_\bullet$.

\item If $(Q,L,b)$ is as in Figure \ref{fig:HandleRes}, then $W$ induces a map of type $\beta$ which is covered by a map of resolutions $\mathcal{H}^\theta_{{L}}(g,-; 0,{Q};b)_\bullet \to \mathcal{B}^\theta_{{L}}(g,-; 1,{Q}';b)_\bullet$.

\item If $(Q,L,b)$ is as in Figure \ref{fig:ProjRes}, then $W$ induces a map of type $\mu$ which is covered by a map of resolutions $\mathcal{P}^\theta_{{L}}(g,-; 0,{Q};b)_\bullet \to \mathcal{P}^\theta_{{L}}(g+1,-; 1,{Q}';b)_\bullet$.
\end{enumerate}
\end{prop}

\subsection{The layers of the resolutions}\label{sec:Layers}

It will be important for us to understand the homotopy types of $\mathcal{H}^\theta_L(g,\pm; t,Q;b)_p$, $\mathcal{B}^\theta_L(g,\pm; t,Q;b)_p$, and $\mathcal{P}^\theta_L(g,-; t,Q;b)_p$, the spaces in the resolutions that we have just defined, in terms of moduli spaces of surfaces with $\theta$-structure.

\begin{defn}
For an embedding $b : \{\pm1\} \times \bR \hookrightarrow L$, let us write
$$\ell_b : \epsilon^1 \oplus T(\{\pm 1\} \times \bR) \overset{\epsilon^1 \oplus Db}\lra \epsilon^1 \oplus TL \overset{\ell_L}\lra \theta^*\gamma_2.$$
Given in addition a choice of sign $\tau$, let $A_p(t;b,\ell_b, \tau)$ be the space of tuples $(a_0, \ldots, a_p;\ell_0, \ldots, \ell_p)$ of $(p+1)$ embeddings
$$a_i : \bI \times \bI \lra \mathcal{U}_{t} := (-\infty,0] \times \bA) \cup ((-\infty,t] \times \bI^\infty),$$
and $(p+1)$ bundle maps $\ell_i : T(\bI \times \bI) \to \theta^*\gamma_2$, such that
\begin{enumerate}[(i)]

\item $a_i(\{-1\} \times \bI) \subset b(\{-1\} \times \bR)$ and $a_i(\{1\} \times \bI) \subset b(\{1\} \times \bR)$, both preserving orientation,

\item the embeddings $a_i$ are collared near $\{\pm 1\} \times \bI$, i.e.\ we have
\begin{align*}
a_i(1-s, t) &= a_i(1,t)-s\\
a_i(-1+s, t) &= a_i(-1,t) -s
\end{align*}
for all small enough $s \geq 0$,

\item the embeddings $a_i$ have disjoint images,

\item the end points of the arcs $a_i(\bI \times \{0\})$ are ordered as
$$a_0(-1,0) < a_1(-1,0) < \cdots < a_p(-1,0)$$
with respect to the standard order on $b(\{-1\} \times \bR)$, and as
$$\begin{cases}
a_0(1,0) < a_1(1,0) < \cdots < a_p(1,0) & \text{ if $\tau=+$},\\
a_0(1,0) > a_1(1,0) > \cdots > a_p(1,0) & \text{ if $\tau=-$},
\end{cases}$$
with respect to the standard order on $b(\{1\} \times \bR)$,

\item $\ell_i\vert_{\{\pm 1\} \times \bI} = (a_i\vert_{\{\pm 1\} \times \bI})^*(\ell_b)$ for each $i$.
\end{enumerate}
\end{defn}

The collection $A_\bullet(t;b,\ell_b, \tau)$ may be given the structure of a semi-simplicial space, where the $i$th face map forgets the data $(a_i, \ell_i)$, but in fact this structure will not play a role in our use of the spaces $A_p(t;b,\ell_b, \tau)$. Instead, the similarity of the axioms defining $A_p(t;b,\ell_b, \tau)$ with the axioms defining $\mathcal{N}_L^\theta(t,Q;b, \tau)_p$ has been chosen so that there is a map
\begin{align*}
r_p : \mathcal{N}_L^\theta(t,Q;b, \tau)_p &\lra A_p(t;b,\ell_b, \tau)\\
(X;a_0, \ldots, a_p) &\longmapsto (a_0, \ldots, a_p;a_0^*(\ell_X), \ldots, (a_p)^*(\ell_X)).
\end{align*}
Note that if $t \leq t'$ then $A_p(t;b,\ell_b, \tau) \subset A_p(t';b,\ell_b, \tau)$, and the inclusion map is a weak homotopy equvalence. If $W : (t, Q) \leadsto (t', Q')$ in an inner cobordism, then the square
\begin{equation*}
\xymatrix{
\mathcal{N}_L^\theta(t,Q;b, \tau)_p \ar[r]^-{W_*} \ar[d]^-{r_p}& \mathcal{N}_L^\theta(t',Q';b, \tau)_p \ar[d]^-{r_p}\\
A_p(t;b,\ell_b, \tau) \ar@{^(->}[r]^-{\simeq}& A_p(t';b,\ell_b, \tau)
}
\end{equation*}
commutes.

\begin{lem}
The map $r_p : \mathcal{N}_L^\theta(t,Q;b, \tau)_p \to A_p(t;b,\ell_b, \tau)$ is a Serre fibration.
\end{lem}
\begin{proof}
Let us explain why this map has path lifting in the case $p=0$: the argument will clearly extend to the parametrised case and to $p > 0$, with the addition of some notation. 

Let $(X, \ell_X;a) \in \mathcal{N}_L^\theta(t,Q;b, \tau)_0$, and let $f_s = (a_s, \ell_s)$ be a path in $A_0(t;b,\ell_b, \tau)$ starting at $(a_0, \ell_0)=r_0(X, \ell_X;a) = (a, a^*(\ell_X))$. By the isotopy extension theorem, the isotopy $s \mapsto a_s$ of embeddings of $\bI \times \bI$ into $\mathcal{U}_t$ extends to an isotopy $\varphi_s : \mathcal{U}_t \to \mathcal{U}_t$ which is constantly the identity near $\partial \mathcal{U}_t$ and is compactly supported (in other words, its support is a compact subset of the interior of $\mathcal{U}_t$). Then $a_s = \varphi_s \circ a$, and we may define a path
$$g_s = (\varphi_s(X), (D\varphi_s)^{-1} \circ \ell_X; \varphi_s \circ a) \in \mathcal{N}_L^\theta(t,Q;b, \tau)_0.$$
This satisfies $r_0(g_s) = (a_s, (a_s)^*((D\varphi_s)^{-1} \circ \ell_X)) = (a_s, a^*(\ell_X))$, which agrees with $f_s$ in the first coordinate but not necessarily in the second. However, this may be easily fixed, using that the restriction map
$$\rho : \Bun_\partial(TX, \theta^*\gamma_2) \lra \Bun_\partial(T(\bI \times \bI), \theta^*\gamma_2)$$
given by $\rho(\ell) = a^*(\ell)$ is a Serre fibration (as $a$ is a cofibration). If $\hat{\ell}_s$ denotes a lift along $\rho$ of the path $\ell_s$ starting at $\ell_X$, then the path
$$g'_s = (\varphi_s(X), (D\varphi_s)^{-1} \circ \hat{\ell}_s; \varphi_s \circ a) \in \mathcal{N}_L^\theta(t,Q;b, \tau)_0$$
is a lift of $f_s$, as required.
\end{proof}

Let $x=(a_0, \ldots, a_p;\ell_0, \ldots, \ell_p) \in A_p(t;b,\ell_b, \tau)$ be such that all the $a_i$ have image inside $[-1,0] \times \bA$, and choose an outer cobordism \emph{without $\theta$-structure}
$$R \subset [-1,0] \times \bA$$
from $L'$ to $L$ in such a way that $R$
\begin{enumerate}[(i)]

\item contains the images of the $a_i$, and 

\item the inclusion
$$((1-\epsilon,1] \times L) \cup \bigcup_{i=0}^p a_i(\bI \times \bI) \hookrightarrow R$$
is an isotopy equivalence, for $\epsilon$ small enough.
\end{enumerate}
By (ii), there is a $\theta$-structure $\ell_{R_x}$ on $R$ extending that on $L$ and such that $a_i^*\ell_{R_x} = \ell_i$, and moreover $\ell_{R_x}$ is unique up to homotopy of $\theta$-structures having these properties. Choose such an $\ell_{R_x}$, and let $R_x :=(R, \ell_{R_x})$ and $L_x$ be the manifold $L'$ with the $\theta$-structure induced by $\ell_{R_x}$. There is then a map
\begin{equation}\label{eq:model}
\begin{aligned}
\iota : \mathcal{N}^\theta_{L_x}(t+1,Q) &\lra r_p^{-1}(x)\\
X &\longmapsto ((X - e_0) \cup R_x; a_0, \ldots, a_p)
\end{aligned}
\end{equation}
into the fibre over $x$, which is the inclusion of the subspace consisting of those manifolds which contain $R_x$. As any compact family of manifolds in this fibre contains $((1-\epsilon,1] \times L) \cup \bigcup_i a_i(\bI \times \bI)$ for some small $\epsilon$, and by assumption this is isotopy equivalent to $R$, it follows that \eqref{eq:model} is a weak homotopy equivalence. Hence the composition
$$\mathcal{N}^\theta_{L_x}(t+1,Q) \lra r_p^{-1}(x) \lra \mathrm{hofib}_{x}(r_p)$$
is a weak homotopy equivalence, and as every path component of $A_p(t;b,\ell_b, \tau)$ contains a point $x$ satisfying the assumptions above (i.e.\ that all $a_i$ lie inside $[-1,0]\times\bA$) we have identified the homotopy fibre of $r_p$ over each path component of $A_p(t;b,\ell_b, \tau)$.

\begin{prop}\label{prop:MapOnFibres}\mbox{}
\begin{enumerate}[(i)]
\item Each homotopy fibre of $r_p: \mathcal{H}_{{L}}^\theta(g,+;t,{Q};b)_p \to A_p(t;b,\ell_b,+)$ has the homotopy type of $\mathcal{M}^\theta_{{L}_x}(g-p-1,+;t+1,{Q})$,

\item Each homotopy fibre of $r_p:\mathcal{B}_{{L}}^\theta(g,+;t,{Q};b)_p \to A_p(t;b,\ell_b,+)$ has the homotopy type of $\mathcal{M}^\theta_{{L}_x}(g-p,+;t+1,{Q})$,

\item Each homotopy fibre of $r_p:\mathcal{P}_L^\theta(g,-;t,Q;b)_p \to A_p(t;b,\ell_b, -)$ has the homotopy type of $\mathcal{M}^\theta_{L_x}(g-p-1,-;t+1,Q)$,

\item Each homotopy fibre of $r_p:\mathcal{H}_L^\theta(g,-;t,Q;b)_p \to A_p(t;b,\ell_b, +)$ has the homotopy type of $\mathcal{M}^\theta_{L_x}(g-2(p+1),-;t+1,Q)$,

\item Each homotopy fibre of $r_p:\mathcal{B}_L^\theta(g,-;t,Q;b)_p \to A_p(t;b,\ell_b, +)$ has the homotopy type of $\mathcal{M}^\theta_{L_x}(g-2p,-;t+1,Q)$.
\end{enumerate}
\end{prop}

\begin{rem}
In the above, we remind the reader that genus 0 surfaces cannot be non-orientable, so $\mathcal{M}^\theta_{L_x}(0,-;t,Q) = \emptyset$.
\end{rem}

\begin{proof}
Each case follows by restricting the homotopy equivalence \eqref{eq:model} to the intersection of $r_p^{-1}(x)$ with one of the subspaces
$$\mathcal{H}_L^\theta(g,\pm;t,Q;b)_p, \mathcal{B}_L^\theta(g,\pm;t,Q;b)_p, \mathcal{P}_L^\theta(g,\pm;t,Q;b)_p \subset \mathcal{N}_L^\theta(t,Q;b, \pm)_p$$
then a) checking that the surfaces obtained are connected and have the orientability type claimed, and b) checking that the surfaces obtained have the genus claimed. Recall that these spaces are only defined when $Q$ consists of a pair of intervals as in Figure \ref{fig:EltStab} a), and the intervals $Q$ and $b(\{\pm 1\} \times \bR) \subset L_Q$ are compatible in a certain way.

Note that the surfaces obtained are precisely those $X \in \mathcal{N}^\theta_{L_x}(t+1,Q)$ such that
$$(X \cup R;a_0, \ldots, a_p) \in \mathcal{H}_L^\theta(g,\pm;t,Q;b)_p, \mathcal{B}_L^\theta(g,\pm;t,Q;b)_p \text{ or } \mathcal{P}_L^\theta(g,-;t,Q;b)_p,$$
so alternatively they are those obtained by subtracting a copy of $R$ from a connected surface of orientability type $\pm$ and genus $g$. Up to diffeomorphism, we may as well subtract $((1-\epsilon,1] \times L) \cup \bigcup_{i=0}^p a_i(\bI \times \bI)$ from a connected surface containing a copy of $L$ embedded in its boundary.

Part a) thus follows immediately from Definition \ref{defn:ArcCx} (\ref{it:ConnOrient}), where we assumed that subtracting the arcs $a_0, \ldots, a_p$ gave a connected surface of the same orientability type.

For part b), note that as the surfaces obtained are connected we may compute their genus by computing their Euler characteristic: subtracting $(p+1)$ arcs from a surface changes its Euler characteristic by $+(p+1)$, and the conditions imposed on the pair of intervals $b(\{\pm 1\} \times \bR) \subset L_Q$ in the three types of resolution determine how the number of boundary conditions change. In particular
\begin{enumerate}[(i)]
\item For $\mathcal{H}_L^\theta(g,\pm;t,Q;b)_p$ the intervals $b(\{\pm 1\} \times \bR) \subset L_Q$ lie in a single component and are oppositely oriented. By the ordering condition at each end of the arcs, it follows that cutting each arc out creates a new boundary component, and so the genus of the new surface is as claimed.

\item For $\mathcal{B}_L^\theta(g,\pm;t,Q;b)_p$ the intervals $b(\{\pm 1\} \times \bR) \subset L_Q$ lie in different components. Subtracting the first arc reduces the number of boundary components by 1, but by the ordering condition at each end of the arcs subtracting subsequent arcs increases the number of boundary components by 1, so the genus of the surface is as claimed.

\item For $\mathcal{P}_L^\theta(g,-;t,Q;b)_p$ the intervals $b(\{\pm 1\} \times \bR) \subset L_Q$ lie in a single component and are coherently oriented. Thus subtracting each arc preserves the number of boundary components, and so the genus of the new surface is as claimed.\qedhere
\end{enumerate}
\end{proof}

The above discussion in the case of Proposition \ref{prop:StabMapsOnRes} (i), where $W$ gives a map of type $\alpha$, gives a commutative diagram
\begin{equation}\label{eq:ExampleSquare}
\begin{gathered}
\xymatrix{
\mathcal{M}^\theta_{{L}_x}(g,+;0,{Q})  \ar@/_5pc/[ddd]_-{\substack{R^* \\ \text{type } \beta}} \ar[d]^-\simeq \ar[rr]^-{W_* \text{ type } \beta} && \mathcal{M}^\theta_{{L}_x}(g,+;1,{Q}') \ar[d]^-\simeq \ar@/^6pc/[ddd]^-{\substack{R^* \\ \text{type } \alpha}}\\
r_0^{-1}(x) \ar[d] \ar[rr] && r_0^{-1}(x) \ar[d]\\
\mathcal{B}^\theta_{{L}}(g,+; 0,{Q};b)_0 \ar[d] \ar[rr] && \mathcal{H}^\theta_{{L}}(g+1,+; 1,{Q}';b)_0 \ar[d]\\
\mathcal{M}^\theta_{{L}}(g,+;0,{Q}) \ar[rr]^-{W_* \text{ type } \alpha} & &\mathcal{M}^\theta_{{L}}(g+1,+;1,{Q}')
}
\end{gathered}
\end{equation}
where the map induced on fibres over $x \in A_0(0;b,\ell_b,+) \overset{\sim}\to A_0(1;b,\ell_b,+)$ is now a map of type $\beta$. A similar observation can be made in each case covered by Proposition \ref{prop:StabMapsOnRes}, and for simplices of all dimensions. In the following proposition we record what happens in each of the five cases covered by Proposition \ref{prop:StabMapsOnRes}. For simplicity of notation, we write $\iota$ for any of the natural maps $A_p(0;b,\ell_b, \tau) \overset{\sim}\to A_p(1;b,\ell_b, \tau)$; which of these maps we mean will be clear from the context.

\begin{prop}\label{prop:FibResOfPairs}
In the case of orientable surfaces
\begin{enumerate}[(i)]
\item\label{it:FibResOfPairs:ResAlpha} a map $\mathcal{B}^\theta_{{L}}(g,+; 0,{Q};b)_p \to \mathcal{H}^\theta_{{L}}(g+1,+; 1,{Q}';b)_p$ arising from resolving an elementary stabilisation map of type $\alpha$ has induced map on fibres over $\iota$ homotopy equivalent to an elementary stabilisation map
$$\mathcal{M}^\theta_{{L}_x}(g-p,+;0,{Q}) \lra \mathcal{M}^\theta_{{L}_x}(g-p,+;1,{Q}')$$
of type $\beta$,

\item\label{it:FibResOfPairs:ResBeta} a map $\mathcal{H}^\theta_{{L}}(g,+; 0,{Q};b)_p \to \mathcal{B}^\theta_{{L}}(g,+; 1,{Q}';b)_p$ arising from resolving an elementary stabilisation map of type $\beta$ has induced map on fibres over $\iota$ homotopy equivalent to an elementary stabilisation map
$$\mathcal{M}^\theta_{{L}_x}(g-p-1,+;0,{Q}) \lra \mathcal{M}^\theta_{{L}_x}(g-p,+;1,{Q}')$$
of type $\alpha$,
\end{enumerate}
and in the case of non-orientable surfaces
\begin{enumerate}[(i)]
\setcounter{enumi}{2}
\item\label{it:FibResOfPairs:ResMu} a map $\mathcal{P}^\theta_L(g,-; 0,Q;b)_p \to \mathcal{P}^\theta_L(g+1,-; 1,Q';b)_p$ arising from resolving an elementary stabilisation map of type $\mu$ has induced map on fibres over $\iota$ homotopy equivalent to an elementary stabilisation map
$$\mathcal{M}^\theta_{L_x}(g-p-1,-;0,Q) \lra \mathcal{M}^\theta_{L_x}(g-p,-;1,Q')$$
of type $\mu$,

\item a map $\mathcal{B}^\theta_L(g,-; 0,Q;b)_p \to \mathcal{H}^\theta_L(g+2,-; 1,Q';b)_p$ arising from resolving an elementary stabilisation map of type $\alpha$ has induced map on fibres over $\iota$ homotopy equivalent to an elementary stabilisation map
$$\mathcal{M}^\theta_{L_x}(g-2p,-;0,Q) \lra \mathcal{M}^\theta_{L_x}(g-2p,-;1,Q')$$
of type $\beta$,

\item a map $\mathcal{H}^\theta_L(g,-; 0,Q;b)_p \to \mathcal{B}^\theta_L(g,-; 1,Q';b)_p$ arising from resolving an elementary stabilisation map of type $\beta$ has induced map on fibres over $\iota$ homotopy equivalent to an elementary stabilisation map
$$\mathcal{M}^\theta_{L_x}(g-2(p+1),-;0,Q) \lra \mathcal{M}^\theta_{L_x}(g-2p,-;1,Q')$$
of type $\alpha$.
\end{enumerate}
\end{prop}

When we take $p=0$ in each of these cases there is a commutative diagram analogous to the outer square of \eqref{prop:FibResOfPairs} (indeed, \eqref{prop:FibResOfPairs} shows the case (\ref{it:FibResOfPairs:ResAlpha})). Considered as a map from the top pair of spaces to the bottom pair, we shall call each of the maps of pairs arising in this way an \emph{approximate augmentation map}.

\section{$k$-triviality, stabilisation of $\pi_0$, and stability ranges}\label{sec:kTriviality}

This is the first section of the paper in which we must directly confront properties of $\theta$-surfaces, and it is rather technical. To aid the reader we first briefly outline what we are trying to achieve, in the case of orientable surface (though we will also treat non-orientable surfaces).

In the course of the proof of homological stability, we would like to know that all approximate augmentation maps induce the zero map on homology in a range of degrees. Failing this, we would like to know that all sufficiently long compositions of approximate augmentation maps (which we defined at the end of Section \ref{sec:Resolutions}) induce the zero map on homology in a range of degrees. In this section we will define the notion of \emph{$k$-triviality} of a tangential structure $\theta$, for a natural number $k$. Later (in Section \ref{sec:ZeroInHomology}) we will show that if $\theta$ satisfies the property of $k$-triviality then certain compositions of $k$ approximate augmentation maps induce the zero map on homology in a certain range of degrees.

In order to begin the inductive proof of homological stability for the spaces $\mathcal{M}^\theta(g,+;{P})$ we will need to know that their zeroth homology, or equivalently their sets of path components, eventually stabilise. In Section \ref{sec:Pi0StabDefn} we define the notion of $\theta$ \emph{stabilising on $\pi_0$ at genus $h$}, which encodes at which genus (we know that) the path components of moduli spaces of orientable $\theta$-surfaces stabilise.

In all, we associate two invariants, $h \in [0,\infty]$ and $k \in [1,\infty]$, to a tangential structure $\theta$. In practice, the number $h$ is more readily computable, so in Section \ref{sec:FormalKTriviality} we show how to estimate the more complicated $k$ in terms of $h$. We then calculate these two invariants in the basic examples of interest: the trivial tangential structure, orientations, and either of these equipped with maps to a simply connected background space.

The stability theorem for orientable surfaces, that is, for the spaces $\mathcal{M}^\theta(g, +;{P})$, will be expressed in terms of certain functions $F, G : \bZ \to \bZ$ which will describe the stability range for $\alpha$ and $\beta$ type maps respectively. These functions depend on the parameters $h$ and $k$, and in Section \ref{sec:StabRangeOrientable} we describe a procedure for constructing them (as well as certain auxiliary functions) from these parameters.

\subsection{$k$-triviality}\label{sec:kTrivialityDefn}

Consider the subspace
$$\bB := \bI^\infty+2\cdot e_1 \subset \bA \subset \bR^\infty,$$
where our convention is that $e_1$ denotes the first basis vector in $\bR^\infty$. Let us say that a 1-dimensional $\theta$-manifold $P \subset \bR^\infty$ is \emph{standard on $\bB$} if it agrees near $\bB$ with the 1-dimensional $\theta$-manifold $\bR \times \{\pm \tfrac{1}{2}\} \times \{0\}^{\infty-2}$. Similarly, say a 2-dimensional $\theta$-cobordism $W \subset [t,t'] \times \bR^\infty$ is \emph{standard on $\bB$} if it agrees near $[t,t'] \times \bB$ with the  $\theta$-manifold $[t,t'] \times(\bR \times \{\pm \tfrac{1}{2}\} \times \{0\}^{\infty-2})$. We will continue to use these terms for manifolds which are only defined inside $\bA$. We also define the 1-dimensional $\theta$-manifold
$$T := \bB \cap (\bR \times \{\pm \tfrac{1}{2}\} \times \{0\}^{\infty-2}).$$

Let us write $\mathcal{C}_{\theta, T} \subset \mathcal{C}_\theta$ for the subcategory of the cobordism category with objects those $P$ such that $P$ is standard on $\bB$, and with morphisms those $(t,W) : P \leadsto P'$ which are standard on $\bB$. 

If $L \subset \bA$ is an outer boundary condition which is standard on $\bB$, and $Q \subset \bI^\infty$ is an inner boundary condition, then $L_Q$ is standard on $\bB$ and so $L_{Q} \in \mathcal{C}_{\theta, T}$.

\begin{defn}\label{defn:absorbs}
Let $W : (0, Q) \leadsto (1,Q')$ be an inner cobordism, and $U : (0,L') \leadsto (1,L)$ be an outer cobordism which is standard on $\bB$. There are then morphisms
$$M(W, L^{(\prime)}) :=  (1, W \cup ([0,1] \times L^{(\prime)})) : L^{(\prime)}_{Q} \leadsto L^{(\prime)}_{Q'}$$
and
$$M(Q^{(\prime)}, U) :=  (1,([0,1] \times Q^{(\prime)}) \cup U) : L_{Q^{(\prime)}} \leadsto L'_{Q^{(\prime)}}$$
in $\mathcal{C}_{\theta, T}$. Say that \emph{$U$ absorbs $W$} if there is a morphism
$$(1,Z) : L'_{Q'} \leadsto L_{Q} \in \mathcal{C}_{\theta, T}$$
such that
\begin{enumerate}[(i)]
\item there is a path from $(1,Z) \circ M(W, L')$ to $M(Q,U) \circ (1,[0,1] \times L'_{Q})$ in the space $\mathcal{C}_{\theta, T}(L'_{Q},L_{Q})$,

\item there is a path from $M(W, L) \circ (1,Z)$ to $(1,[0,1] \times L_{Q'}) \circ M(Q',U)$ in the space $\mathcal{C}_{\theta, T}(L'_{Q'},L_{Q'})$.
\end{enumerate}
\end{defn}

At first sight this may seem like a difficult condition to check, but notice that the spaces of morphisms $\mathcal{C}_{\theta, T}(P, P')$ have the homotopy type of $\mathcal{N}^\theta(R)$, where $R$ is the 1-manifold 
$$(\{0\} \times (P \setminus T)) \cup ([0,1] \times \{\pm 1\} \times \{\pm \tfrac{1}{2}\}) \cup (\{1\} \times (P' \setminus T))$$
(with its corners unbent) equipped with its induced $\theta$-structure. Hence questions about the set of path components $\pi_0\mathcal{C}_{\theta, T}(P, P')$ are just questions about the isomorphism classification of $\theta$-surfaces.

\begin{defn}\label{defn:kTriv}
We call a pair of an elementary stabilisation map ${W} : (0,{Q}_{\sticks}) \leadsto (1,{Q}'_{\cross})$ and an outer cobordism $U : (0,L') \leadsto (1,L)$ which is standard on $\bB$ an \emph{orientable test pair of height $k$} if $U$ has $k$ 1-handles relative to $L$, attached via the restrictions of disjoint embeddings $b_1, \ldots, b_k : \{\pm 1\} \times \bR \hookrightarrow L$ to $\{ \pm1\} \times \bI$, and this data has one of the two combinatorial forms shown in Figure \ref{fig:kTRivOr}. ($L$ may contain several components, but the $b_i$ should only map to those which intersect $\bB$, as shown.)

\begin{figure}[h]
\centering
\includegraphics[bb = 0 0 327 114]{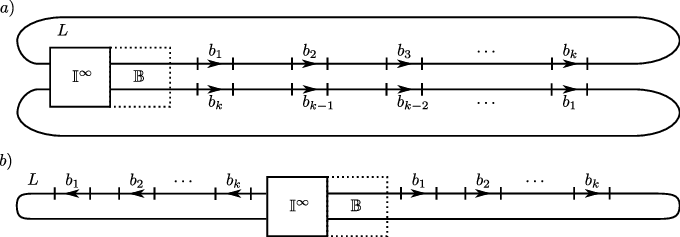}
\caption{}\label{fig:kTRivOr}
\end{figure}

Similarly, we say $(W,U)$ is a \emph{non-orientable test pair of height $k$} if if $U$ has $k$ 1-handles relative to $L$, attached via restrictions of disjoint embeddings $b_1, \ldots, b_k : \{\pm 1\} \times \bR \hookrightarrow L$ to $\{\pm1 \} \times \bI$, and this data has the combinatorial form shown in Figure \ref{fig:kTRivNO}.

\begin{figure}[h]
\centering
\includegraphics[bb = 0 0 319 67]{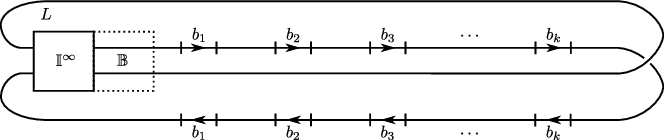}
\caption{}\label{fig:kTRivNO}
\end{figure}

We say that a tangential structure $\theta$ is \emph{$k$-trivial} if for every orientable test pair $(W, U)$ of height $k$, $U$ absorbs $W$. Similarly, we say that a tangential structure $\theta$ is \emph{$k'$-trivial for projective planes} if for every non-orientable test pair $(W, U)$ of height $k'$, $U$ absorbs $W$.
\end{defn}

The combinatorial forms we have singled out in this definition have the following motivation: attaching 1-handles to $L$ along these intervals gives the orientable (in Figure \ref{fig:kTRivOr}) or non-orientable (in Figure \ref{fig:kTRivNO}) surface with the largest possible genus. In Figure \ref{fig:kTRivOr} a) or b) the surface obtained by attaching such 1-handles is connected and orientable, and has 2 boundary components and genus $\tfrac{k-2}{2}$ if $k$ is even, and 1 boundary component and genus $\tfrac{k-1}{2}$ if $k$ is odd. In Figure \ref{fig:kTRivNO} the surface obtained by attaching such 1-handles consists of the disjoint union of a disc and a non-orientable surface with 1 boundary component and genus $k$.

\subsection{Stabilisation of $\pi_0$}\label{sec:Pi0StabDefn}

A necessary condition for a family of spaces to exhibit homological stability is that their sets of path components (or equivalently, their zeroth homology) stabilises. This will be one of the two requirements of our stability theorem: the other is that the tangential structure should be $k$-trivial for some $k$. In order to not have to distinguish cases, we introduce the following piece of \emph{ad hoc} notation: if ${P}$ is a 1-manifold with $\theta$-structure consisting of a pair of circles, we let
$$\mathcal{M}^\theta(-1,+;{P}) \subset \mathcal{N}^\theta(P)$$
be the subspace of those $\theta$-surfaces which are diffeomorphic to a pair of discs. Then if ${W} : (t,{P}) \leadsto (t',{P}')$ is a $\theta$-cobordism which has a single 1-handle relative to $P$ which joins the two components, there is an induced map
$$W_* : \mathcal{M}^\theta(-1,+;{P}) \lra \mathcal{M}^\theta(0,+;{P}')$$
which we shall consider as being of type $\alpha$.

\begin{defn}\label{defn:StabilisingO}
For an integer $h \geq 0$, say that a tangential structure $\theta$ \textit{stabilises for orientable surfaces at genus $h$} if all stabilisation maps
$$\mathcal{M}^\theta(g,+;P) \lra \mathcal{M}^\theta(g+1,+;P')$$
of type $\alpha$, and all stabilisation maps
$$\mathcal{M}^\theta(g,+;P) \lra \mathcal{M}^\theta(g,+;P')$$
of type $\beta$, are bijections for all $g \geq h$ and surjections for all $g \geq h-1$.
\end{defn}

In the non-orientable case, we must also introduce a piece of \emph{ad hoc} notation. If, and only if, $P$ consists of a single circle, let us write $\mathcal{M}^\theta(0,-;P) := \mathcal{M}^\theta(0,+;P)$ for the moduli space of discs with boundary $P$ (despite the fact that a disc is not non-orientable).

\begin{defn}\label{defn:StabilisingNO}
For an integer $h' \geq 2$, say a tangential structure $\theta$ \textit{stabilises at genus $h'$ for projective planes} if all stabilisation maps of type $\mu$
\begin{equation}\label{eq:MuStab}
\mathcal{M}^\theta(g,-;P) \lra \mathcal{M}^\theta(g+1,-;P')
\end{equation}
are bijections for all $g \geq h'$ and surjections for all $g \geq h'-1$. Say $\theta$ \textit{stabilises at genus $1$ for projective planes} if all stabilisation maps of type $\mu$
$$\mathcal{M}^\theta(0,-;P) \lra \mathcal{M}^\theta(1,-;P')$$
are surjective (where, by our convention above, $P$ is a single circle), and all stabilisation maps of type $\mu$ as in \eqref{eq:MuStab} are bijective for $g \geq 1$ (in this case for arbitrary $P$).

Finally, for an integer $h \geq 0$, say that a tangential structure $\theta$ \textit{stabilises for non-orientable surfaces at genus $h$} if all stabilisation maps
$$\mathcal{M}^\theta(g,-;P) \lra \mathcal{M}^\theta(g+2,-;P')$$
of type $\alpha$, and all stabilisation maps
$$\mathcal{M}^\theta(g,-;P) \lra \mathcal{M}^\theta(g,-;P')$$
of type $\beta$, are surjections for all $g \geq h$. 
\end{defn}

\subsection{Formal $k$-triviality}\label{sec:FormalKTriviality}

In this section we shall show that for a tangential structure $\theta$, exhibiting stability for $\pi_0$ is enough to ensure that it is $k$-trivial for some $k$. This method often delivers a non-optimal $k$, and as the slope of the stability range we will produce depends on $k$, this then gives a non-optimal stability range. However, for some purposes it is enough to know merely the existence of a stability range --- for example, to apply the methods of \cite{GMTW} to identify the stable homology of $\mathcal{M}^\theta(F; \ell_{\partial F})$.

\begin{prop}\label{prop:FormalkTriv}
If $\theta$ stabilises for orientable surfaces at genus $h \geq 0$ then it is $(2h+1)$-trivial.
\end{prop}

\begin{proof}
Let $({W}, U)$ be an orientable test pair of height $(2h+1)$. We have elements
$$M(Q,U) \circ (1,[0,1] \times L'_{Q})    \in \mathcal{C}_{\theta,T}(L'_{Q}, L_{Q})$$
and
$$(1,[0,1] \times L_{Q'})\circ M(Q',U)  \in \mathcal{C}_{\theta,T}(L'_{Q'}, L_{Q'}),$$
and we are searching for a morphism $(1,Z) \in \mathcal{C}_{\theta,T}(L'_{Q'}, L_{Q})$ such that
$$[(1,Z) \circ M(W,L')] = [M(Q,U) \circ (1,[0,1] \times L'_{Q})] \in \pi_0\mathcal{C}_{\theta,T}(L'_{Q}, L_{Q})$$
and
$$[M(W, L) \circ (1,Z)] = [(1,[0,1] \times L_{Q'})\circ M(Q',U)] \in \pi_0\mathcal{C}_{\theta,T}(L'_{Q'}, L_{Q'}).$$
There is a commutative square
\begin{equation}\label{eq:FormalKTrivSq}
\begin{gathered}
\xymatrix{
\mathcal{C}_{\theta,T}(L'_{Q'}, L_{Q}) \ar[d]^-{M(W, L) \circ -} \ar[rr]^-{- \circ M(W,L')} & & \mathcal{C}_{\theta,T}(L'_{Q}, L_{Q}) \ar[d]\ar[d]^-{M(W, L) \circ -}\\
\mathcal{C}_{\theta,T}(L'_{Q'}, L_{Q'}) \ar[rr]^-{- \circ M(W,L')} & & \mathcal{C}_{\theta,T}(L'_{Q}, L_{Q'}),
}
\end{gathered}
\end{equation}
and we are searching for a path component of the top left-hand corner that maps to certain path components under the two maps out of this space.

To address this problem, let us note that in order to consider surfaces which are standard on $\bB$, we may as well cut out the interior of $[0,t] \times T$ and consider surfaces with corners and prescribed boundary conditions for the $\theta$-structure. We may then unbend the corners by gluing on a suitable cobordism.

From this point of view, the commutative square \eqref{eq:FormalKTrivSq} may be replaced (when restricted to the path-components consisting of those surfaces of the correct topological type) by a homotopy-commutative square consisting of moduli spaces $\mathcal{M}^\theta(g,+;P)$ for certain genera $g$ and boundary conditions $P$. The pattern of handle attachments in Figure \ref{fig:kTRivOr} a) or b) shows that the manifolds $M(Q,U)$ and $M(Q',U)$ are both abstractly the disjoint union of a genus $h$ surface with 3 boundary components and a collection of cylinders. The cylinders play no role in this discussion, and when we remove the interior of $[0,1] \times T$ from the genus $h$ component we obtain a surface of genus $h$ with a single boundary component. Hence those path components of the spaces in the commutative square \eqref{eq:FormalKTrivSq} which are relevant for this problem give a commutative square homotopy equivalent to the following,
\begin{equation*}
\xymatrix{
\mathcal{M}^\theta(h-1,+;\bigcirc \bigcirc) \ar[d]^-{\text{type } \alpha} \ar[r]^-{\text{type } \alpha}& \mathcal{M}^\theta(h,+;\bigcirc) \ar[d]\ar[d]^-{\text{type } \beta}\\
\mathcal{M}^\theta(h,+; \bigcirc) \ar[r]^-{\text{type } \beta}& \mathcal{M}^\theta(h,+;\bigcirc \bigcirc),
}
\end{equation*}
where we have just indicated the number of boundary components of each surface (the precise $\theta$-structure on them does not matter for what follows, though it is fixed), and the type of elementary stabilisation map that each of the four maps gives. (If $h=0$ then the top left corner must be interpreted using the convention introduced in Section \ref{sec:Pi0StabDefn}.)

As we have supposed that $\theta$ stabilises for orientable surfaces at genus $h$, the map on path components induced by the top map
$$\pi_0(\text{type } \alpha) : \pi_0\mathcal{M}^\theta(h-1,+;\bigcirc \bigcirc) \lra \pi_0\mathcal{M}^\theta(h,+;\bigcirc)$$
is surjective, and the map on path components induced by the bottom map
$$\pi_0(\text{type } \beta) : \pi_0\mathcal{M}^\theta(h,+;\bigcirc) \lra \pi_0\mathcal{M}^\theta(h,+;\bigcirc \bigcirc)$$
is injective. Hence the top and bottom maps in \eqref{eq:FormalKTrivSq} also have these properties.

Thus we may choose $[(1,Z)] \in \pi_0\mathcal{C}_{\theta,T}(L'_{Q'}, L_{Q})$ so that it maps to $[M(Q,U) \circ (1,[0,1] \times L'_{Q})]$ under $- \circ M(W,L')$. Then using a path from
$$M(W,L) \circ M(Q,U) \circ (1, [0,1] \times L'_Q) \text{ to } (1, [0,1] \times L_{Q'}) \circ M(Q', U) \circ M(W, L')$$
by stretching, we have that
$$[M(W,L) \circ (1,Z) \circ M(W,L')] = [(1, [0,1] \times L_{Q'}) \circ M(Q', U) \circ M(W, L')]$$
but by injectivity of $-\circ M(W,L')$ in genus $h$ it follows that
$$[M(W,L) \circ (1,Z)] = [(1,[0,1] \times L_{Q'})\circ M(Q',U)] \in \pi_0\mathcal{C}_{\theta,T}(L'_{Q'}, L_{Q'})$$
as required.
\end{proof}

There is an analogous statement for stabilisation by projective planes, proved in a similar way.

\begin{prop}\label{prop:FormalkTrivNO}\mbox{}
Suppose $\theta$ stabilises for projective planes at genus $h' \geq 1$. Then it is $h'$-trivial for projective planes.
\end{prop}
\begin{proof}
The proof is largely the same as the last case. If we are trying to show $h'$-triviality for $h' \geq 2$ the relevant part of the diagram \eqref{eq:FormalKTrivSq} is now homotopy equivalent to
\begin{equation*}
\xymatrix{
\mathcal{M}^\theta(h'-1,-;\bigcirc) \ar[d]^-{\text{type } \mu} \ar[r]^-{\text{type } \mu}& \mathcal{M}^\theta(h',-;\bigcirc) \ar[d]\ar[d]^-{\text{type } \mu}\\
\mathcal{M}^\theta(h',-; \bigcirc) \ar[r]^-{\text{type } \mu}& \mathcal{M}^\theta(h'+1,-;\bigcirc ).
}
\end{equation*}
(In fact, after removing $[0,1] \times T$ the surface $Z$ we are trying to find decomposes as a disc and a non-orientable surface of genus $(h'-1)$ with one boundary; the disc plays no role, and the above commutative square is used to determine the isomorphism type of the remaining $\theta$-surface.)

For the argument to go through, we require the top map to induce a surjection on $\pi_0$, and the bottom map to induce an injection on $\pi_0$, but this follows immediately from stabilisation for projective planes at genus $h'$. 

In the case $h'=1$ we repeat the argument, but the top right-hand corner is now $\mathcal{M}^\theta(0,-;\bigcirc)=\mathcal{M}^\theta(0,+;\bigcirc)$ as defined using the convention introduced just before Definition \ref{defn:StabilisingNO}, and the surjectivity of the top map on $\pi_0$ follows from the definition of stabilising at genus $1$ for projective planes.
\end{proof}

\subsection{Examples}\label{sec:kTrivialityExamples}

The two main examples we will discuss in this paper are oriented surfaces and non-orientable surfaces, as well as these with maps to a simply connected background space.

\begin{prop}\label{prop:OrientedUnoriented1Triviality}
The tangential structures given by the maps $BO(2) \to BO(2)$ and $BSO(2) \to BO(2)$ are both 1-trivial, and they both stabilise at genus 0 for both orientable and non-orientable surfaces. Furthermore, $BO(2) \to BO(2)$ is 1-trivial for projective planes and stabilises for projective planes at genus $1$.
\end{prop}
\begin{proof}
If we show that both structures stabilise at genus 0, then the results of the previous section show that they are also 1-trivial. If we show that $BO(2) \to BO(2)$ stabilises for projective planes at genus $1$, then the results of the previous section show that it is also 1-trivial for projective planes.

To do this, we show that the spaces $\mathcal{M}^\theta(g,+;{P})$ for $g \geq 0$ and $\mathcal{M}^\theta(g,-;{P})$ for $g \geq 1$ have a single path component, using the description of Section \ref{sec:PathComponents}. We choose a surface $F$ of topological type $(g,\pm)$ with boundary $P$ (such that the orientation of ${P}$ extends to $F$, in the orientable case) so that
$$\mathcal{M}^\theta(g,\pm;P) \simeq \Bun_\partial(TF, \theta^*\gamma_2; \ell_P) \hcoker \Diff_\partial(F).$$
For both choices of $\theta$ the space $\Bun_\partial(TF, \theta^*\gamma_2;\ell_P)$ is contractible, so the Borel construction of any group acting on it is path connected. For $BO(2) \to BO(2)$ it is contractible as this is the universal property of the universal bundle $\gamma_2 \to BO(2)$. For $BSO(2) \to BO(2)$ we use that a connected surface with boundary has a unique orientation compatible with a given one on the boundary.
\end{proof}

Given a tangential structure $\theta : \X \to BO(2)$, we may consider the new tangential structure $\theta \times Y := \theta \circ \pi_B : \X \times Y \to BO(2)$ for a space $Y$. 

\begin{prop}\label{prop:AddingBackgroundSpace}
Let $\theta : \X \to BO(2)$ be a tangential structure which is $k$-trivial and stabilises at genus $h$. Let $Y$ be a simply connected space. Then $\theta \times Y  : \X \times Y \to BO(2)$ is $k$-trivial and stabilises at genus $h$.

Similarly, suppose $\theta$ is $k'$-trivial for projective planes and stabilises at genus $h'$ for projective planes. Then $\theta \times Y$ is $k'$-trivial and stabilises at genus $h'$ for projective planes.
\end{prop}
\begin{proof}
Note that for a surface $X$, a $\theta \times Y$-structure on $X$ is a pair of a $\theta$-structure $\ell_X$ on $X$ and a continuous map $f_X : X \to Y$. Let us agree to write $\theta \times Y$-structures as pairs $(\ell_X, f_X)$. As a simplifying preliminary, note that if $P$ is a 1-manifold then as $Y$ is simply-connected any $(\ell_P, f_P)$ can be changed by a quasi-isomorphism in $\mathcal{C}_{\theta \times Y}$ so that $f_P$ is the constant map $c_{y_0}$ to a basepoint $y_0 \in Y$. Hence it is enough to only consider boundary conditions of this form.

Let us first show that $\theta \times Y$ stabilises on $\pi_0$ at genus $h$. There is a fibre sequence
\begin{equation}\label{eq:AddingBGSpaceSequence}
\map(X, \partial X ; Y, *) \lra \mathcal{M}^{\theta \times Y}(g, \pm;(Q,\ell_Q \times c_{y_0})) \lra \mathcal{M}^{\theta}(g, \pm;(Q, \ell_Q))
\end{equation}
where $X$ is a connected surface of orientation type $\pm$ and genus $g$. This fibration has a section, by giving each surface the constant map to the basepoint $y_0$. Thus the preimage of $[\xi]$ under the surjection
$$\pi_0(\mathcal{M}^{\theta \times Y}(g, \pm;(Q,\ell_Q \times c_{y_0})) \lra \pi_0(\mathcal{M}^{\theta}(g, \pm;(Q, \ell_Q))$$
is in natural bijection with $\pi_0(\map(X, \partial X ; Y, *))$.

In the non-orientable case there is a natural bijection $\pi_0(\map(X, \partial X ; Y, *)) \cong H_2(Y, \{y_0\} ; \bF_2)$, given by sending a map to the image of its $\bF_2$ fundamental class. This extends to a map
$$\pi_0(\mathcal{M}^{\theta \times Y}(g, \pm;(Q,\ell_Q \times c_{y_0})) \lra H_2(Y, \{y_0\} ; \bF_2),$$
which together with the surjection above gives a bijection
$$\pi_0(\mathcal{M}^{\theta \times Y}(g, -;(Q,\ell_Q \times c_{y_0})) \lra \pi_0(\mathcal{M}^{\theta}(g, -;(Q, \ell_Q)) \times H_2(Y, \{y_0\} ; \bF_2).$$
This bijection is natural for stabilisation maps, so if $\theta$ stabilises at genus $h$, then so does $\theta \times Y$.

Similarly, in the orientable case there is a natural bijection $\pi_0(\map(X, \partial X ; Y, *)) \cong H_2(Y, \{y_0\} ; \bZ)$, given by sending a map to the image of its $\bZ$ fundamental class, which gives a bijection
$$\pi_0(\mathcal{M}^{\theta \times Y}(g, +;(Q,\ell_Q \times c_{y_0})) \lra \pi_0(\mathcal{M}^{\theta}(g, +;(Q, \ell_Q)) \times H_2(Y, \{y_0\} ; \bZ).$$
This bijection is natural for stabilisation maps, so if $\theta$ stabilises at genus $h$, then so does $\theta \times Y$.

There is one special case which must be treated carefully: if $h'=1$ then we must consider stabilising a genus zero surface by attaching a M\"{o}bius band, which corresponds to a map $H_2(Y;\bZ) \to H_2(Y;\bF_2)$ given by reduction modulo 2 followed by addition of some element. This will not be an isomorphism, but is an epimorphism, which is all that is required in this case.

\vspace{2ex}

Let us now show that if $\theta$ is $k$-trivial then so is $\theta \times Y$. Let us be given an orientable test pair $({W},U)$ of height $k$, all equipped with $\theta\times Y$-structures, which we call $\ell_W \times f_W$ and $\ell_U \times f_U$. As $\theta$ is $k$-trivial, we can find a $\theta$-cobordism $(1,(Z, \ell_Z)) \in \mathcal{C}_{\theta,T}(L'_{Q'}, L_Q)$ which exhibits $(U, \ell_U)$ as absorbing $(W,\ell_W)$. Let $f_Z : (Z, ([0,1] \times T) \cup \partial Z) \to (Y, y_0)$ be a continuous map, which is given by a collection of classes $(z_1, \ldots, z_i) \in H_2(Y, \{y_0\};\bZ)$, one for each component of $Z \setminus ([0,1] \times T)$. To check whether $(Z, \ell_Z \times f_Z)$ exhibits $(U, \ell_U \times f_U)$ as absorbing $(W, \ell_W \times f_W)$, knowing that it does so as a $\theta$-manifold, then by the calculation of $\pi_0(\mathcal{M}^{\theta \times Y}(g, +;(Q,\ell_Q \times c_{y_0}))$ above we just have to check that an equation among homology classes are satisfied, namely
$$z_1 + \cdots + z_i + (f_W)_*([W]) = (f_U)_*([U]).$$
If this does not hold, we may simply re-choose the homology class $z_1$ to ensure that it does (this corresponds to changing the map $f_Z$ on one path component by adding on a class in $\pi_2(Y, y_0)$).

A similar argument with $\bF_2$-homology shows that if $\theta$ is $k'$-trivial for projective planes, so is $\theta \times Y$.
\end{proof}

\subsection{A non-example}

The following non-example was described by Galatius and the author in \cite[\S 5.2]{GR-W}. Let $\theta : BSO(2) \to BO(2)$ be the tangential structiure corresponding to orientations, and consider $\theta \times B\bZ/2 : BSO(2) \times B\bZ/2 \to BO(2)$, that is, the tangential structure given by an orientation and a map to $B\bZ/2$. Using the Borel construction model, and taking boundary condition $\ell_{\partial \Sigma_{g,1}} \times c_*$ having constant map to the basepoint $* \in B\bZ/2$, we identify
$$\pi_0(\mathcal{M}^\theta(\Sigma_{g,1};\ell_{\partial \Sigma_{g,1}} \times c_*)) = H^1(\Sigma_{g,1}, \partial \Sigma_{g,1};\bZ/2)/\Gamma_{g,1},$$
which by Poincar{\'e} duality may be identified with $H_1(\Sigma_{g,1};\bZ/2)/\Gamma_{g,1}$. Using the formulas of \cite[Lemma 5.1]{GR-W} it follows that this orbit set has precisely two elements as long as $g \geq 1$: the orbit consisting of just $0 \in H_1(\Sigma_{g,1};\bZ/2)$, and an orbit consisting of all non-zero elements.

However, even though this moduli space always has two components (for $g \geq 1$), this tangential structure never stabilises on $\pi_0$. It follows from the formulas in \cite[Lemma 5.1]{GR-W} that there exists an $x \in \pi_0(\mathcal{M}^\theta(\Sigma_{1,2};\ell_{\partial \Sigma_{1,2}} \times c_*))$ such that the map
$$\pi_0(\mathcal{M}^\theta(\Sigma_{g,1};\ell_{\partial \Sigma_{g,1}} \times c_*)) \lra \pi_0(\mathcal{M}^\theta(\Sigma_{g+1,1};\ell_{\partial \Sigma_{g+1,1}} \times c_*)),$$
given by gluing with $x$ lands entirely in the orbit of non-zero elements: in particular, this stabilisation map is not surjective. Factorising $x$ into a map of type $\beta$ followed by a map of type $\alpha$, it follows that this tangential structure never stabilises on $\pi_0$.

Despite this faliure of stability, in \cite[\S 5.2]{GR-W} Galatius and the author still compute the ``stable homology" of these moduli spaces, suitably interpreted.

\subsection{Stability range for orientable surfaces}\label{sec:StabRangeOrientable}

Let $k \geq 1$ and $h \geq 0$ be integers. In this section we will describe four functions $F, G, X, Y : \bZ \to \bZ$ which depend on the integers $h$ and $k$. When $\theta$ is a tangential structure which is $k$-trivial and stabilises on $\pi_0$ at genus $h$, the functions $F$ and $G$ described below will occur in the statement of the stability theorem for $\theta$ (Theorem \ref{thm:StabOrientableSurfaces}), and the functions $X$ and $Y$ will occur in the proof of that theorem.

\begin{defn}\label{defn:Orientable:Range}
Let $F, G, X, Y : \bZ \to \bZ$ be defined to be
$$F(g)=G(g)=X(g)=Y(g)=-1 \text{ for } g \leq h-2$$
and satisfy
$$F(h-1)=G(h-1)=X(h-1)=Y(h-1)=0.$$
For values $g \geq h$ we define these functions recursively as follows (where we use the notation $n \vee 0 = \max(n,0)$).
\begin{enumerate}[(i)]
\item\label{it:Orientable:XY} Let
$$
X(g) = \min \begin{cases}
F(g-1)+1\\
G(g)+1\\
X(g-1)+1\\
Y(g-1)+1\\
0 \text{ if } g \leq 0
\end{cases}\quad\quad
Y(g) = \min \begin{cases}
F(g-1)+1\\
G(g-1)+1\\
X(g-2)+1\\
Y(g-1)+1\\
0 \text{ if } g \leq 1.
\end{cases}
$$

\item\label{it:Orientable:kIs1} If $k=1$ let
\begin{eqnarray*}
F(g) = \min \begin{cases}
F(g-1)+1\\
X(g)
\end{cases}\quad\quad
G(g) = \min \begin{cases}
G(g-1)+1\\
Y(g).
\end{cases}
\end{eqnarray*}

\item\label{it:Orientable:kIsEven} If $k = 2l$ with $l > 0$ let
\begin{eqnarray*}
F(g) = \min \begin{cases}
F(g -1)+1 \\
G(g - l)+1\\
X(g+1-l) \vee 0\\
Y(g+1-l) \vee 0\\
\end{cases}\quad
G(g) = \min \begin{cases}
F(g - l-1)+1\\
G(g -1)+1 \\
X(g - l) \vee 0\\
Y(g+1-l) \vee 0.
\end{cases}
\end{eqnarray*}

\item\label{it:Orientable:kIsOdd} If $k = 2l+1$ with $l > 0$ let
\begin{eqnarray*}
F(g) = \min \begin{cases}
F(g - l-1)+1\\
X(g-l) \vee 0\\
Y(g+1-l) \vee 0\\
\end{cases}\quad
G(g) = \min \begin{cases}
G(g - l-1)+1 \\
X(g - l) \vee 0\\
Y(g-l) \vee 0.
\end{cases}
\end{eqnarray*}
\end{enumerate}
\end{defn}

By changing the equalities to $\leq$ in the above definition, and for each of the functions $F$, $G$, $X$, and $Y$ trying the ansatz $\lfloor\tfrac{a \cdot g + b}{c}\rfloor$ with $a,b,c \in \bZ$, we obtain the following lower bounds, assuming that $h > 0$. (If $h=0$ these must be slightly modified, to account for the bottom condition in the definition of $X(g)$ and $Y(g)$: we leave such a modification to the reader.) In fact, it is easy but laborious to check that these lower bounds are in fact all equalities, by considering a minimal $g$ for which one of them is not, and deriving a contradiction.

\begin{enumerate}[(i)]
\item If $k=1$ and $h > 0$, then 
$$F(g),  X(g) \geq \left\lfloor \tfrac{2g-2h+3}{3}\right\rfloor \quad\quad G(g), Y(g) \geq \left\lfloor \tfrac{2g-2h+2}{3}\right\rfloor.$$

\item If $k=2l$ with $l>0$, and $h > 0$, then
\begin{align*}
F(g) &\geq \left\lfloor \tfrac{2g-2h+2}{2l+1}\right\rfloor \quad & X(g) &\geq \left\lfloor \tfrac{2g-2h+2l+1}{2l+1}\right\rfloor\\
G(g) &\geq \left\lfloor \tfrac{2g-2h+1}{2l+1}\right\rfloor \quad & Y(g) &\geq \left\lfloor \tfrac{2g-2h+2l}{2l+1}\right\rfloor.
\end{align*}

\item If $k=2l+1$ with $l>0$, and $h > 0$, then
\begin{align*}
F(g) &\geq \left\lfloor \tfrac{g-h+1}{l+1}\right\rfloor \quad & X(g) &\geq \left\lfloor \tfrac{g-h+l+1}{l+1}\right\rfloor\\
G(g) &\geq \left\lfloor \tfrac{g-h}{l+1}\right\rfloor \quad & Y(g) &\geq \left\lfloor \tfrac{g-h+l}{l+1}\right\rfloor.
\end{align*}
\end{enumerate}

\subsection{Stability range for non-orientable surfaces}\label{sec:StabRangeNonorientable}

Similarly to the last section, let $k' \geq 1$ and $h' \geq 1$ be integers, and let us define functions $H', Z' : \bZ \to \bZ$.

\begin{defn}\label{defn:Nonorientable:Range}
Let $H', Z' : \bZ \to \bZ$ be defined to be
$$H'(g)=Z'(g)=-1 \text{ for }  g \leq h'-2$$
and satisfy
$$H'(h'-1)=Z'(h'-1)=0.$$
For values $g \geq h'$ we define these functions as follows.
\begin{eqnarray*}
Z'(g) = \min \begin{cases}
H'(g-2)+1\\
\lfloor\tfrac{g}{3}\rfloor
\end{cases}
\quad\quad\quad
H'(g) = \min \begin{cases}
Z'(g-k'+1) \vee 0\\
H'(g-k'-1)+1.
\end{cases}
\end{eqnarray*}
\end{defn}

By changing the equalities to $\leq$ in the above definition, and for each of the functions $H'$ and $Z'$ trying the ansatz $\lfloor\tfrac{a \cdot g + b}{c}\rfloor$ with $a,b,c \in \bZ$, we obtain the following lower bounds.

\begin{enumerate}[(i)]
\item If $k'=1$, then
$$H'(g), Z'(g) \geq \left\lfloor \tfrac{g-h'+1}{3}\right\rfloor.$$
This is an equality if $h'=1$, but for $h' > 1$ better lower bounds can be obtained if we do not insist on (floors of) linear functions: for example, one may show that $H'(g),Z'(g)  \geq \lfloor \tfrac{g-h'+1}{2} \rfloor$ for $h' \leq g \leq 3h'$ and $H'(g),Z'(g) \geq \lfloor \tfrac{g}{3} \rfloor$ for $3h' \leq g$, which for $h'$ large is better than the range given above.

\item If $k' > 1$ and $h' > 1$, then
$$H'(g) \geq \left\lfloor \tfrac{g-h'+1}{k'+1}\right\rfloor \quad\quad Z'(g) \geq \left\lfloor \tfrac{g-h'+k'}{k'+1}\right\rfloor.$$
Again, a laborious check will show that these lower bounds are in fact equalities.
\end{enumerate}

\section{The stability theorems}\label{sec:StabilityTheorems}

We will now give the full statements of the quantitative homological stability theorems for moduli spaces of surfaces with $\theta$-structure. As always, there are slight differences between the orientable and non-orientable cases, so we have two statements.

\begin{thm}\label{thm:StabOrientableSurfaces}
Let $\theta : \X \to BO(2)$ be a tangential structure such that $\theta^*\gamma_2$ is orientable, which stabilises on $\pi_0$ for orientable surfaces at genus $h$ and is $k$-trivial. If $F$ and $G$ are the functions given in Definition \ref{defn:Orientable:Range}, then:
\begin{enumerate}[(i)]
	\item\label{it:OStabAlpha} Any stabilisation map $W_* : \mathcal{M}^\theta(g,+;{P}) \to \mathcal{M}^\theta(g+1,+;{P}')$ of type $\alpha$ induces an epimorphism in homology in degrees $* \leq F(g)$ and an isomorphism in homology in degrees $* \leq F(g)-1$.
	\item\label{it:OStabBeta} Any stabilisation map $W_* : \mathcal{M}^\theta(g,+;{P}) \to \mathcal{M}^\theta(g,+;{P}')$ of type $\beta$ induces an epimorphism in homology in degrees $* \leq G(g)$ and an isomorphism in homology in degrees $* \leq G(g)-1$.
	\item\label{it:OStabGamma} Any stabilisation map $W_* : \mathcal{M}^\theta(g,+;{P}) \to \mathcal{M}^\theta(g,+;{P}')$ of type $\gamma$ induces an isomorphism in homology in degrees $* \leq G(g)$. It always induces an epimorphism in homology in degrees $* \leq G(g)+1$, and induces an epimorphism in homology in all degrees as long as $P' \neq \emptyset$.
\end{enumerate}
\end{thm}

\begin{thm}\label{thm:StabNonorientableSurfaces}
Let $\theta : \X \to BO(2)$ be a tangential structure which stabilises on $\pi_0$ at genus $h'$ for projective planes and is $k'$-trivial for projective planes. If $H'$ is the function given in Definition \ref{defn:Nonorientable:Range}, then:
\begin{enumerate}[(i)]
\item\label{it:NOStabMu} Any stabilisation map $W_* : \mathcal{M}^\theta(g,-;{P}) \to \mathcal{M}^\theta(g+1,-;{P}')$ of type $\mu$ induces an epimorphism in homology in degrees $* \leq H'(g)$ and an isomorphism in homology in degrees $* \leq H'(g)-1$.

\end{enumerate}
Suppose in addition that $\theta$ stabilises for non-orientable surfaces on $\pi_0$ at some genus. Then:
\begin{enumerate}[(i)]
\setcounter{enumi}{1}
	\item\label{it:NOStabAlpha} Any stabilisation map $W_* : \mathcal{M}^\theta(g,-;{P}) \to \mathcal{M}^\theta(g+2,-;{P}')$ of type $\alpha$ induces an isomorphism in homology in degrees $* \leq H'(g)-1$.

	\item\label{it:NOStabBeta} Any stabilisation map $W_* : \mathcal{M}^\theta(g,-;{P}) \to \mathcal{M}^\theta(g,-;{P}')$ of type $\beta$ induces an isomorphism in homology in degrees $* \leq H'(g)-1$.

	\item\label{it:NOStabGamma} Any stabilisation map $W_* : \mathcal{M}^\theta(g,-;{P}) \to \mathcal{M}^\theta(g,-;{P}')$ of type $\gamma$ induces an isomorphism in homology in degrees $* \leq H'(g)-1$. It always induces an epimorphism in homology in degrees $* \leq H'(g)$, and induces an epimorphism in homology in all degrees as long as $P' \neq \emptyset$.

\end{enumerate}
\end{thm}

\begin{rem}\label{rem:BetaGamma}
If a stabilisation map of type $\beta$ creates a new boundary component whose boundary condition bounds a disc, then it has a right inverse by gluing in that disc, and hence is injective in all degrees on homology. In the case of orientable surfaces this increases the range in which it is an isomorphism by 1. For tangential structures such as $BO(2) \to BO(2)$, $BSO(2) \to BO(2)$, and these along with maps to a simply-connected background space, all boundary conditions bound a disc, and so all maps of type $\beta$ are injective in homology in all degrees. 

Similarly, whenever a stabilisation map of type $\gamma$ is not closing the last boundary component, there is a map of type $\beta$ which is a left inverse to it (this uses in an essential way our convention that the total space $\X$ defining a tangential structure $\theta$ is path connected). Thus, such a map is surjective in all degrees on homology, and is an isomorphism in the same range that the corresponding map of type $\beta$ is. This observation shows that the stability ranges for maps of type $\gamma$ and $P' \neq \emptyset$ in Theorem \ref{thm:StabOrientableSurfaces} or \ref{thm:StabNonorientableSurfaces} follow from the corresponding stability range for maps of type $\beta$. Stability for maps of type $\gamma$ when $P' = \emptyset$ requires a different argument, which we give in Section \ref{sec:LastBoundary}.
\end{rem}

\subsection{Proof of the qualitative stability theorem (Theorem \ref{thm:QualStabThm})}

For simplicity let us consider orientable surfaces. If $H_0(\mathcal{M}^\theta)$ stabilises then so must $\pi_0(\mathcal{M}^\theta)$, so $\theta$ stabilises on $\pi_0$ for orientable surfaces at some genus $h$. By Proposition \ref{prop:FormalkTriv} it is then $k$-trivial for some $k$, so by Theorem \ref{thm:StabOrientableSurfaces} it satisfies homology stability for orientable surfaces in a range of degrees given by the functions $F$ and $G$ of Definition \ref{defn:Orientable:Range}. By the estimates at the end of Section \ref{sec:StabRangeOrientable}, these functions diverge.

\subsection{Oriented surfaces}\label{sec:OrientedSurfaces}
We consider the tangential structure $\theta :BSO(2) \to BO(2)$. By Proposition \ref{prop:OrientedUnoriented1Triviality}, this tangential structure is 1-trivial and stabilises at genus 0.

It is easy to verify that in this case Definition \ref{defn:Orientable:Range} gives $F(g) = X(g) = \lfloor\frac{2g+1}{3}\rfloor$ and $G(g) = Y(g) = \lfloor\frac{2g}{3}\rfloor$. Furthermore, all boundary conditions on orientable surfaces bound a disc, so all maps of type $\beta$ are injective on homology in all degrees. In terms of the notation used in Section \ref{sec:QuantOrientable}: $\alpha(g)_*$ is an epimorphism for $3* \leq 2g+1$ and an isomorphism for $3* \leq 2g-2$; $\beta(g)_*$ is an isomorphism for $3* \leq 2g$ and always a monomorphism; $\gamma(g)_*$ is an isomorphism for $3* \leq 2g$ and an epimorphism in all degrees as long as one is not closing the last boundary component, or in degrees $3* \leq 2g+3$ if closing the last boundary component. 

This stability range coincides with the range recently obtained by Boldsen \cite{Boldsen} for surfaces with boundary, and improves it slightly for closing the last boundary.

\subsection{Non-orientable surfaces}\label{sec:DerivingNonorientableSurfacesRange}

We consider the tangential structure $\theta: BO(2) \to BO(2)$. By Proposition \ref{prop:OrientedUnoriented1Triviality}, this tangential structure is both 1-trivial and 1-trivial for projective planes, stabilises for projective planes at genus 1, and stabilises for non-orientable surfaces at genus 0.

It is easy to see that Definition \ref{defn:Nonorientable:Range} gives $H'(g) = Z'(g) = \lfloor\tfrac{g}{3}\rfloor$ in this case. Furthermore, all boundary conditions for this tangential structure bound a disc so maps of type $\beta$ are injective on homology in all degrees. In terms of the notation used in Section \ref{sec:QuantNonorientable}: $\mu(g)_*$ is an epimorphism for $3* \leq g$ and an isomorphism for $3* \leq g-3$; $\alpha(g)_*$ is an isomorphism for $3* \leq g-3$; $\beta(g)_*$ is an isomorphism for $3* \leq g-3$ and a monomorphism in all degrees; $\gamma(g)_*$ is an isomorphism for $3* \leq g-3$ and an epimorphism in all degrees as long as one is not closing the last boundary component, or in degrees $3* \leq g$ if closing the last boundary component.  

This stability range improves on the previously best known range, due to Wahl \cite{Wahl}, which is of slope $1/4$.

\subsection{Surfaces with maps to a background space}\label{sec:DerivingBGSpaceSurfacesRange}

Let $Y$ be a simply connected space, and consider the tangential structure $\theta : BSO(2)\times Y \to BO(2)$. If we consider the boundary condition on $\Sigma_{g,b}$ where all the boundary is sent to a basepoint in $Y$,  the moduli space we obtain is equivalent to the space $\mathcal{S}_{g,b}(Y)$ introduced by Cohen and Madsen \cite{CM}.

By Propositions \ref{prop:OrientedUnoriented1Triviality} and \ref{prop:AddingBackgroundSpace}, this tangential structure is 1-trivial and stabilises at genus 0, so it has the same stability range as oriented surfaces: $\alpha(g)_*$ is an epimorphism for $3* \leq 2g+1$ and an isomorphism for $3* \leq 2g-2$; $\beta(g)_*$ is an isomorphism for $3* \leq 2g$ and always a monomorphism; $\gamma(g)_*$ is an isomorphism for $3* \leq 2g$ and is an epimorphism in all degrees as long as one is not closing the last boundary component, or in degrees $3* \leq 2g+3$ if closing the last boundary component. 

This stability range slightly improves the range recently obtained by Boldsen \cite{Boldsen} for surfaces with boundary, but crucially also works for closing the last boundary.

\vspace{2ex}

Similarly, we can consider the tangential structure $\theta : BO(2) \times Y \to BO(2)$ for non-orientable surfaces. By Propositions \ref{prop:OrientedUnoriented1Triviality} and \ref{prop:AddingBackgroundSpace} we see that this has the same stability range as Section \ref{sec:DerivingNonorientableSurfacesRange}: $H'(g) = Z'(g) = \lfloor\tfrac{g}{3}\rfloor$.

\section{The consequence of absorption}\label{sec:ZeroInHomology}

We now take the first step in the proof of Theorems \ref{thm:StabOrientableSurfaces} and \ref{thm:StabNonorientableSurfaces}. As promised in Section \ref{sec:kTriviality}, the purpose of the notion of $k$-triviality is to ensure that certain compositions of approximate augmentation maps of length $k$ induce the zero map in homology in a certain range of degrees. This will in fact be a more general property associated to a cobordism $U : L' \leadsto L$ of outer boundary conditions which contains $T$ absorbing an inner cobordism $W : (0, Q) \leadsto (1,Q')$. in this section we shall prove a general proposition in this direction, and later (in Corollaries \ref{cor:kTrivMeansZero} and \ref{cor:kTrivMeansZero2}) explain its consequences for the notion of $k$-triviality.

\begin{prop}\label{prop:AbsorptionProp}
Let ${W} : (0, {Q}_{\sticks}) \leadsto (1,{Q}'_{\cross})$ be an inner cobordism and $U : L' \leadsto L$ be an outer cobordism which is standard on $\bB$, such that $U$ absorbs $W$ (in the sense of Definition \ref{defn:absorbs}). Then the map induced on relative homology by the commutative square
\begin{equation}\label{eq:SquareNeedsDiag}
\begin{gathered}
\xymatrix{
\mathcal{N}^\theta_{L'}(1,Q) \ar[r]^-{(W+e_0)_*} \ar[d]^-{U^*}& \mathcal{N}^\theta_{L'}(2,Q') \ar[d]^-{U^*} \\
\mathcal{N}^\theta_L(0,Q) \ar[r]^-{W_*}& \mathcal{N}^\theta_L(1,Q')
}
\end{gathered}
\end{equation}
factors as
\begin{equation}\label{eq:factorisation}
\begin{gathered}
\xymatrix{
H_*(\mathcal{N}^\theta_{L'}(2,Q'), \mathcal{N}^\theta_{L'}(1,Q)) \ar[d]^-{(U^*)_*} \ar[r]^-\partial & H_{*-1}(\mathcal{N}^\theta_{L'}(1,Q)) \ar[d]^-\Delta\\
H_*(\mathcal{N}^\theta_L(1,Q'), \mathcal{N}^\theta_L(0,Q)) & H_*(\mathcal{N}^\theta_L(1,Q')) \ar[l]
}
\end{gathered}
\end{equation}
for a certain map $\Delta$. Furthermore, there is a cobordism of inner boundary conditions ${Y} : (0,{Q}_{\cross}'') \leadsto (1,{Q}_{\sticks})$ having a single relative 1-handle, such that the composition
\begin{align*}
H_{*-1}(\mathcal{N}^\theta_{L'}(0,Q'')) &\overset{Y_*}\lra H_{*-1}(\mathcal{N}^\theta_{L'}(1,Q)) \\
\overset{\Delta}\lra H_*(\mathcal{N}^\theta_L(1,Q')) &\lra H_*(\mathcal{N}^\theta_L(1,Q'), \mathcal{N}^\theta_L(0,Q))
\end{align*}
is zero. In particular, the map induced by $U^*$ on relative homology is zero on the subgroup
$$\partial^{-1}\mathrm{Im}\left(Y_* : H_{*-1}(\mathcal{N}^\theta_{L'}(0,Q'')) \to H_{*-1}(\mathcal{N}^\theta_{L'}(1,Q))\right).$$
\end{prop}

Before we begin with the proof of this proposition, we require a technical lemma concerning commutative squares which admit a diagonal map up to homotopy. In earlier drafts of this paper the use of this lemma was rather implicit, and it has been made explicit, at the level of spaces rather than homology, by several authors \cite[Lemma 6.2]{MartinPalmer} \cite[Lemma 6.2]{CR-W}. We give another proof, only at the level of homology.

\begin{lem}\label{lem:Factorisation}
Suppose we are given a commutative square
\begin{equation*}
\xymatrix{
A \ar[r]^f \ar[d]_h& B \ar@{-->}[ld]_d\ar[d]^i\\
C \ar[r]_g& D,
}
\end{equation*}
a map $d : B \to C$, and homotopies $F : d \circ f \simeq h$ and $G : g \circ d \simeq i$. We obtain homotopies
$$g \circ F : g \circ d \circ f \simeq g \circ h \quad\quad G \circ f : g \circ d \circ f \simeq i \circ f$$
and as $g \circ h = i \circ f$ we have a pair of homotopies between the same two maps. These glue to a map $\delta : ([0,2] / \partial[0,2]) \times A \to D$ which starts at $g \circ d \circ f$, then does the homotopy $G \circ f$, then does the reverse of the homotopy $g \circ F$. Let $\sigma \in H_1([0,2] / \partial[0,2])$ be the fundamental class, and define $\Delta = \Delta(d, F, G) : H_{*-1}(A) \to H_{*}(D)$ to be the map on homology $\delta_*(\sigma \otimes -)$.

Then there is  a factorisation
$$(i, h)_* : H_*(B,A) \overset{\partial}\lra H_{*-1}(A) \overset{\Delta}\lra H_*(D) \lra H_*(D,C)$$
of the induced map of relative homology.
\end{lem}

\begin{proof}
Without loss of generality, we may assume that $f$ is the inclusion of a closed subspace (by replacing $B$ with the mapping cylinder of $f$). Consider the space
$$X = ([0,1] \times B)\cup ([1,2] \times A) \subset [0,2] \times B$$
and define a map $\varphi : X \to D$ by $\varphi\vert_{[0,1] \times B} = G$ and $\varphi\vert_{[1,2] \times A}(t, a) = g \circ F(2-t, a)$. Similarly, define $\phi : \{0\} \times B \cup [1,2] \times A \to C$ by $\phi\vert_{\{0\} \times B} = d$ and $\phi\vert_{[1,2] \times A}(t, a) = F(2-t, a)$. We identify the pair $(B,A)$ as $(\{1\} \times B, \{1\} \times A)$, and have the following commutative diagram
\begin{equation*}
\xymatrix{
(B,A) \ar[d]^-\simeq \ar@/_4pc/[dd]_{(i,h)}\\
 (X, [1,2] \times A) \ar[r] \ar[d]^{(\varphi, \phi\vert_{[1,2] \times A})} & (X, \{0\} \times B \cup [1,2] \times A) \ar[ld]^{(\varphi, \phi)} & ([0,2] \times A, \{0, 2\} \times A) \ar[l]\\
(D,C).
}
\end{equation*}
The factorisation claim now follows, as
$$H_*(B, A) \lra H_*(X, \{0\} \times B \cup [1,2] \times A) \overset{\cong}\lla H_*([0,2] \times A, \{0,2\} \times A) \cong H_{*-1}(A)$$
is the usual boundary map, and
$$H_{*-1}(A) \cong H_*([0,2] \times A, \{0,2\} \times A) \overset{(\varphi, \phi)\vert_{[0,2] \times A}}\lra H_*(D,C)$$
is the map $\Delta$ we have constructed, followed by $H_*(D) \to H_*(D,C)$.
\end{proof}

\begin{proof}[Proof of Proposition \ref{prop:AbsorptionProp}]
By definition of $U$ absorbing $W$, there is a cobordism $(1,Z) : L'_{Q'} \leadsto L_{Q} \in \mathcal{C}_{\theta, T}$ such that
\begin{enumerate}[(i)]
\item\label{it:Homotopy1} there is a path from $(1,Z) \circ M(W, L')$ to $M(Q,U) \circ (1,[0,1] \times L'_{Q})$ in the space $\mathcal{C}_{\theta, T}(L'_{Q},L_{Q})$, and

\item\label{it:Homotopy2} there is a path from $M(W, L) \circ (1,Z)$ to $(1,[0,1] \times L_{Q'}) \circ M(Q',U)$ in the space $\mathcal{C}_{\theta, T}(L'_{Q'},L_{Q'})$.
\end{enumerate}
Let us define a manifold
$$\bar{Z} := ([0,2] \times L') \cup (Z + 2\cdot e_0) \subset ([0,2] \times \bA) \cup ([2,3] \times \bR^\infty)$$
and so a map
\begin{align*}
d: \mathcal{N}^\theta_{L'}(2,Q') & \lra \mathcal{N}^\theta_L(0,Q)\\
X &\longmapsto X \cup \bar{Z} - 3\cdot e_0.
\end{align*}
This gives a diagonal map for the square \eqref{eq:SquareNeedsDiag}, and we will show that it makes both triangles commute up to homotopy. In order to explain the homotopies making the triangles commute, it is convenient to use the graphical calculus shown in Figure \ref{fig:Absorption}.

\begin{figure}[h]
\centering
\includegraphics[bb = 0 0 383 71]{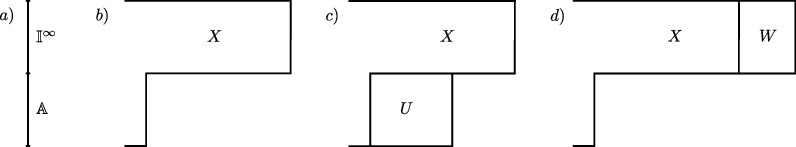}
\caption{a) We represent manifolds in $[0,t] \times \bR^\infty$ by pictures in $[0,t] \times \bI$, by putting the part inside $\bI^\infty$ in $[0,t] \times [0,1]$ and the part inside $\bA$ in $[0,t] \times [-1,0]$; b) A manifold $X \in \mathcal{N}^\theta_{L}(t, Q)$; c) Gluing on an outer cobordism $U : L \leadsto L'$; d) Gluing on an inner cobordism $W : Q \leadsto Q'$.}\label{fig:Absorption}
\end{figure}

Firstly, in the top triangle we have the homotopy given in Figure \ref{fig:Absorption2}, from $U^*$ to $d \circ (W + e_0)_*$.

\begin{figure}[h]
\centering
\includegraphics[bb = 0 0 383 71]{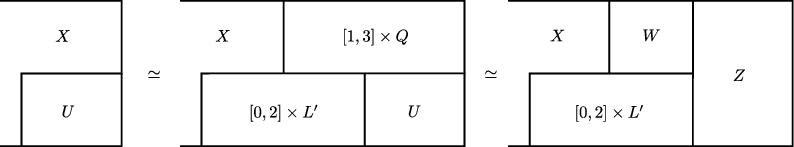}
\caption{}\label{fig:Absorption2}
\end{figure}

The homotopy of Figure \ref{fig:Absorption2} starts from the map $X \mapsto U^*(X) = X \cup (U+e_0) - e_0$, changes it by a homotopy to the map
$$X \longmapsto (X \cup ([0,2]\times L') \cup ([1,3]\times Q) \cup (U+3e_0))-3e_0$$
by stretching, and then uses the path from 
$$([1,3]\times Q) \cup ([1,2] \times L') \cup (U+3e_0) = M(Q,U) \circ ([0,1] \times L') + e_0$$
to
$$(W+e_0) \cup ([1,2] \times L') \cup (Z + 2\cdot e_0) = Z \circ M(W, L') + e_0$$
given by item (\ref{it:Homotopy1}) above.

Secondly, in the bottom triangle we have the homotopy from $U^*$ to $W_* \circ d$ given by postcomposing the homotopy shown in Figure \ref{fig:Absorption3} with a homotopy which shrinks down $[3,4] \times L'$ and $[4,5] \times Q'$.

\begin{figure}[h]
\centering
\includegraphics[bb = 0 0 313 71]{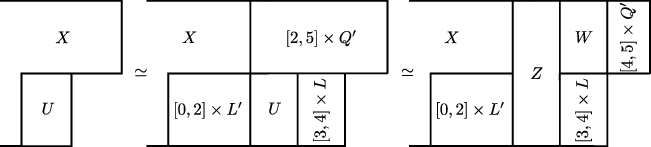}
\caption{}\label{fig:Absorption3}
\end{figure}

The homotopy of Figure \ref{fig:Absorption3} starts from the map $X \mapsto U^*(X) = X \cup (U+e_0) - e_0$, and changes it by a homotopy to the map
$$X \longmapsto (X \cup([0,2] \times L') \cup ([2,5] \times Q') \cup (U+2e_0) \cup ([3,4] \times L)) -4e_0$$
by stretching. Then it uses the path from
$$([2,4] \times Q') \cup (U+2e_0) \cup ([3,4] \times L) = ([0,1] \times L_{Q'}) \circ M(Q',U) + 2e_0$$
to
$$(Z + 2e_0) \cup (W + 3e_0) \cup ([3,4] \times L) = M(W,L) \circ Z + 2e_0$$
given by item (\ref{it:Homotopy2}) above.

To this diagonal map $d$ and these homotopies we may apply Lemma \ref{lem:Factorisation}, which gives the factorisation \eqref{eq:factorisation} where the map $\Delta$ arises from the self-homotopy of the map
$$(- \cup U \cup (W+e_0))-e_0 : \mathcal{N}^\theta_{L'}(1,Q) \lra \mathcal{N}^\theta_{L}(1,Q')$$
that we have constructed. This proves the first part of the proposition.

Considering the morphism $V:=(1,W \cup (U+e_0)) : L'_Q \leadsto L_{Q'} \in \mathcal{C}_{\theta,T}$, an equivalent model for this above map is
\begin{align*}
V_* : \mathcal{N}^\theta(L'_Q) &\lra \mathcal{N}^\theta(L_{Q'})\\
X &\longrightarrow (X \cup V) -e_0,
\end{align*}
and in this model $\Delta$ is induced by a loop $\gamma : S^1 \to \mathcal{C}_{\theta, T}(L'_Q, L_{Q'})$ based at the point $V$. This loop is recalled graphically in Figure \ref{fig:Absorption4}.
\begin{figure}[h]
\centering
\includegraphics[bb = 0 0 312 127]{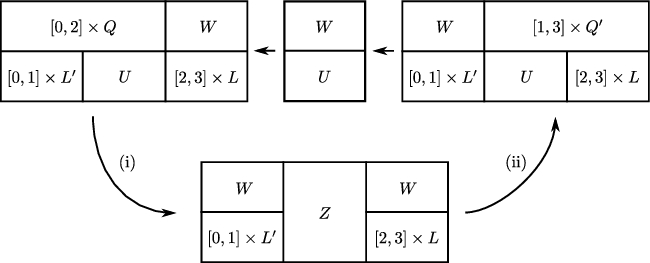}
\caption{}\label{fig:Absorption4}
\end{figure}
We wish to find a cobordism of inner boundary conditions $Y : Q''_{\cross} \leadsto Q_{\sticks}$ having a single relative 1-handle, and a loop $\gamma'$ in $\mathcal{C}_{\theta, T}(L'_{Q''}, L_{Q})$ based at some $V'$, so that for the associated map $\delta' : S^1 \times \mathcal{N}^\theta(L'_{Q''}) \to \mathcal{N}^\theta(L_{Q})$ the square
\begin{equation}\label{eq:NecessarySquare}
\begin{gathered}
\xymatrix{
S^1 \times \mathcal{N}^\theta(L'_{Q''}) \ar[d]^-{\mathrm{Id} \times Y_*}\ar[r]^-{\delta'}& \mathcal{N}^\theta(L_{Q}) \ar[d]^-{W_*}\\
S^1 \times \mathcal{N}^\theta(L'_Q) \ar[r]^-\delta & \mathcal{N}^\theta(L_{Q'})
}
\end{gathered}
\end{equation}
commutes up to homotopy. This will imply that $\mathrm{Im}(\Delta \circ Y_* : H_{*-1}(\mathcal{N}^\theta(L'_{Q''})) \to H_*(\mathcal{N}^\theta(L_{Q'})))$ is contained inside $\mathrm{Im}(W_* : H_*(\mathcal{N}^\theta(L_Q)) \to H_*(\mathcal{N}^\theta(L_{Q'}))$, and passing back to the original model this proves the second part of the proposition.

So far we have not used that all the cobordisms and paths of cobordisms that we have constructed are standard inside $\bB$ (and so that the loop $\gamma$ is one of cobordisms standard inside $\bB$), but we shall now do so. Consider a cobordism of inner boundary conditions ${Y} : {Q}''_{\cross} \leadsto {Q}_{\sticks}$ having a single relative 1-handle, so there are maps
\begin{equation}\label{eq:NeedDashedMap}
\begin{gathered}
\xymatrix{
 & \mathcal{C}_{\theta,T}(L'_{Q''}, L_Q) \ar[d]^-{W \circ -}\\
\mathcal{C}_{\theta,T}(L'_{Q}, L_{Q'}) \ar@{-->}[ur] \ar[r]^-{-\circ Y}& \mathcal{C}_{\theta,T}(L'_{Q''}, L_{Q'}).
}
\end{gathered}
\end{equation}
We claim that the $\theta$-structure on $Y$ (and $Q''$) can be chosen so that there is a dashed map making this triangle commute up to homotopy. To see this, note that after identifying each of these morphism spaces in $\mathcal{C}_{\theta,T}$ with a space of nullbordisms the diagram becomes
\begin{equation*}
\xymatrix{
 & \mathcal{N}^\theta(P'') \ar[d]^-{- \cup B}\\
\mathcal{N}^\theta(P') \ar@{-->}[ur] \ar[r]^-{- \cup A}& \mathcal{N}^\theta(P),
}
\end{equation*}
where
$$P \cong (Q''_\cross \cup (L' \setminus \mathrm{int}(T))) \cup_{\partial T} (Q'_\cross \cup (L \setminus \mathrm{int}(T)))$$
and the cobordisms $A : P' \leadsto P$ and $B: P'' \leadsto P$ are both obtained relative to $P$ by attaching a single 1-handle, along $Q''_\cross$ and $Q'_\cross$ respectively. The pair of oriented intervals ${Q}''_\cross$ in $P$ are isotopic to the pair of intervals ${Q}'_\cross$ with the opposite orientation, and reversing the orientation does not change the diffeomorphism type of surface obtained. Thus we may choose the $\theta$-structure on $Y$ so that $P'$ and $P''$ are isomorphic in $\mathcal{C}_\theta$, and a choice of such an isomorphism gives the required dotted arrow.

For this $Y$, applying the dashed map in \eqref{eq:NeedDashedMap} to the loop $\gamma$ in $\mathcal{C}_{\theta,T}(L'_{Q}, L_{Q'})$ gives a loop $\gamma'$ in $\mathcal{C}_{\theta,T}(L'_{Q''}, L_Q)$, and by constructing this choice makes the square \eqref{eq:NecessarySquare} commute up to homotopy, as required.
\end{proof}

\section{Proof of Theorem \ref{thm:StabOrientableSurfaces}}\label{sec:PfOrientable}

The statement of Theorem \ref{thm:StabOrientableSurfaces} (\ref{it:OStabAlpha}) concerns the effect on homology of a stabilisation map $$W_* : \mathcal{M}^\theta(g,+;{P}) \lra \mathcal{M}^\theta(g+1,+;{P}')$$ of type $\alpha$, saying that it is surjective in degrees $* \leq F(g)$ and injective in degrees $* \leq F(g)-1$. We may equivalently phrase this as saying that the relative homology groups $H_*(\mathcal{M}^\theta(g+1,+;{P}'), \mathcal{M}^\theta(g,+;{P}))$ defined by the map $W_*$ vanish in degrees $* \leq F(g)$. We will prove this latter statement, simultaneously with the corresponding statement for Theorem \ref{thm:StabOrientableSurfaces} (\ref{it:OStabBeta}), by induction on $g$. The induction is based on the map of resolutions described in Proposition \ref{prop:StabMapsOnRes}, and the description of the $p$-simplices of these resolutions given in Section \ref{sec:Layers}. In fact, these ingredients already prove vanishing on the above relative homology group in degrees $* < X(g)$ (cf.\ Figure \ref{fig:sseqNew1}). To increase the vanishing range to $* \leq F(g)$ we use Proposition \ref{prop:AbsorptionProp}, and a somewhat technical spectral sequence comparison argument.

Before beginning the proof of Theorem \ref{thm:StabOrientableSurfaces} in earnest, we introduce some convenient notation, and record the consequence of Proposition \ref{prop:AbsorptionProp} which we shall use.

\begin{defn}\label{defn:CpctNotation}
Let ${W} : (0,{Q}) \leadsto (1,{Q}')$ be an inner cobordism, and ${L}$ be an outer boundary condition. If $W \cup ([0,1] \times L) : {L}_{{Q}} \leadsto {L}_{{Q}'}$ is a stabilisation map of type $\alpha$ then we write $\alpha_{{L}}(g,+;W)$ for the pair of spaces given by the map
$$W_* : \mathcal{M}^\theta_{{L}}(g,+;0,Q) \lra \mathcal{M}^\theta_{{L}}(g+1,+;1,Q').$$
Similarly, if $W \cup ([0,1] \times L)$ is a stabilisation map of type $\beta$ then we write $\beta_{{L}}(g,+;W)$ for the pair of spaces given by the map
$$W_* : \mathcal{M}^\theta_{{L}}(g,+;0,Q) \lra \mathcal{M}^\theta_{{L}}(g,+;1,Q').$$
\end{defn}

If ${W} : (0,{Q}_{\sticks}) \leadsto (1,{Q}'_{\cross})$ is an elementary stabilisation map and $U : L' \leadsto L$ is an outer cobordism such that $({W},U)$ is an orientable test pair of height $k$, then restricting the commutative square \eqref{eq:SquareNeedsDiag} to connected orientable surfaces of particular genera gives a commutative square
\begin{equation*}
\begin{gathered}
\xymatrix{
\mathcal{M}^\theta_{L'}(h,+;1,Q) \ar[r]^-{(W+e_0)_*} \ar[d]^-{U^*}& \mathcal{M}^\theta_{L'}(h',+;2,Q') \ar[d]^-{U^*} \\
\mathcal{M}^\theta_L(g,+;0,Q) \ar[r]^-{W_*}& \mathcal{M}^\theta_L(g',+;1,Q')
}
\end{gathered}
\end{equation*}
for certain $g$, $g'$, $h$ and $h'$, and if $\theta$ is $k$-trivial, so $U$ absorbs $W$, then Proposition \ref{prop:AbsorptionProp} has a consequence for the induced map
\begin{equation*}
\xymatrix{
H_*(\mathcal{M}^\theta_{L'}(h',+;2,Q'),\mathcal{M}^\theta_{L'}(h,+;1,Q)) \ar[d]^-{(U^*)_*}\\
H_*(\mathcal{M}^\theta_L(g',+;1,Q'), \mathcal{M}^\theta_L(g,+;0,Q))
}
\end{equation*}
on relative homology. The following corollary records these consequences in the various cases that we shall require, using the notation introduced in Definition \ref{defn:CpctNotation}.

\begin{cor}\label{cor:kTrivMeansZero}
Let $({W}, U)$ be an orientable test pair of height $k$.

\begin{enumerate}[(i)]
\item\label{it:cor:kTrivMeansZero:1} If $k=2K$ and the stabilisation map $W \cup ([0,1] \times L)$ is of type $\alpha$, then the induced map on relative homology is of the form
\begin{align*}
(U^*)_* : H_*(\alpha_{{L}'}(g-K,+;W)) \lra H_*(\alpha_{{L}}(g,+;W))
\end{align*}
and is zero in those homological degrees $*$ where all stabilisation maps
$$Y_*: H_{*-1}(\mathcal{M}^\theta_{{L}'}(g-K,+;0,Q'')) \lra H_{*-1}(\mathcal{M}^\theta_{{L}'}(g-K,+;1,Q))$$
of type $\beta$ are surjective.

\item If $k=2K+1$ and the stabilisation map $W \cup ([0,1] \times L)$ is of type $\alpha$, then the induced map on relative homology is of the form
\begin{align*}
(U^*)_* : H_*(\beta_{{L}}(g-K,+;W)) \lra H_*(\alpha_{{L}}(g,+;W))
\end{align*}
and is zero in those homological degrees $*$ where all stabilisation maps
$$Y_*: H_{*-1}(\mathcal{M}^\theta_{{L}'}(g-K-1,+;0,Q'')) \lra H_{*-1}(\mathcal{M}^\theta_{{L}'}(g-K,+;1,Q))$$
of type $\alpha$ are surjective.

\item If $k=2K$ and the stabilisation map $W \cup ([0,1] \times L)$ is of type $\beta$, then the induced map on relative homology is of the form
\begin{align*}
(U^*)_* : H_*(\beta_{{L}}(g-K,+;W)) \lra H_*(\beta_{{L}}(g,+;W))
\end{align*}
and is zero in those homological degrees $*$ where all stabilisation maps
$$Y_*: H_{*-1}(\mathcal{M}^\theta_{{L}'}(g-K-1,+;0,Q'')) \lra H_{*-1}(\mathcal{M}^\theta_{{L}'}(g-K,+;1,Q))$$
of type $\alpha$ are surjective.

\item If $k=2K+1$ and the stabilisation map $W \cup ([0,1] \times L)$ is of type $\beta$, then the induced map on relative homology is of the form
\begin{align*}
(U^*)_* : H_*(\alpha_{{L}}(g-K-1,+;W)) \lra H_*(\beta_{{L}}(g,+;W))
\end{align*}
and is zero in those homological degrees $*$ where all stabilisation maps
$$Y_*: H_{*-1}(\mathcal{M}^\theta_{{L}'}(g-K-1,+;0,Q'')) \lra H_{*-1}(\mathcal{M}^\theta_{{L}'}(g-K-1,+;1,Q))$$
of type $\beta$ are surjective.
\end{enumerate}
\end{cor}

Now, suppose that $\theta$ is a tangential structure that stabilises on $\pi_0$ for orientable surfaces at genus $h$ and is $k$-trivial. Let $F, G, X, Y : \bZ \to \bZ$ be the functions given in Definition \ref{defn:Orientable:Range}, using which we may express the following statements:
\begin{equation}\tag{F$_y$}
\text{For all $g \leq y$ and all $W$ and $L$, }  H_*(\alpha_{{L}}(g,+;W)) = 0 \text{ in degrees } * \leq F(g).
\end{equation}
\begin{equation}\tag{G$_y$}
\text{For all $g \leq y$ and all $W$ and $L$, } H_*(\beta_{{L}}(g,+;W)) = 0 \text{ in degrees } * \leq G(g).
\end{equation}

If we have proved these statements for all $y$, then we have proved parts (\ref{it:OStabAlpha}) and (\ref{it:OStabBeta}) of Theorem \ref{thm:StabOrientableSurfaces} (using Lemma \ref{lem:EltStab}, which shows that is is enough to consider elementary stabilisation maps); by Remark \ref{rem:BetaGamma} part (\ref{it:OStabGamma}) follows from part (\ref{it:OStabBeta}), and so we have proved Theorem \ref{thm:StabOrientableSurfaces}.

In proving these two statements, we must in parallel prove two more technical statements, via an induction which combines all four statements. Let us explain these two more technical statements. By Proposition \ref{prop:FibResOfPairs} (\ref{it:FibResOfPairs:ResAlpha}), when an elementary stabilisation map $W_* : \mathcal{M}^\theta_{{L}}(g,+;0,Q) \to \mathcal{M}^\theta_{{L}}(g+1,+;1,Q')$ of type $\alpha$ is resolved (using resolution data $b : \{\pm1\} \times \bR \hookrightarrow L$), it gives a map
$$\mathcal{B}^\theta_{{L}}(g,+; 0,{Q};b)_0 \lra \mathcal{H}^\theta_{{L}}(g+1,+; 1,{Q}';b)_0$$
on $0$-simplices, which is a map over $A_0(0;b,\ell_b,+) \overset{\sim}\hookrightarrow A_0(1;b,\ell_b,+)$. On fibres over $x \in A_0(0;b,\ell_b,+)$, the map is homotopy equivalent to the elementary stabilisation map of type $\beta$
$$W_*: \mathcal{M}^\theta_{{L}_x}(g,+;0,{Q}) \lra \mathcal{M}^\theta_{{L}_x}(g,+;1,{Q}'),$$
where the new outer boundary condition ${L}_x$ depends on the point $x \in A_0(0;b,\ell_b,+)$ we are working over. Thus we obtain a map of pairs
$$\epsilon_x : \beta_{{L}_x}(g,+;W) \lra \alpha_{{L}}(g,+;W).$$
By choosing one point $x$ in each path component of $A_0(0;b,\ell_b)$, and summing together the maps $\epsilon_x$ on homology, we obtain a map
\begin{equation}\label{eq:AugmentationAlpha}
\epsilon_* : \bigoplus_{\mathclap{[x] \in \pi_0(A_0(0;b,\ell_b,+))}} H_*(\beta_{{L}_x}(g,+;W)) \lra H_*(\alpha_{{L}}(g,+;W)).
\end{equation}

Similarly, for $W_* : \mathcal{M}^\theta_{{L}}(g,+;0,Q) \to \mathcal{M}^\theta_{{L}}(g,+;1,Q')$ an elementary stabilisation map of type $\beta$ we obtain a map
\begin{equation}\label{eq:AugmentationBeta}
\epsilon_* : \bigoplus_{\mathclap{[x] \in \pi_0(A_0(0;b,\ell_b,+))}} H_*(\alpha_{{L}_x}(g-1,+;W)) \lra H_*(\beta_{{L}}(g,+;W)).
\end{equation}

We can now give the additional two statements that we will simultaneously prove:
\begin{equation}\tag{X$_y$}
\text{For all $g \leq y$ and all $W$, $L$, and $b$,  \eqref{eq:AugmentationAlpha} is epi in degrees} \,\, * \leq X(g).
\end{equation}
\begin{equation}\tag{Y$_y$}
\text{For all $g \leq y$ and all $W$, $L$, and $b$, \eqref{eq:AugmentationBeta} is epi in degrees} \,\, * \leq Y(g).
\end{equation}

By assumption, $\theta$ stabilises on $\pi_0$ at genus $h$, so $H_0(\alpha_{{L}}(g,+;W))=0$ and $H_0(\beta_{{L}}(g,+;W))=0$ for all $g \geq h-1$ and all $W$ and $L$, so certainly the statements $F_{h-1}$, $G_{h-1}$, $X_{h-1}$ and $Y_{h-1}$ hold, as each of the functions $F$, $G$, $X$ and $Y$ take the value 0 at $h-1$ and $-1$ below $h-1$. This starts our induction.

The inductive step is provided by the following technical theorem; we leave it to the sceptical reader to convince themselves that these implications do indeed supply the inductive step.

\begin{thm}\label{thm:InductiveStepOrientable}
Suppose the hypotheses of Theorem \ref{thm:StabOrientableSurfaces} hold, then there are implications
\begin{enumerate}[(i)]
	\item\label{it:thm:InductiveStepOrientable:1} $F_{g-1}$, $G_g$ and $Y_{g-1}$ imply $X_g$.
	\item\label{it:thm:InductiveStepOrientable:2} $F_{g-1}$, $G_{g-1}$ and $X_{g-2}$ imply $Y_g$.
	\item\label{it:thm:InductiveStepOrientable:3} If $k=1$, $X_g$ and $F_{g-1}$ imply $F_g$. 

\noindent If $k > 1$, $X_g$, $Y_{g}$ and $\begin{cases}G_{g-K} & \text{if $k=2K$}\\F_{g-K-1}& \text{if $k=2K+1$}\end{cases}$ imply $F_g$.
	\item\label{it:thm:InductiveStepOrientable:4} If $k=1$, $Y_g$ and $G_{g-1}$ imply $G_g$. 

\noindent If $k > 1$, $Y_{g}$, $X_{g-1}$ and $\begin{cases}F_{g-K-1} & \text{if $k=2K$}\\G_{g-K-1}& \text{if $k=2K+1$}\end{cases}$ imply $G_g$.
\end{enumerate}
\end{thm}

\begin{rem}
The argument below can be simplified if we are willing to strengthen Definition \ref{defn:Orientable:Range} (\ref{it:Orientable:XY}) so that the functions defining the stability range are
$$
X(g) = \min \begin{cases}
F(g-1)+1\\
G(g-1)+1\\
X(g-1)+1\\
Y(g-1)+1\\
0 \text{ if } g \leq 0
\end{cases}\quad\quad
Y(g) = \min \begin{cases}
F(g-2)+1\\
G(g-1)+1\\
X(g-2)+1\\
Y(g-1)+1\\
0 \text{ if } g \leq 1.
\end{cases}$$
In this case, Theorem \ref{thm:InductiveStepOrientable} (\ref{it:thm:InductiveStepOrientable:1}) and (\ref{it:thm:InductiveStepOrientable:2}) may be replaced by the statements that $G_g$ implies $X_g$, and $F_{g-1}$ implies $Y_g$. In the proof below, the difficult Step 3 becomes unnecessary, as the grey dot in Figure \ref{fig:sseqNew1} is zero.

In particular, in the case $k=1$ the reader can check that the functions given by Definition \ref{defn:Orientable:Range} already have the stronger property described above, that is, they satisfy
$$X(g) \leq G(g-1)+1 \quad\quad Y(g) \leq F(g-2)+1.$$
Thus in this case Step 3 of the argument below may be omitted.
\end{rem}

\begin{proof}[Proof of Theorem \ref{thm:InductiveStepOrientable} (\ref{it:thm:InductiveStepOrientable:1}) and (\ref{it:thm:InductiveStepOrientable:2})]

For concreteness, let us prove statement (\ref{it:thm:InductiveStepOrientable:1}); (\ref{it:thm:InductiveStepOrientable:2}) is completely analogous.

Given an elementary stabilisation map $W_* : \mathcal{M}^\theta_{{L}}(g,+;0,Q) \to \mathcal{M}^\theta_{{L}}(g+1,+;1,Q')$ of type $\alpha$ and a certain embedding $b : \{ \pm1\} \times \bR \hookrightarrow L$, in Proposition \ref{prop:StabMapsOnRes} we have explained how it may be covered by a map of resolutions
\begin{equation}\label{eq:InductiveStepOrientable:ResMap}
\begin{gathered}
\xymatrix{
\mathcal{B}^\theta_{{L}}(g,+; 0,{Q};b)_\bullet \ar[d]\ar[r] & \mathcal{H}^\theta_{{L}}(g+1,+; 1,{Q}';b)_\bullet \ar[d]\\
\mathcal{M}^\theta_{{L}}(g,+;0,Q) \ar[r]^-{W_*} & \mathcal{M}^\theta_{{L}}(g+1,+;1,Q').
}
\end{gathered}
\end{equation}

\noindent \textbf{Step 1}. We have shown how the map between spaces of $p$-simplices of these resolutions may be understood using the map of Serre fibrations
\begin{equation}\label{eq:InductiveStepOrientable:Fibnmap}
\begin{gathered}
\xymatrix{
\mathcal{B}^\theta_{{L}}(g,+; 0,{Q};b)_p \ar[d]^-{r_p}\ar[r] & \mathcal{H}^\theta_{{L}}(g+1,+; 1,{Q}';b)_p \ar[d]^-{r_p}\\
A_p(0;b, \ell_b,+) \ar[r]^-\simeq & A_p(1;b, \ell_b,+),
}
\end{gathered}
\end{equation}
and in Proposition \ref{prop:MapOnFibres} we have shown that on each fibre over $x \in A_p(0;b, \ell_b,+)$ this may be modelled as a stabilisation map of type $\beta$,
$$W_* : \mathcal{M}^\theta_{{L}_x}(g-p,+; 0,{Q}) \lra \mathcal{M}^\theta_{{L}_x}(g-p,+; 1,{Q}'),$$
for some outer boundary condition ${L}_x$ which depends on $x$. If we pull back the right-hand fibration to $A_p(0;b, \ell_b,+)$, then we have a map of Serre fibrations over the same base space, and so a relative Serre spectral sequence
\begin{align*}
E^2_{s,t} &= H_s(A_p(0;b, \ell_b,+) ; \underline{H_t}(\beta_{{L}_x}(g-p,+;W))) \\
&\Rightarrow H_{s+t}(\mathcal{H}^\theta_{{L}}(g+1,+; 1,{Q}')_p, \mathcal{B}^\theta_{{L}}(g,+; 0,{Q})_p),
\end{align*}
where $\underline{H_t}(\beta_{{L}_x}(g-p,+;W))$ denotes the system of local coefficients on $A_p(0;b, \ell_b,+)$ having fibre $H_t(\beta_{{L}_x}(g-p,+;W))$ over $x$.
As we have assumed that $G_g$ holds, it follows that $E^2_{s,t}=0$ for $t \leq G(g-p)$ and so
\begin{equation}\label{eq:InductiveStepOrientable:Vanishing}
H_{*}(\mathcal{H}^\theta_{{L}}(g+1,+; 1,{Q}')_p, \mathcal{B}^\theta_{{L}}(g,+; 0,{Q})_p)=0 \quad \text{ for } * \leq G(g-p).
\end{equation}
We can extract slightly more information: there is a surjection
$$\bigoplus_{\mathclap{[x] \in \pi_0(A_p(0;b,\ell_b,+))}} H_t(\beta_{{L}_x}(g-p,+;W)) \lra H_0(A_p(0;b, \ell_b,+) ; \underline{H_t}(\beta_{{L}'}(g-p,+;W))),$$
and in total degree $s+t = G(g-p)+1$ only the group $E^2_{0, G(g-p)+1}$ is non-zero, so composing with the edge homomorphism we find that the natural map
$$\bigoplus_{\mathclap{[x] \in \pi_0(A_p(0;b,\ell_b,+))}} H_t(\beta_{{L}_x}(g-p,+;W)) \lra H_t(\mathcal{H}^\theta_{{L}}(g+1,+; 1,{Q}')_p, \mathcal{B}^\theta_{{L}}(g,+; 0,{Q})_p)$$
is surjective for $t \leq G(g-p)+1$.

\vspace{2ex}

\noindent \textbf{Step 2}. We now study the spectral sequence (\ref{SSRelativeAugmentedRestrictedSimplicialSpace}) from Section \ref{sec:AugSSSpaces} for the map of augmented semi-simplicial spaces \eqref{eq:InductiveStepOrientable:ResMap}, which takes the form
$$E^1_{p,q} = H_q(\mathcal{H}^\theta_{{L}}(g+1,+; 1,{Q}')_p, \mathcal{B}^\theta_{{L}}(g,+; 0,{Q})_p) \quad p \geq -1,\, q \geq 0.$$
It follows from Theorem \ref{thm:ResConnectivity1} that after geometric realisation the homotopy fibres of the vertical maps in \eqref{eq:InductiveStepOrientable:ResMap} are $(g-2)$- and $(g-1)$-connected respectively, and so this spectral sequence converges to zero in degrees $p+q \leq g-1$. We wish to draw a conclusion about the groups $E^1_{-1,q}$ for $q \leq X(g)$, but $X(g) \leq g$ because $X(0) \leq 0$ and $X(g) \leq X(g-1)+1$ by Definition \ref{defn:Orientable:Range} (\ref{it:Orientable:XY}). Thus these groups are in the range where the spectral sequence converges to zero.

By \eqref{eq:InductiveStepOrientable:Vanishing}, $E^1_{p,q}=0$ for $p \geq 0$ and $q \leq G(g-p)$, and there is a surjection
$$\bigoplus_{\mathclap{[x] \in \pi_0(A_p(0;b,\ell_b,+))}} H_{q}(\beta_{{L}_x}(g-p,+;W)) \lra E^1_{p, q}$$
for $q \leq G(g-p)+1$. As $G(g) \geq X(g)-1$ (by Definition \ref{defn:Orientable:Range} (\ref{it:Orientable:XY})) and $G(g-2) \geq X(g)-2$ (by the inequalities $X(g) \leq Y(g-1)+1 \leq G(g-2)+2$ of Definition \ref{defn:Orientable:Range} (\ref{it:Orientable:XY})), a chart of the $E^1$-page of this spectral sequence is as shown in Figure \ref{fig:sseqNew1}.

\begin{figure}[h]
\centering
\includegraphics[bb = 0 0 163 120]{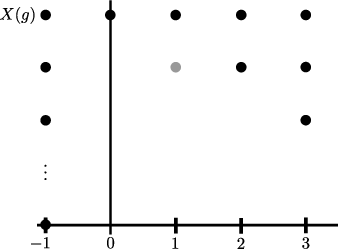}
\caption{The black or grey dots represent unknown groups, and the abscence of a dot represents the trivial group.}\label{fig:sseqNew1}
\end{figure}

The map \eqref{eq:AugmentationAlpha} is the composition
$$\bigoplus_{\mathclap{[x] \in \pi_0(A_0(0;b,\ell_b,+))}} H_{t}(\beta_{{L}_x}(g,+;W)) \lra E^1_{0, t} \overset{d^1}\lra E^1_{-1, t},$$
and we want to show it is surjective for $t \leq X(g)$. As the first map is surjective for $t \leq G(g)+1$, and $X(g) \leq G(g)+1$ by Definition \ref{defn:Orientable:Range} (\ref{it:Orientable:XY}), it will be enough to show that $d^1 : E^1_{0, t} \to E^1_{-1, t}$ is surjective for $t \leq X(g)$. As $E^\infty_{-1,t}=0$ for $t \leq X(g)$, it will thus be enough to show that all other differentials arriving at $E^r_{-1,t}$ are zero for $t \leq X(g)$. 

\vspace{2ex}

\noindent \textbf{Step 3}. From the chart in Figure \ref{fig:sseqNew1}, we see that there is a single possible additional differential, the differential $d^2 : E^2_{1, X(g)-1} \to E^2_{-1, X(g)}$ starting from the dot marked in grey. Because $E^1_{0, X(g)-1}=0$, the group $E^2_{1, X(g)-1}$ is a quotient of $E^1_{1, X(g)-1}$, and so there is a surjection
$$\bigoplus_{\mathclap{[x] \in \pi_0(A_1(0;b,\ell_b,+))}} H_{X(g)-1}(\beta_{{L}_x}(g-1,+;W)) \lra E^1_{1, X(g)-1} \lra E^2_{1, X(g)-1}.$$
It is enough to show that the composition of this surjection with $d^2$ is zero, and to do this, it is enough to show that for each $x \in A_1(0;b,\ell_b,+)$ the map
\begin{equation}\label{eq:NeedsToVanish}
H_{X(g)-1}(\beta_{{L}_x}(g-1,+;W)) \lra E^2_{1, X(g)-1} \overset{d^2}\lra E^2_{-1, X(g)}
\end{equation}
is zero. 

To do this, we shall first take a small detour. Let ${V} : {L}' \leadsto {L}$ be an outer cobordism which contains $[-1,0] \times b(\{\pm1\} \times \bR)$, and as an abstract manifold has a single 1-handle relative to $L$ attached along $b'\vert_{\{\pm 1\} \times \bI}$ for an embedding  $b' : \{ \pm 1\} \times \bR \hookrightarrow L$ which is disjoint from $b$ but isotopic to it, so there is a commutative square
\begin{equation*}
\xymatrix{
\mathcal{M}^\theta_{{L}'}(g,+;1,{Q}) \ar[d]^-{V^*}\ar[r]^-{(W + e_1)_*} & \mathcal{M}^\theta_{{L}'}(g,+;2,{Q}') \ar[d]^-{V^*}\\
\mathcal{M}^\theta_{{L}}(g,+;0,{Q}) \ar[r]^-{W_*} & \mathcal{M}^\theta_{{L}}(g+1,+;1,{Q}'),
}
\end{equation*} 
where the top map is a stabilisation map of type $\beta$. As the outer cobordism $V$ contains $[-1,0] \times b(\{\pm 1\} \times \bR)$, the map $V^*$ lifts to a map between arc resolutions (by extending arcs in $[-1,0] \times b(\{\pm1\} \times \bR)$ using the product structure, and then reparameterising). Thus the above square may be covered by a commutative square of semi-simplicial spaces
\begin{equation*}
\xymatrix{
\mathcal{H}^\theta_{{L}'}(g,+;1,{Q};b)_\bullet \ar[d]^-{V^*}\ar[r]^-{(W + e_1)_*} & \mathcal{B}^\theta_{{L}'}(g,+;2,{Q}';b)_\bullet \ar[d]^-{V^*}\\
\mathcal{B}^\theta_{{L}}(g,+; 0,{Q};b)_\bullet \ar[r]^-{W_*} & \mathcal{H}^\theta_{{L}}(g+1,+; 1,{Q}';b)_\bullet .
}
\end{equation*}
Thus we obtain a map of spectral sequences from that of the top map, which we call $\tilde{E}^r_{p,q}(V)$, to that of the lower map, which is the spectral sequence ${E}^r_{p,q}$ we have been working with above.

Repeating the calculation of Step 1 for the resolution of the top map, and using that $F_{g-1}$ holds, we find that $\tilde{E}^1_{p,q}(V)=0$ for $p \geq 0$ and $q \leq F(g-1-p)$. As $F(g-1-p) \leq X(g-p)-1 \leq X(g)-p-1$, it follows that $\tilde{E}^1_{p,q}(V)=0$ for $p \geq 0$ and $p+q \leq X(g)-1$. Furthermore, there is a surjection
$$\bigoplus_{\mathclap{[x] \in \pi_0(A_1(0;b,\ell_b,+))}} H_{X(g)-1}(\alpha_{{L}'_x}(g-2,+;W)) \lra \tilde{E}^1_{1, X(g)-1}(V) \lra \tilde{E}^2_{1, X(g)-1}(V),$$
because $X(g)-1 \leq F(g-2)+1$ and $\tilde{E}^1_{0, X(g)-1}(V)=0$. 

When choosing the cobordism $R_x : L_x \leadsto L$, we may suppose that it contains the strip $[-1,0] \times b'(\{ \pm 1\} \times \bR)$, and so that $b'$ defines an embedding into $L_x$ too. We now have cobordisms $V \circ R'_x : L'_x \leadsto L' \leadsto L$ and $R_x : L_x \leadsto L$ both ending at $L$. As $V \circ R'_x$ contains a handle relative to $L$ attached inside $b(\{ \pm 1 \} \times \bR)$ and with $\theta$-structure given by $x$, there is an embedding $R_x \hookrightarrow V \circ R'_x$ relative to $L$. Let $V_x : L'_x \leadsto L_x$ be the complementary outer cobordism, which has a single handle relative to $L_x$ attached via $b'\vert_{\{ \pm 1\} \times \bI}$. With this choice, there is a commutative diagram
\begin{equation*}
\xymatrix{
H_{X(g)-1}(\alpha_{{L}'_x}(g-2,+;W)) \ar[r]\ar[d]^-{V^*_x} & \tilde{E}^2_{1, X(g)-1}(V) \ar[d] \ar[r]^-{d^2} & \tilde{E}^2_{-1, X(g)}(V) \ar[d]\\
H_{X(g)-1}(\beta_{{L}_x}(g-1,+;W)) \ar[r]& {E}^2_{1, X(g)-1} \ar[r]^-{d^2} & {E}^2_{-1, X(g)}.
}
\end{equation*}

We claim that the right-hand vertical map is zero: the image of the map
$$H_q(\beta_{{L}'}(g,+;W)) = \tilde{E}^1_{-1, q}(V) \lra {E}^1_{-1, q} = H_q(\alpha_{{L}}(g,+;W))$$
is contained in the image of the differential $d^1 : {E}^1_{0, q} \to {E}^1_{-1, q}$, because the pair $\beta_{{L}'}(g,+;W)$ may be obtained as a particular fibre of the map \eqref{eq:InductiveStepOrientable:Fibnmap}: by construction, $V$ contains a single handle relative to $L$, which may be taken to be attached along the map $b\vert_{\{\pm 1\} \times \bI}$. Thus the map $\tilde{E}^2_{-1, q} \to {E}^2_{-1, q}$ is trivial.

Hence the map \eqref{eq:NeedsToVanish} is trivial on the image of 
$$V_x^* : H_{X(g)-1}(\alpha_{{L}'_x}(g-2,+;W)) \lra H_{X(g)-1}(\beta_{{L}_x}(g-1,+;W)).$$
However, this discussion holds for any choice of $V$, and we may choose $V$ so that $V_x$ has a single handle relative to $L_x$ attached via $b'\vert_{\{\pm 1\} \times \bI}$ but having any $\theta$-structure we like. Thus, summing over all possible $V$'s, we find that the map \eqref{eq:NeedsToVanish} is trivial on the image of the map
$$\bigoplus_{\mathclap{[y] \in \pi_0(A_0(0;b',\ell_{b'},+))}} H_{X(g)-1}(\alpha_{{L}_{x,y}}(g-2,+;W)) \lra H_{X(g)-1}(\beta_{{L}_x}(g-1,+;W)).$$
But as we have assumed that $Y_{g-1}$ holds, and $X(g)-1 \leq Y(g-1)$ by Definition \ref{defn:Orientable:Range} (\ref{it:Orientable:XY}), this map is surjective, and hence \eqref{eq:NeedsToVanish} is trivial as required.
\end{proof}

If we omit Step 3 in the proof above, and consult Figure \ref{fig:sseqNew1}, we see that we have proved the vanishing of $H_*(\alpha_L(g,+;W))$ in degrees $* \leq X(g)-1$. This is not sufficient for an induction to proceed. Our introduction of the statements ($X_y$) and ($Y_y$), of the functions $X$ and $Y$, and especially of the notion of $k$-triviality, is to allow the following argument instead.

\begin{proof}[Proof of Theorem \ref{thm:InductiveStepOrientable} (\ref{it:thm:InductiveStepOrientable:3}) and (\ref{it:thm:InductiveStepOrientable:4})]
Suppose that $k=2l$ and let us prove statement (\ref{it:thm:InductiveStepOrientable:3}): that $X_g$, $Y_g$, and $G_{g-l}$ imply $F_g$; the remaining cases are completely analogous. 

We are considering an elementary stabilisation map ${W} : {Q}_{\sticks} \leadsto {Q}'_{\cross}$ which induces a stabilisation map 
$$W_* : \mathcal{M}^\theta_{{L}}(g,+;0,{Q}) \lra \mathcal{M}^\theta_{{L}}(g+1,+;1,{Q})$$
of type $\alpha$, and we may suppose that $L$ is standard on $\bB$ (by changing it by an isomorphism of outer boundary conditions). Choose embeddings $b_i : \{ \pm1\} \times \bR \hookrightarrow L$ for $1 \leq i \leq k$ disjoint from $\bB$ arranged as in Figure \ref{fig:kTRivOr} b) on page \pageref{fig:kTRivOr}. Choose a sequence of outer cobordisms \emph{without $\theta$-structure}
$$L_{2l} \overset{R_{2l}}\leadsto L_{2l-1} \leadsto \cdots \leadsto L_2 \overset{R_2}\leadsto L_1 \overset{R_1}\leadsto L_0 := L$$
which are standard on $\bB$, so that $R_j$ contains $[-1,0] \times b_i(\{ \pm1\} \times \bR)$ for all $i > j$, and is obtained from $L_{j-1}$ by attaching a 1-handle along
$$b_j\vert_{\{ \pm 1\} \times \bI} : \{ \pm 1\} \times \bI \lra L_{j-1}.$$

If we resolve the map $W^*$ using the embedding $b_1$, then the corresponding map \eqref{eq:AugmentationAlpha} is
$$\bigoplus_{\mathclap{[x_1] \in \pi_0(A_0(0;b_1,\ell_{b_1},+))}} H_*(\beta_{{L}_{x_1}}(g,+;W)) \lra H_*(\alpha_{{L}}(g,+;W)).$$
Moreover, we may take the outer cobordism $R_{x_1} : L_{x_1} \leadsto L$ to have underlying manifold $R_1$, as the only requirement on the underlying manifold is that it should be obtained by attaching a handle along $b_1\vert_{\{\pm\} \times \bI}$, which $R_1$ is. As $R_1$ contains $[-1,0] \times b_i(\{\pm1\} \times \bR)$ for all $i \geq 2$, the maps $b_i$ for $i \geq 2$ have image inside each $L_{x_1}$, so we may use $b_2$ to resolve each $\beta_{{L}_{x_1}}(g,+;W)$. Continuing in this way, we obtain a long composition
\begin{align*}
\hspace*{1.2cm}
\bigoplus_{\mathclap{\substack{[x_1] \in \pi_0(A_0(0;b_1,\ell_{b_1},+)) \\ \vdots \\ [x_{2l}] \in \pi_0(A_0(0;b_{2l},\ell_{b_{2l}},+))}}}  H_*(\alpha_{{L}_{x_1, \ldots,x_{2l}}}(g-l,+;W)) \lra \cdots \lra \bigoplus_{\mathclap{\substack{[x_1] \in \pi_0(A_0(0;b_1,\ell_{b_1},+)) \\ [x_2] \in \pi_0(A_0(0;b_2,\ell_{b_2},+))}}}  H_*(\alpha_{{L}_{x_1, x_2}}(g-1,+;W)) \\
\lra \bigoplus_{\mathclap{[x_1] \in \pi_0(A_0(0;b_1,\ell_{b_1},+))}} H_*(\beta_{{L}_{x_1}}(g,+;W)) \lra H_*(\alpha_{{L}}(g,+;W)).
\end{align*}
Each map in this composition is a direct sum of maps of type \eqref{eq:AugmentationAlpha} or \eqref{eq:AugmentationBeta}: the leftmost is of type \eqref{eq:AugmentationBeta} and as we have assumed $Y_g$ holds is surjective in degrees $* \leq Y(g-l+1)$; the next is of type \eqref{eq:AugmentationAlpha} and as we have assumed $X_g$ holds is surjective in degrees $* \leq X(g-l+1)$; as we move to the right, the genus increases and the maps become surjective in an increasing range of degrees. Thus the composition is surjective in degrees $* \leq \min(Y(g-l+1),X(g-l+1))$, and so by Definition \ref{defn:Orientable:Range} (\ref{it:Orientable:kIsEven}) it is surjective in degrees $* \leq F(g)$.

On the other hand, on each summand of the source the map
\begin{equation}\label{eq:MapToVanish}
H_*(\alpha_{{L}_{x_1, \ldots,x_{2l}}}(g-l,+;W)) \lra H_*(\alpha_{{L}}(g,+;W))
\end{equation}
is induced by gluing on the cobordism $U := R_1 \circ R_2 \circ \cdots \circ R_{2l}$ equipped with some $\theta$-structure $\ell_U$ (which depends on the $x_i$). By the pattern in which we chose the intervals $b_i$ (which was that shown in Figure \ref{fig:kTRivOr} b)), the pair $({W}, U)$ is an orientable test pair of height $2l$. Thus, by Corollary \ref{cor:kTrivMeansZero} (\ref{it:cor:kTrivMeansZero:1}) the map \eqref{eq:MapToVanish} is zero in those degrees $*$ in which all stabilisation maps of type $\beta$
$$Y_*: H_{*-1}(\mathcal{M}^\theta_{{L}'}(g-l,+;0,Q'')) \lra H_{*-1}(\mathcal{M}^\theta_{{L}'}(g-l,+;1,Q))$$
are surjective. As we have assumed that $G_{g-l}$ holds, all such maps are surjective for $*-1 \leq G(g-l)$, and so by Definition \ref{defn:Orientable:Range} (\ref{it:Orientable:kIsEven}) for $*-1 \leq F(g)-1$. But then the long composition is both zero and surjective in degrees $* \leq F(g)$, and so $H_*(\alpha_{{L}}(g,+;W))=0$ in this range, as required.
\end{proof}

\section{Proof of Theorem \ref{thm:StabNonorientableSurfaces}}\label{sec:PfNonorientable}

To prove Theorem \ref{thm:StabNonorientableSurfaces} we take a somewhat different approach to the last section. Firstly we shall prove Theorem \ref{thm:StabNonorientableSurfaces} (\ref{it:NOStabMu}), that is, homological stability for forming the boundary connect-sum with  projective planes. The proof of this is analogous to the proof given in the case of orientable surfaces, but enjoys two (related) simplifications: firstly, as described in Proposition \ref{prop:FibResOfPairs} (\ref{it:FibResOfPairs:ResMu}), when we resolve an elementary stabilisation map of type $\mu$ with the projective plane resolution, the map at the level of $p$-simplices can again be expressed in terms of elementary stabilisation maps of type $\mu$; secondly, the analogue of Step 3 in the proof of Theorem \ref{thm:InductiveStepOrientable} (\ref{it:thm:InductiveStepOrientable:1}) and (\ref{it:thm:InductiveStepOrientable:2}) does not arise.

Once we have proved Theorem \ref{thm:StabNonorientableSurfaces} (\ref{it:NOStabMu}), we employ an idea of Wahl: in order to prove Theorem \ref{thm:StabNonorientableSurfaces} (\ref{it:NOStabAlpha}) and (\ref{it:NOStabBeta}) we may as well stabilise by gluing on countably-many projective planes, which then only requires us to show homological stability in all degrees for maps of type $\alpha$ and $\beta$ defined for ``infinite genus surfaces". We prove this by induction on homological degree (rather than genus), using similar techniques as in the previous section.

We first introduce some useful notation, and record a corollary of Proposition \ref{prop:AbsorptionProp}.

\begin{defn}
Let ${W} : (0,{Q}_{\sticks}) \leadsto (1,{Q}'_{\cross})$ be an elementary stabilisation map, and ${L}$ be an outer boundary condition. If $W \cup ([0,1] \times L) : {L}_{{Q}} \leadsto {L}_{{Q}'}$ is a stabilisation map of type $\mu$ then we write $\mu_{{L}}(g,-;W)$ for the pair of spaces given by the map
$$W_* : \mathcal{M}^\theta_{{L}}(g,-;0,Q) \lra \mathcal{M}^\theta_{{L}}(g+1,-;1,Q').$$

We similarly write $\alpha_{{L}}(g,-;W)$ or $\beta_{{L}}(g,-;W)$ if $W \cup ([0,1] \times L) : {L}_{{Q}} \leadsto {L}_{{Q}'}$ is a stabilisation map of type $\alpha$ or $\beta$.
\end{defn}

The following corollary is deduced from Proposition \ref{prop:AbsorptionProp}, using the fact that if $({W}, U)$ be a non-orientable test pair of height $k'$ and $\theta$ is $k'$-trivial for projective planes, then $U$ absorbs $W$.

\begin{cor}\label{cor:kTrivMeansZero2}
Let $({W}, U)$ be a non-orientable test pair of height $k'$. If the stabilisation map $W \cup ([0,1] \times L) : {L}_{{Q}} \leadsto {L}_{{Q}'}$ is of type $\mu$, then the induced map on relative homology is of the form
\begin{align*}
(U^*)_* : H_*(\mu_{{L}'}(g-k',-;W)) \lra H_*(\mu_{{L}}(g,-;W))
\end{align*}
and is zero in those homological degrees $*$ where all stabilisation maps
$$Y_*: H_{*-1}(\mathcal{M}^\theta_{{L}'}(g-k'-1,-;0,Q'')) \lra H_{*-1}(\mathcal{M}^\theta_{{L}'}(g-k',-;1,Q))$$
of type $\mu$ are surjective.
\end{cor}

\subsection{Proof of Theorem \ref{thm:StabNonorientableSurfaces} (\ref{it:NOStabMu})}

Let $H', Z' : \bZ \to \bZ$ be the functions given in Definition \ref{defn:Nonorientable:Range}, and consider the statement
\begin{equation}\tag{H$'_y$}
\text{For $g \leq y$, and all $W$ and $L$, } H_*(\mu_{L}(g,-;W)) = 0 \text{ in degrees } * \leq H'(g).
\end{equation}

This statement implies part (\ref{it:NOStabMu}) of Theorem \ref{thm:StabNonorientableSurfaces} (using Lemma \ref{lem:EltStab}). As in the last section, in order to prove this statement we must simultaneously prove another more technical statement. It has the same form as the more technical sattements of the previous section, but we give its definition in full. By Proposition \ref{prop:FibResOfPairs} (\ref{it:FibResOfPairs:ResMu}), when an elementary stabilisation map $W_* : \mathcal{M}^\theta_{{L}}(g,-;0,Q) \to \mathcal{M}^\theta_{{L}}(g+1,-;1,Q')$ of type $\mu$ is resolved (using resolution data $b : \{\pm1\} \times \bR \hookrightarrow L$), it gives a map
$$\mathcal{P}^\theta_{{L}}(g,-; 0,{Q};b)_0 \lra \mathcal{P}^\theta_{{L}}(g+1,-; 1,{Q}';b)_0$$
on $0$-simplices, which is a map over $A_0(0;b,\ell_b,-) \overset{\sim}\hookrightarrow A_0(1;b,\ell_b,-)$. On fibres over $x \in A_0(0;b,\ell_b,-)$, the map is homotopy equivalent to the elementary stabilisation map of type $\mu$
$$W_*: \mathcal{M}^\theta_{{L}_x}(g-1,-;0,{Q}) \lra \mathcal{M}^\theta_{{L}_x}(g,-;1,{Q}'),$$
where the new outer boundary condition ${L}_x$ depends on the point $x \in A_0(0;b,\ell_b)$ we are working over. Thus we obtain a map of pairs
$$\epsilon_x : \mu_{{L}_x}(g-1,-;W) \lra \mu_{{L}}(g,-;W).$$
By choosing one point $x$ in each path component of $A_0(0;b,\ell_b,-)$, and summing together the maps $\epsilon_x$ on homology, we obtain a map
\begin{equation}\label{eq:AugmentationMu}
\epsilon_* : \bigoplus_{\mathclap{[x] \in \pi_0(A_0(0;b,\ell_b,-))}} H_*(\mu_{{L}_x}(g-1,-;W)) \lra H_*(\mu_{{L}}(g,-;W)).
\end{equation}
The additional statement that we will simultaneously prove is
\begin{equation}\tag{Z$'_y$}
\text{For all $g \leq y$ and all $W$, $L$, and $b$,  \eqref{eq:AugmentationMu} is epi in degrees} \,\, * \leq Z'(g).
\end{equation}

By assumption, $\theta$ stabilises on $\pi_0$ for projective planes at genus $h'$, so we have $H_0(\mu_{{L}}(g,-;W))=0$ for all $g \geq h'-1$ and all $W$ and $L$, so certainly the statements $H'_{h'-1}$ and $Z'_{h'-1}$ hold, as the functions $H'$ and $Z'$ take the value 0 at $h'-1$ and $-1$ below $h'-1$. This starts our induction, and the inductive step is provided by the following technical theorem.

\begin{thm}\label{thm:InductiveStepMu}
Suppose the hypotheses of Theorem \ref{thm:StabNonorientableSurfaces} (\ref{it:NOStabMu}) hold, then there are implications
\begin{enumerate}[(i)]
	\item\label{it:thm:InductiveStepMu:1} $H'_{g-1}$ implies $Z'_g$,
	\item\label{it:thm:InductiveStepMu:2} $Z'_{g}$ and $H'_{g-k'-1}$ imply $H'_g$.
\end{enumerate}
\end{thm}
\begin{proof}[Proof of Theorem \ref{thm:InductiveStepMu} (\ref{it:thm:InductiveStepMu:1})]
We consider the semi-simplicial resolution of an elementary stabilisation map of type $\mu$,
\begin{equation*}
\xymatrix{
\mathcal{P}^\theta_{{L}}(g,-; 0,{Q};b)_\bullet \ar[r] \ar[d] & \mathcal{P}^\theta_{{L}}(g+1,-; 1,{Q}';b)_\bullet \ar[d]\\
\mathcal{M}^\theta_{{L}}(g,-;0,Q) \ar[r]^-{W_*}& \mathcal{M}^\theta_{{L}}(g+1,-;1,Q').
}
\end{equation*}

\vspace{2ex}

\noindent \textbf{Step 1}. We may study the map of $p$-simplices of these resolutions as in the proof of Theorem \ref{thm:InductiveStepOrientable} (\ref{it:thm:InductiveStepOrientable:1}) and (\ref{it:thm:InductiveStepOrientable:2}), where there is a relative Serre spectral sequence
\begin{align*}
E^2_{s,t} &= H_s(A_p(0;b, \ell_b,-) ; \underline{H_t}(\mu_{{L}_x}(g-p-1,-;W))) \\
&\Rightarrow H_{s+t}(\mathcal{P}^\theta_{{L}}(g+1,-; 1,{Q}';b)_p, \mathcal{P}^\theta_{{L}}(g,-; 0,{Q};b)_p).
\end{align*}
As we have supposed that $H'_{g-1}$ holds, it follows that $E^2_{s,t}=0$ in degrees $t \leq H'(g-p-1)$, and so
$$H_{t}(\mathcal{P}^\theta_{{L}}(g+1,-; 1,{Q}';b)_p, \mathcal{P}^\theta_{{L}}(g,-; 0,{Q};b)_p)=0 \quad \text{ for } t \leq H'(g-p-1).$$
In addition, for $p=0$ we find that the natural map
$$\bigoplus_{\mathclap{[x] \in \pi_0(A_0(0;b,\ell_b,-))}} H_t(\mu_{{L}_x}(g-1,-;W)) \lra H_t(\mathcal{P}^\theta_{{L}}(g+1,-; 1,{Q}')_0, \mathcal{P}^\theta_{{L}}(g,-; 0,{Q})_0)$$
is surjective for $t \leq H'(g-1)+1$.

\vspace{2ex}

\noindent \textbf{Step 2}. We now study the spectral sequence (\ref{SSRelativeAugmentedRestrictedSimplicialSpace}) for the map of augmented semi-simplicial spaces above, which takes the form
$$E^1_{p,q} = H_q(\mathcal{P}^\theta_{{L}}(g+1,-; 1,{Q}')_p, \mathcal{P}^\theta_{{L}}(g,-; 0,{Q})_p) \quad p \geq -1,\, q \geq 0.$$
It follows from Theorem \ref{thm:ResConnectivity1} that after geometric realisation the homotopy fibres of the two vertical maps are $(\lfloor \tfrac{g-1}{3} \rfloor-1)$- and $(\lfloor \tfrac{g}{3} \rfloor-1)$-connected respectively, and so this spectral sequence converges to zero in degrees $p+q \leq \lfloor \tfrac{g}{3} \rfloor-1$. We wish to draw a conclusion about the groups $E^1_{-1,q}$ for $q \leq Z'(g)$, but $Z'(g) \leq \lfloor \tfrac{g}{3} \rfloor$ by Definition \ref{defn:Nonorientable:Range}. Thus the groups we wish to study are in the range where the spectral sequence converges to zero.

By the first part of Step 1, we find that $E^1_{p,q}=0$ for $p \geq 0$ and $q \leq H'(g-p-1)$. By Definition \ref{defn:Nonorientable:Range} there is an inequality $Z'(g) \leq H'(g-2)+1$, and so in particular (as the inequality $Z'(n) \leq Z'(n-1)+1$ also holds) we have $E^1_{p,q}=0$ when $p > 0$ and $p+q \leq Z'(g)$. As $Z'(g)-1 \leq H'(g-2) \leq H'(g-1)$, we also have $E^1_{0,q}=0$ for $q \leq Z'(g)-1$. A chart of the $E^1$-page of the spectral sequence is shown in Figure \ref{fig:sseqNew2}.

\begin{figure}[h]
\centering
\includegraphics[bb = 0 0 158 120]{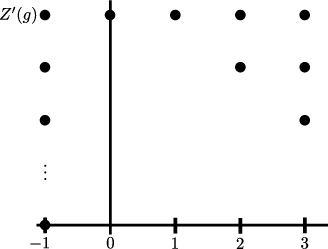}
\caption{}\label{fig:sseqNew2}
\end{figure}

Consequently, the only differentials which can land in $E^r_{-1,t}$ for $t \leq Z'(g)$ are $d^1 : E^1_{0,t} \to E^1_{-1,t}$, and because $E^\infty_{-1,t}=0$ for $t \leq Z'(g)$ it follows that these $d^1$-differentials must be surjective. Combining this with the last part of Step 1 shows that the composition
$$\bigoplus_{\mathclap{[x] \in \pi_0(A_0(0;b,\ell_b,-))}} H_{t}(\mu_{{L}_x}(g-1,-;W)) \lra E^1_{0, t} \overset{d^1}\lra E^1_{-1, t},$$
is surjective for $t \leq Z'(g)$, but this is precisely the map \eqref{eq:AugmentationMu}. This proves $Z'_g$.
\end{proof}

\begin{proof}[Proof of Theorem \ref{thm:InductiveStepMu} (\ref{it:thm:InductiveStepMu:2})]
Consider $W_* : \mathcal{M}^\theta_{{L}}(g,-;0,Q) \to \mathcal{M}^\theta_{{L}}(g+1,-;1,Q')$ an elementary stabilisation map of type $\mu$, and suppose that $L$ is standard on $\bB$. Choose embeddings $b_i : \{\pm1\} \times \bR \hookrightarrow L$ for $1 \leq i \leq k'$ disjoint from $\bB$, so that the data $(L, \{b_1, \ldots, b_k\})$ are arranged as shown in Figure \ref{fig:kTRivNO} on page \pageref{fig:kTRivNO}.

Proceeding precisely as in the proof of Theorem \ref{thm:InductiveStepOrientable} (\ref{it:thm:InductiveStepOrientable:3}) and (\ref{it:thm:InductiveStepOrientable:4}) using this data, we obtain a long composition
\begin{align*}
\bigoplus_{\mathclap{\substack{[x_1] \in \pi_0(A_0(0;b_1,\ell_{b_1},-)) \\ \vdots \\ [x_{k'}] \in \pi_0(A_0(0;b_{k'},\ell_{b_{k'}},-))}}}  H_*(\mu_{L_{x_1, \ldots,x_{k'}}}(g-k',-;W)) \lra \cdots \lra \bigoplus_{\mathclap{\substack{[x_1] \in \pi_0(A_0(0;b_1,\ell_{b_1},-)) \\ [x_2] \in \pi_0(A_0(0;b_2,\ell_{b_2},-))}}}  H_*(\mu_{L_{x_1, x_2}}(g-2,-;W)) \\
\lra \bigoplus_{\mathclap{[x_1] \in \pi_0(A_0(0;b_1,\ell_{b_1},-))}} H_*(\mu_{L_{x_1}}(g-1,-;W)) \lra H_*(\mu_{L}(g,-;W)).
\end{align*}
Each map in this composition is a direct sum of maps of type \eqref{eq:AugmentationMu}, so as we have assumed that $Z'_g$ holds it follows that all the maps are surjective in degrees $* \leq Z'(g-k'+1)$, and so by Definition \ref{defn:Nonorientable:Range} in degrees $* \leq H'(g)$.

As in the proof of Theorem \ref{thm:InductiveStepOrientable} (\ref{it:thm:InductiveStepOrientable:3}) and (\ref{it:thm:InductiveStepOrientable:4}), on each summand of the source this map is induced by gluing on an outer cobordism $U$, and by the pattern in which we chose the intervals $b_i$ (that of Figure \ref{fig:kTRivNO}), the pair $(W, U)$ is a non-orientable test pair of height $k'$. By Corollary \ref{cor:kTrivMeansZero2} it follows that the long composition is zero in those degrees $*$ such that all stabilisation maps of type $\mu$
$$Y_*: H_{*-1}(\mathcal{M}^\theta_{{L}'}(g-k'-1,-;0,Q'')) \lra H_{*-1}(\mathcal{M}^\theta_{{L}'}(g-k',-;1,Q))$$
are surjective. As we have assumed that $H'_{g-k'-1}$ holds, all such maps are surjective in degrees $*$ such that $*-1 \leq H'(g-k'-1)$, so in particular in degrees $* \leq H'(g)$ by Definition \ref{defn:Nonorientable:Range}. But then the long composition is both zero and surjective in degrees $* \leq H'(g)$, and so $H_*(\mu_{{L}}(g,-;W))=0$ in this range, as required.
\end{proof}

\subsection{Proof of Theorem \ref{thm:StabNonorientableSurfaces} (\ref{it:NOStabAlpha}) and (\ref{it:NOStabBeta})}

Consider an elementary stabilisation map of type $\alpha$, $W_* : \mathcal{M}_L^\theta(g,-;0,Q) \to \mathcal{M}_L^\theta(g+2,-;1,Q')$. Recall from the discussion preceeding Theorem \ref{thm:ResConnectivity2} that we may choose a sequence of composable stabilising cobordisms (in the sense of Definition \ref{defn:StabCob})
$$L = L_0 \overset{K_1}\leadsto L_1 \overset{K_1}\leadsto L_2 \leadsto \cdots,$$
where each $K_i$ contains $[-1,0] \times b(\{\pm 1\} \times \bR)$ as a $\theta$-submanifold, and is obtained up to diffeomorphism from $[-1,0] \times L_{i-1}$ by forming the connect-sum with $\bR\bP^2$ in a component which touches the image of the map $b$. As gluing on inner and outer cobordisms commute with each other, there is an induced ladder diagram
\begin{equation*}
\xymatrix{
\mathcal{M}_L^\theta(g,-;0,Q) \ar[r]^-{K_1^*} \ar[d]^-{W_*}& \mathcal{M}_{L_1}^\theta(g+1,-;-1,Q) \ar[r]^-{K_2^*}\ar[d]^-{W_*} & \mathcal{M}_{L_2}^\theta(g+2,-;-2,Q) \cdots\ar[d]^-{W_*}\\
\mathcal{M}_L^\theta(g+2,-;1,Q') \ar[r]^-{K_1^*}& \mathcal{M}_{L_1}^\theta(g+3,-;0,Q') \ar[r]^-{K_2^*} & \mathcal{M}_{L_2}^\theta(g+4,-;-1,Q') \cdots
}
\end{equation*}
and so a map on homotopy colimits, which fits into a commutative square
\begin{equation*}
\xymatrix{
\mathcal{M}_L^\theta(g,-;0,Q) \ar[r] \ar[d]^-{W_*}& \relax \underset{n \to \infty}\hocolim \mathcal{M}_{L_n}^\theta(g+n,-;-n,Q) \ar[d]^-{W_*} \\
\mathcal{M}_L^\theta(g+2,-;1,Q') \ar[r]& \relax \underset{n \to \infty}\hocolim \mathcal{M}_{L_n}^\theta(g+2+n,-;1-n,Q').
}
\end{equation*}
By Theorem \ref{thm:StabNonorientableSurfaces} (\ref{it:NOStabMu}) the horizontal maps are isomorphisms in homology in degrees $* \leq H'(g)-1$, so to show the left-hand vertical map is an isomorphism on homology in this range, it suffices to show that the right-hand vertical map is an isomorphism on homology (in principle only in degrees $* \leq H'(g)-1$, but we shall in fact show that it is an isomorphism in all degrees). The square above may be expressed as a map of pairs, from the left-hand pair to the right-hand one, as
$$\alpha_L(g,-;W) \lra \underset{n \to \infty}\hocolim \alpha_{L_n}(g+n,-;W),$$
so we consider the following statement
\begin{equation}\tag{F$'_y$}
\text{For all $t \leq y$, all $W$, and all sequences $K_i$, } \underset{n \to \infty} \colim H_t(\alpha_{L_n}(g+n,-;W)) = 0.
\end{equation}
Repeating the above discussion discussion for an elementary stabilisation map of type $\beta$, we also consider the statement
\begin{equation}\tag{G$'_y$}
\text{For all $t \leq y$, all $W$, and all sequences $K_i$, } \underset{n \to \infty} \colim H_t(\beta_{L_n}(g+n,-;W)) = 0.
\end{equation}
If the statements $F'_y$ and $G'_y$ hold for all $y$, then we have proved Theorem \ref{thm:StabNonorientableSurfaces} (\ref{it:NOStabAlpha}) and (\ref{it:NOStabBeta}); by Remark \ref{rem:BetaGamma}, part (\ref{it:NOStabGamma}) follows from part (\ref{it:NOStabBeta}), and so we have proved Theorem \ref{thm:StabNonorientableSurfaces}.

As usual, in order to prove the above statements we will consider a pair of auxiliary statements, which concern the maps
\begin{equation}\label{eq:AugmentationAlphaNO}
\epsilon_* : \underset{n \to \infty} \colim \bigoplus_{\mathclap{[x] \in \pi_0(A_0(0;b,\ell_b,+))}} H_*(\beta_{(L_n)_x}(g+n,-;W)) \lra \underset{n \to \infty} \colim H_*(\alpha_{L_n}(g+n,-;W))
\end{equation}
and
\begin{equation}\label{eq:AugmentationBetaNO}
\epsilon_* : \underset{n \to \infty} \colim \bigoplus_{\mathclap{[x] \in \pi_0(A_0(0;b,\ell_b,+))}} H_*(\alpha_{(L_n)_x}(g+n-2,-;W)) \lra \underset{n \to \infty} \colim H_*(\beta_{{L}_n}(g+n,-;W))
\end{equation}
obtained by taking the non-orientable versions of the maps \eqref{eq:AugmentationAlpha} and \eqref{eq:AugmentationBeta} and forming the limit over gluing on the outer cobordisms $K_i$. More precisely, for each $x \in A_0(0;b,\ell_b,+)$ there is a commutative diagram of solid maps
\begin{equation*}
\xymatrix{
\mathcal{M}_{(L_n)_x}^\theta(g+n,-;1-n,Q) \ar[d]^-\simeq \ar@{-->}[r]& \mathcal{M}_{(L_{n+1})_x}^\theta(g+n+1,-;-n,Q) \ar[d]^-\simeq\\
\pi_0^{-1}(x) \ar[d]  \ar[r]& \pi_0^{-1}(x) \ar[d]\\
\mathcal{B}_{L_n}^\theta(g+n,-;-n,Q;b)_0 \ar[d]^-\epsilon \ar[r]^-{(K_{n+1})^*}& \mathcal{B}_{L_{n+1}}^\theta(g+n+1,-;-n-1,Q;b)_0 \ar[d]^-\epsilon\\
\mathcal{M}_{L_n}^\theta(g+n,-;-n,Q) \ar[r]^-{(K_{n+1})^*}& \mathcal{M}_{L_{n+1}}^\theta(g+n+1,-;-n-1,Q)
}
\end{equation*}
and the colimit is formed from the relative version of this diagram, using the dashed map obtained by inverting the top right weak homotopy equivalence. The auxiliary statements are then as follows.
\begin{equation}\tag{X$'_y$}
\text{For all $* \leq y$, all $W$, and all sequences $K_i$, the maps \eqref{eq:AugmentationAlphaNO} are surjective.}
\end{equation}
\begin{equation}\tag{Y$'_y$}
\text{For all $* \leq y$, all $W$, and all sequences $K_i$, the maps \eqref{eq:AugmentationBetaNO} are surjective.}
\end{equation}

That the statements $F'_0$, $G'_0$, $X'_0$ and $Y'_0$ hold is immediate from stability of $\pi_0$ for non-orientable surfaces at genus $h$, as this implies that $H_0(\alpha_{L_n}(g+n,-;W))$ and $H_0(\beta_{L_n}(g+n,-;W))$ are both zero whenever $g+n \geq h$. 

\begin{prop}\label{prop:InductiveStepNonorientable}
Suppose the hypotheses of Theorem \ref{thm:StabNonorientableSurfaces} hold, then there are implications
\begin{enumerate}[(i)]
\item\label{it:prop:InductiveStepNonorientable:1} $G'_{y-1}$ implies $X'_{y}$,

\item\label{it:prop:InductiveStepNonorientable:2} $F'_{y-1}$ implies $Y'_{y}$,

\item\label{it:prop:InductiveStepNonorientable:3} $F'_{y-1}$ and $X'_y$ imply $F'_y$,

\item\label{it:prop:InductiveStepNonorientable:4} $G'_{y-1}$ and $Y'_y$ imply $G'_y$.
\end{enumerate}
\end{prop}

\begin{proof}[Proof of Proposition \ref{prop:InductiveStepNonorientable} (\ref{it:prop:InductiveStepNonorientable:1}) and (\ref{it:prop:InductiveStepNonorientable:2})]
Both statements will be proved identically, so let us consider the first for concreteness. Suppose we are given the data $W$, $b$, and $K_i$ necessary to express a map of the type \eqref{eq:AugmentationAlphaNO}. There is then an associated map of semi-simplicial resolutions
\begin{equation*}
\xymatrix{
\relax \underset{n \to \infty}\hocolim \mathcal{B}_{L_n}^\theta(g+n,-;-n,Q;b)_\bullet \ar[r] \ar[d] & \relax \underset{n \to \infty}\hocolim \mathcal{H}_{L_n}^\theta(g+n+2,-;1-n,Q';b)_\bullet \ar[d]\\
\relax \underset{n \to \infty}\hocolim \mathcal{M}_{L_n}^\theta(g+n,-;-n,Q) \ar[r] & \relax \underset{n \to \infty}\hocolim \mathcal{M}_{L_n}^\theta(g+n+2,-;1-n,Q').
}
\end{equation*}

\vspace{2ex}

\noindent \textbf{Step 1}. As in Step 1 of the proof of Theorem \ref{thm:InductiveStepOrientable}, we can understand the map of $p$-simplices of these resolutions using the map of Serre fibrations
\begin{equation*}
\xymatrix{
\mathcal{B}_{L_n}^\theta(g+n,-;-n,Q;b)_p \ar[d] \ar[r]& \mathcal{H}_{L_n}^\theta(g+n+2,-;1-n,Q';b)_p \ar[d]\\
A_p(-n;b,\ell_b,+) \ar[r]^-\simeq& A_p(1-n;b,\ell_b,+),
}
\end{equation*}
which on fibres is modelled by the maps
$$W_* : \mathcal{M}_{(L_n)_x}^\theta(g+n-2p,-;-n,Q) \lra \mathcal{M}_{(L_n)_x}^\theta(g+n+2-2(p+1),-;1-n,Q')$$
of type $\beta$. There is thus a relative Serre spectral sequence
\begin{align*}
E^2_{s,t}(n) &= H_s(A_p(-n;b, \ell_b,+) ; \underline{H_t}(\beta_{(L_n)_x}(g+n-2p,-;W))) \\
&\Rightarrow H_{s+t}(\mathcal{H}_{L_n}^\theta(g+n+2,-;1-n,Q';b)_p, \mathcal{B}_{L_n}^\theta(g+n,-;-n,Q;b)_p),
\end{align*}
and taking the colimit of these as $n \to \infty$ gives a spectral sequence
$$\underset{n \to \infty}\colim E^2_{s,t}(n) \Rightarrow \underset{n \to \infty}\colim H_{s+t}(\mathcal{H}_{L_n}^\theta(g+n+2,-;1-n,Q';b)_p, \mathcal{B}_{L_n}^\theta(g+n,-;-n,Q;b)_p).$$
As we have assumed that $G'_{y-1}$ holds, so $\underset{n \to \infty}\colim {H}_t(\beta_{(L_n)_x}(g+n-2p,-;W))=0$ for $t \leq y-1$, it follows that $\underset{n \to \infty}\colim E^2_{s,t}(n)=0$ for $t \leq y-1$, and so
\begin{equation*}
\underset{n \to \infty}\colim H_{*}(\mathcal{H}_{L_n}^\theta(g+n+2,-;1-n,Q';b)_p, \mathcal{B}_{L_n}^\theta(g+n,-;-n,Q;b)_p)=0 \quad \text{ for } * \leq y-1.
\end{equation*}
In the case $p=0$, we also deduce that the natural map
\begin{align*}
\underset{n \to \infty}\colim \bigoplus_{\mathclap{[x] \in \pi_0(A_0(0;b,\ell_b,+))}} & H_t(\beta_{(L_n)_x}(g+n,-;W))\\
& \lra \underset{n \to \infty}\colim H_{t}(\mathcal{H}_{L_n}^\theta(g+n+2,-;1-n,Q';b)_0, \mathcal{B}_{L_n}^\theta(g+n,-;-n,Q;b)_0)
\end{align*}
is surjective for $t \leq y$.

\vspace{2ex}

\noindent \textbf{Step 2}. Associated to the map of semi-simplicial resolutions above, we have a spectral sequence of type (\ref{SSRelativeAugmentedRestrictedSimplicialSpace}) which takes the form
$$E^1_{p,q} = \underset{n \to \infty}\colim H_{q}(\mathcal{H}_{L_n}^\theta(g+n+2,-;1-n,Q';b)_p, \mathcal{B}_{L_n}^\theta(g+n,-;-n,Q;b)_p)$$
for $p \geq -1$ and $q \geq 0$. By Theorem \ref{thm:ResConnectivity2} the two vertical maps have contractible homotopy fibres after geometric realisation, and so this spectral sequence converges to zero in all degrees, that is, $E^\infty_{p,q}=0$ for all $p$ and $q$. By the first part of Step 1, $E^1_{p,q}=0$ for $p \geq 0 $ and $q \leq y-1$, so the $E^1$-page of the spectral sequence is as shown in Figure \ref{fig:sseqNew3}.

\begin{figure}[h]
\centering
\includegraphics[bb = 0 0 149 120]{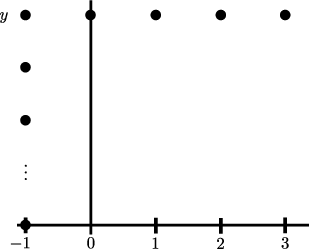}
\caption{}\label{fig:sseqNew3}
\end{figure}

From this we deduce that the differential $d^1 : E^1_{0,t} \to E^1_{-1,t}$ is surjective for $t \leq y$, and combining this with the second part of Step 1 shows that
$$\underset{n \to \infty}\colim \bigoplus_{\mathclap{[x] \in \pi_0(A_0(0;b,\ell_b,+))}}  H_t(\beta_{(L_n)_x}(g+n,-;W)) \lra \underset{n \to \infty}\colim H_t(\alpha_{L_n}(g+n,-;W))$$
is surjective for $t \leq y$, as required.
\end{proof}

For the second half of Proposition \ref{prop:InductiveStepNonorientable} we require a result analogous to Corollary \ref{cor:kTrivMeansZero} but in the non-orientable and stable setting.

\begin{lem}\label{lem:AugIsZero}
Suppose that $\theta$ stabilises on $\pi_0$ for non-orientable surfaces at genus $h$. Then
\begin{enumerate}[(i)]
\item if $F'_{y-1}$ holds, then the map \eqref{eq:AugmentationAlphaNO} is zero in degrees $* \leq y$,

\item if $G'_{y-1}$ holds, then the map \eqref{eq:AugmentationBetaNO} is zero in degrees $* \leq y$.
\end{enumerate}
\end{lem}
\begin{proof}
Let us consider the first case. As $F'_{y-1}$ is assumed to hold, and we have proved homological stability for stabilisation by projective planes, there is a $g' \geq 0$ such that every elementary stabilisation map of type $\alpha$,
$$Y_* : H_{t}(\mathcal{M}^\theta_{L_x}(g-2,-;0,Q'')) \lra H_{t}(\mathcal{M}^\theta_{L_x}(g,-;1,Q))$$
is surjective in degrees $t \leq y-1$ for every $g \geq g'$.

To show the map \eqref{eq:AugmentationAlphaNO} is zero in degrees $* \leq y$, it is enough to show that the composition
$$H_*(\beta_{L_x}(g,-;W)) \lra H_*(\alpha_{L}(g,-;W))\lra  H_*(\alpha_{L_{h+1}}(g+h+1,-;W))$$
is zero in degrees $* \leq y$ for every $g \geq g'$. Here the first map is induced by the outer cobordism $R_x : L_x \leadsto L$ associated to an element $[x] \in \pi_0(A_0(0;b,\ell_b,+))$ and the second map is induced by the composition of the outer cobordisms $K_i$ for $1 \leq i \leq h+1$. Let us write $U : L_x \leadsto L_{h+1}$ for the composition of these cobordisms.

We claim that $U$ absorbs $W$. We shall show this using the technique introduced in Section \ref{sec:FormalKTriviality}, which was used in the proof of Propositions \ref{prop:FormalkTriv} and \ref{prop:FormalkTrivNO} to deduce $k$-triviality from stabilisation on $\pi_0$. The relevant diagram in this case is 
\begin{equation*}
\xymatrix{
\mathcal{M}^\theta(h-1,-;\bigcirc \bigcirc) \ar[d]^-{\text{type } \alpha} \ar[r]^-{\text{type } \alpha}& \mathcal{M}^\theta(h+1,-;\bigcirc) \ar[d]\ar[d]^-{\text{type } \beta}\\
\mathcal{M}^\theta(h+1,-; \bigcirc) \ar[r]^-{\text{type } \beta}& \mathcal{M}^\theta(h+1,-;\bigcirc \bigcirc),
}
\end{equation*}
and we require the top map so be surjective on $\pi_0$ and the bottom map to be injective on $\pi_0$, but this implied by stabilisation on $\pi_0$ at genus $h$. Hence $U$ absorbs $W$.

We now apply Proposition \ref{prop:AbsorptionProp} to the pair $(W, U)$. Restricting to certain path components, it follows from that proposition that the map
$$(U^*)_* : H_*(\beta_{L_x}(g,-;W)) \lra H_*(\alpha_{L_{h+1}}(g+h+1,-;W))$$
is zero in those degrees $*$ where all elementary stabilisation maps of type $\alpha$
$$Y_* : H_{*-1}(\mathcal{M}^\theta_{L_x}(g-2,-;0,Q'')) \lra H_{*-1}(\mathcal{M}^\theta_{L_x}(g,-;1,Q))$$
are surjective. By our assumption that $g \geq g'$, any such map is surjective in degrees $*-1 \leq y-1$, i.e.\ for $* \leq y$. It follows that
$$(U^*)_* : H_*(\beta_{L_x}(g,-;W)) \lra H_*(\alpha_{L_{h+1}}(g+h+1,-;W))$$
is zero in degrees $* \leq y$, as required.
\end{proof}

\begin{proof}[Proof of Proposition \ref{prop:InductiveStepNonorientable} (\ref{it:prop:InductiveStepNonorientable:3}) and (\ref{it:prop:InductiveStepNonorientable:4})]
Both statements are proved in the same way, so for concreteness we prove (\ref{it:prop:InductiveStepNonorientable:3}), that $F'_{y-1}$ and $X'_y$ imply $F'_y$. By Lemma \ref{lem:AugIsZero}, as $F'_{y-1}$ holds, the map \eqref{eq:AugmentationAlphaNO} is zero in degrees $* \leq y$. As $X'_y$ holds, the same map is surjective in degrees $* \leq y$. Hence, the target of \eqref{eq:AugmentationAlphaNO} is zero in degrees $* \leq y$.
\end{proof}

\section{Closing the last boundary}\label{sec:LastBoundary}

In order to prove homology stability of $\mathcal{M}^\theta(F)$ for closed surfaces $F$, we cannot use resolutions constructed in terms of arcs with ends on the boundary of a surface, as we have no boundary. Instead, for any surface $F$ (with or without boundary) we will define a new resolution using discs in the surface, and resolve $\mathcal{M}^\theta(F;\ell_{\partial F})$ by moduli spaces of surfaces of the same genus but strictly more boundary components than $F$. The resolution is quite general, and in fact exists for manifolds of any dimension and having any tangential structure. Thus until Section \ref{sec:StabLastBoundarySurface} we work with manifolds of arbitrary dimension.

\subsection{Orientable and non-orientable manifolds}

There are a few differences between the cases of orientable and non-orientable manifolds, which we will deal with before starting.

Fix a connected $d$-manifold $M$, possibly with boundary, and a map $\theta : B \to BO(d)$, and let $\mathcal{M}^\theta(M ; \ell_{\partial M})$ be the moduli space of $\theta$-manifolds with underlying manifold diffeomorphic to $M$ and boundary condition $\ell_{\partial M}$, defined just as in Definition \ref{defn:BorelConstModel}. \emph{If $M$ is orientable we will choose an orientation $\omega_M$, and we will also assume that $\theta^*\gamma_d$ is orientable and choose once and for all an orientation $\omega_\theta$}.

\begin{enumerate}[(i)]

\item\label{it:Orientability:1} If $M$ is orientable and admits an orientation-reversing diffeomorphism (so $\partial M = \emptyset$), let $\Diff^+_\partial(M)$ be the index 2 subgroup of diffeomorphisms preserving $\omega_M$, and $\Bun^+_\partial(TM, \theta^*\gamma_d) \subset \Bun_\partial(TM, \theta^*\gamma_d)$ be the subspace of those bundle maps which on each fibre sends the orientation $\omega_M$ to $\omega_\theta$.

\item\label{it:Orientability:2} If $M$ is orientable and does not admit an orientation-reversing diffeomorphism (e.g.\ if $\partial M \neq \emptyset$), then we write $\Bun^+_\partial(TM, \theta^*\gamma_d;\ell_{\partial M})$ and $\Diff^+_\partial(M)$ for the entire space of bundle maps and the entire group of diffeomorphisms.

\item\label{it:Orientability:3} If $M$ is non-orientable, then to unify notation we write $\Bun^+_\partial(TM, \theta^*\gamma_d;\ell_{\partial M})$ and $\Diff^+_\partial(M)$ for the entire space of bundle maps and the entire group of diffeomorphisms.
\end{enumerate}

\begin{lem}
The map
$$\Bun^+_\partial(TM, \theta^*\gamma_d;\ell_{\partial M}) \hcoker \Diff^+_\partial(M) \lra \Bun_\partial(TM, \theta^*\gamma_d;\ell_{\partial M}) \hcoker \Diff_\partial(M)$$
is a homotopy equivalence.
\end{lem}
\begin{proof}
In cases (\ref{it:Orientability:2}) and (\ref{it:Orientability:3}) above there is nothing to show. In case (\ref{it:Orientability:1}), precomposing with the differential of the orientation-reversing diffeomorphism gives a map
$$\Bun^+_\partial(TM, \theta^*\gamma_d) \lra \Bun_\partial(TM, \theta^*\gamma_d)$$
which is a homeomorphism onto the complement of $\Bun^+_\partial(TM, \theta^*\gamma_d)$. Thus we may identify
$$\Bun_\partial(TM, \theta^*\gamma_d)=\Bun^+_\partial(TM, \theta^*\gamma_d) \times \{\text{orientations of } M\}$$
as a $\Diff(M)$-space, and the claim follows.
\end{proof}

By this lemma, we may as well define $\mathcal{M}^\theta(M;\ell_{\partial M})$ using $\Bun^+_\partial(TM, \theta^*\gamma_d;\ell_{\partial M})$ and $\Diff^+_\partial(M)$, and from now on we shall do so.

\subsection{The disc resolution}

Let us write $[p]$ for the standard ordered set $\{0 < 1 < \cdots < p\}$, which is an object of the simplicial category $\Delta$. Let
$$D(M)_p \subset \Emb\left([p] \times D^d, \mathring{M} \right)$$
be the subspace of those embeddings into the interior of $M$ which, if $M$ is orientable (and hence has a given orientation $\omega_M$ by our conventions), restrict to orientation-preserving embeddings of each $\{i\} \times D^d$. (If $M$ is non-orientable then $D(M)_p$ is the whole space of embeddings.) Define
$$\mathcal{D}^\theta(M;\ell_{\partial M})_p := (D(M)_p \times \Bun^+_\partial(TM, \theta^*\gamma_d;\ell_{\partial M})) \hcoker \Diff^+_\partial(M)$$
where the group acts diagonally. The map $d_j : D(M)_p \to D(M)_{p-1}$ induced by the unique strictly monotonic map $[p-1] \to [p]$ which misses $j$ induces a map $d_j : \mathcal{D}^\theta(M;\ell_{\partial M})_i \to \mathcal{D}^\theta(M;\ell_{\partial M})_{i-1}$, and there are maps $\mathcal{D}^\theta(M;\ell_{\partial M})_i \to \mathcal{M}^\theta(M;\ell_{\partial M})$ that forget all the discs.

\begin{prop}\label{prop:DiscResolution}
This data makes $\mathcal{D}^\theta(M;\ell_{\partial M})_\bullet \to \mathcal{M}^\theta(M;\ell_{\partial M})$ into an augmented semi-simplicial space, and a resolution.
\end{prop}
\begin{proof}
It is clear that the data defines an augmented semi-simplicial space, so we must show that $\vert \mathcal{D}^\theta(M;\ell_{\partial M})_\bullet\vert \to \mathcal{M}^\theta(M;\ell_{\partial M})$ is a homotopy equivalence. By Lemma \ref{lem:AugTriv} the homotopy fibre of this map is  $\vert D(M)_\bullet \vert$, which we must then show is contractible.

The semi-simplicial space $D(M)_\bullet$ is constructed from spaces of embeddings of disjoint discs in $M$, but it is convenient to pass to an equivalent semi-simplicial space whose space of $p$-simplices is the space of $(p+1)$ distinct points of $M$ each equipped with a framing of the tangent space of $M$ at that point (if $M$ is orientable, we insist that the framing is compatible with the chosen orientation $\omega_M$). One may see that this is an equivalent semi-simplicial space by the fibration sequence over the total space of the frame bundle of $M$,
$$* \simeq \mathrm{Fib} \lra \Emb({D}^d, M^d) \overset{\pi}\lra \mathrm{Fr}(M)$$
where $\pi$ is the map sending an embedding $e$ to the image under $De$ of the standard frame at $0$. That this map has contractible fibres follows by a scanning argument, similar to that which shows that the diffeomorphism group of an open disc is homotopy equivalent to the general linear group. Call this semi-simplicial space $F(M)_\bullet$.

Suppose first that $M$ has a non-empty boundary component $\partial_0 M$, choose a Riemannian metric on $M$ and let $F_\epsilon(M)_\bullet \subset F(M)_\bullet$ denote the subspace consisting of configurations which have no point within $\epsilon$ of $\partial_0 M$. If $\epsilon$ is small enough, the inclusion $F_\epsilon(M)_\bullet \hookrightarrow F(M)_\bullet$ is a levelwise homotopy equivalence, and so induces a homotopy equivalence on geometric realisation. However, adding a new framed point $y$ inside the $\epsilon$-neighbourhood of $\partial_0 M$ first to the list of framed points gives a semi-simplicial nullhomotopy $s_{-1} : F_\epsilon(M)_p \to F(M)_{p+1}$ of the inclusion $F_\epsilon(M)_\bullet \hookrightarrow F(M)_{\bullet}$, and hence $\vert F(M)_\bullet\vert$ is contractible.

Now suppose that $M$ is a $d$-manifold without boundary, let $D^d \hookrightarrow M$ be an embedded closed disc and $\bar{M} = M - \mathrm{int}({D}^d)$ be the complement of the interior. Then the inclusion $\bar{M} \to M$ induces an inclusion $F(\bar{M})_\bullet \to F(M)_\bullet$. In simplicial degree $p$ we can identify the homotopy cofibre as
$$\bigvee_{j=0}^i F(\bar{M})_{p-1} \ltimes (\mathrm{Fr}(M)|_{D^d} / \mathrm{Fr}(M)|_{\partial D^d}),$$
where we remind the reader that for a space $X$ and a pointed space $(Y, y_0)$, the half-smash product is the pointed space $X \ltimes Y := (X \times Y)/(X \times \{y_0\})$.

As $\vert F(\bar{M})_\bullet \vert \simeq *$, it will be enough to show that this homotopy cofibre has contractible geometric realisation. For a semi-simplicial space $X_\bullet$, let us define a pointed semi-simplicial space $(X_{\bullet-1} \times [\bullet])_+$ to have space of $p$-simplices given by $(X_{p-1} \times [p])_+$ (the subscript $+$ denotes the addition of a disjoint basepoint, which we call $*$), and face maps given (on points other than $*$) by the formula
\begin{align*}
d_j : (X_{p-1} \times [p])_+ &\lra (X_{p-2} \times [p-1])_+\\
(x, i) & \longmapsto \begin{cases}
(d_j(x), i-1) & j < i\\
* & j=i\\
(d_{j-1}(x), i) & j > i.
\end{cases}
\end{align*}
One may easily check that this defines a semi-simplicial space.

Using this construction, an alternative description of the homotopy cofibre above is the semi-simplicial space
$$(F(\bar{M})_{\bullet-1} \times [\bullet])_+ \wedge (\mathrm{Fr}(M)|_{D^d} / \mathrm{Fr}(M)|_{\partial D^d}),$$
so it will be enough to show that $(F(\bar{M})_{\bullet-1} \times [\bullet])_+$ has contractible geometric realisation. But $\bar{M}$ has non-empty boundary, so as in the previous case the inclusion
$$(F_\epsilon(\bar{M})_{\bullet-1} \times [\bullet])_+ \hookrightarrow (F(\bar{M})_{\bullet-1} \times [\bullet])_+$$
is a levelwise homotopy equivalence if $\epsilon$ is small enough, but it also has a simplicial contraction given (on points other than $*$) by the formula 
$$s_{-1}(x, i) = ((y, x), i+1),$$
where $y$ is a chosen framed point in the $\epsilon$-neighbourhood of $\partial \bar{M}$.
\end{proof}

We now wish to relate the spaces $\mathcal{D}^\theta(M;\delta)_i$ to $\mathcal{M}^\theta(M -\{(i+1) \,\,\text{discs}\})$. There is a $\Diff^+_\partial(M)$-equivariant map
$$\tilde{\pi} : D(M)_p \times \Bun^+_\partial(TM, \theta^*\gamma_d ;\ell_{\partial M}) \lra \Bun(T([p]\times D^d), \theta^*\gamma_d) =: B^\theta_p$$
given by sending $(e : [p] \times D^d \hookrightarrow M, \xi)$ to $e^*(\xi)$, and this descends to a map
$$\pi : \mathcal{D}^\theta(M)_p \lra B^\theta_p.$$

If we fix an $e \in D(M)_p$ and let $M \setminus e := M \setminus e([p] \times \mathring{D}^d)$, then a choice of $\xi \in B^\theta_p$ induces a boundary condition $\ell_{\partial M} \cup \xi$ on $M \setminus e$ and we have a map
$$i: \mathcal{M}^\theta(M\setminus e, \ell_{\partial M} \cup \xi) \lra \mathcal{D}^\theta(M)_p$$
induced by the inclusion $\Diff_\partial(M\setminus e) \subset \Diff_\partial^+(M)$ and the map $\Bun_\partial(T(M\setminus e), \theta^*\gamma_d; \ell_{\partial M} \cup \xi) \to \Bun_\partial(TM, \theta^*\gamma_d ; \ell_{\partial M} )$ defined by gluing in the $\theta$-manifold $([p] \times D^d, \xi)$.

\begin{prop}\label{prop:FibrationDiscResolution}
The maps
$$\mathcal{M}^\theta(M\setminus e, \ell_{\partial M} \cup \xi) \overset{i}\lra \mathcal{D}^\theta(M)_p \overset{\pi}\lra B^\theta_p$$
form a homotopy fibre sequence.
\end{prop}
\begin{proof}
The map $\tilde{\pi}$ is easily seen to be a Serre fibration, as for a fixed $e \in D(M)_p$ the restriction map
\begin{equation*}
\ell  \mapsto e^*(\ell) :\Bun^+_\partial(TM, \theta^*\gamma_d ;\ell_{\partial M}) \lra \Bun(T([p]\times D^d), \theta^*\gamma_d)
\end{equation*}
is a Serre fibration, because $e$ is a cofibration. The map $\pi$ is obtained from the map $\tilde{\pi}$ by forming the Borel construction for the action of $\Diff^+_\partial(M)$ on the source. The fibre of $\tilde{\pi}$ over $\xi$ is the space
$$F := \{(f, \ell) \in D(M)_p \times \Bun_\partial(TM, \theta^*\gamma_d ;\ell_{\partial M}) \,\, \vert \,\, f^*(\ell)=\xi\},$$
and the commutative diagram
\begin{equation*}
\xymatrix{
{F} \ar[r] \ar[d]& {F \hcoker \Diff^+_\partial(M)} \ar[r] \ar[d]& {B\Diff^+_\partial(M)} \ar@{=}[d]\\
{D(M)_p \times \Bun^+_\partial(TM, \theta^*\gamma_d ;\ell_{\partial M})}  \ar[r] \ar[d]^-{\tilde{\pi}}& {\mathcal{D}^\theta(M)_p} \ar[r] \ar[d]^-\pi& {B\Diff^+_\partial(M)} \ar[d]\\
{B_p^\theta} \ar@{=}[r]& {B_p^\theta} \ar[r]& {*}
}
\end{equation*}
has the outer columns and all rows homotopy fibre sequences: thus the middle column is also a homotopy fibre sequence.

The group $\Diff_\partial^+(M)$ acts transitively on $D(M)_p$, and the stabiliser of the fixed embedding $e$ is the subgroup $\Diff_\partial(M \setminus e)$. Thus the map
$$\{\ell \in \Bun_\partial(TM, \theta^*\gamma_d ;\ell_{\partial M}) \,\, \vert\,\, e^*(\ell)=\xi\} \hcoker \Diff_\partial(M\setminus e) \lra F \hcoker \Diff_\partial^+(M)$$
is a homotopy equivalence, but
$$\{\ell \in \Bun_\partial(TM, \theta^*\gamma_d ;\ell_{\partial M}) \,\, \vert\,\, e^*(\ell)=\xi\} \cong \Bun_\partial(T(M\setminus e), \theta^*\gamma_d; \ell_{\partial M} \cup \xi),$$
which proves the claim.
\end{proof}

\subsection{Stability for closing the last boundary}\label{sec:StabLastBoundarySurface}

Let us return now to the case of surfaces. Let
$$W_* : \mathcal{M}^\theta(g, \pm;P) \lra \mathcal{M}^\theta(g, \pm;P')$$
be a stabilisation map of type $\gamma$, so $W$ has a handle structure relative to $P$ consisting of a single 2-handle. We wish to understand the effect of this map on homology. If $P' \neq \emptyset$ then by Remark \ref{rem:BetaGamma} we know a stability range for this stabilisation map, by Theorem \ref{thm:StabOrientableSurfaces} (\ref{it:OStabGamma}) or Theorem \ref{thm:StabNonorientableSurfaces} (\ref{it:NOStabGamma}) in this case. The following theorem gives a stability range when $P' = \emptyset$, and finishes the proofs of Theorems \ref{thm:StabOrientableSurfaces} and \ref{thm:StabNonorientableSurfaces}.

\begin{thm}\label{thm:LastBoundary}
Suppose that $\theta$ satisfies the assumptions of Theorem \ref{thm:StabOrientableSurfaces} and $G : \bZ \to \bZ$ is the function appearing in that theorem. Then any stabilisation map $W_* : \mathcal{M}^\theta(g,+;{P}) \to \mathcal{M}^\theta(g,+;{P}')$ of type $\gamma$ induces an epimorphism in homology in degrees $* \leq G(g)+1$ and an isomorphism in homology in degrees $* \leq G(g)$.

Similarly, suppose that $\theta$ satisfies \emph{all} the assumptions of Theorem \ref{thm:StabNonorientableSurfaces} and  $H' : \bZ \to \bZ$ is the function appearing in that theorem.  Then any stabilisation map $W_* : \mathcal{M}^\theta(g,-;{P}) \to \mathcal{M}^\theta(g,-;{P}')$ of type $\gamma$ induces an epimorphism in homology in degrees $* \leq H'(g)$ and an isomorphism in homology in degrees $* \leq H'(g)-1$.
\end{thm}
\begin{proof}
Let us treat the orientable case: the non-orientable case is identical. Let us also work with the Borel construction model of Section \ref{sec:BorelConst}, where we need only consider the stabilisation maps
$$\mathcal{M}^\theta(\Sigma_{g,1};\ell_{\partial \Sigma_{g,1}}) \lra \mathcal{M}^\theta(\Sigma_{g})$$
defined using our standard model surfaces. Such a stabilisation map  induces a simplicial map
$$\mathcal{D}^\theta(\Sigma_{g,1};\ell_{\partial \Sigma_{g,1}})_\bullet \lra \mathcal{D}^\theta(\Sigma_g)_\bullet$$
on disc resolutions, and we study the associated map of spectral sequences (\ref{SSRelativeRestrictedSimplicialSpace}). These are
$$E^1_{p,q}(\Sigma_{g,1}) = H_q(\mathcal{D}^\theta(\Sigma_{g,1};\ell_{\partial \Sigma_{g,1}})_p) \Longrightarrow H_{p+q}(\mathcal{M}^\theta(\Sigma_{g,1};\ell_{\partial \Sigma_{g,1}}))$$
and 
$$E^1_{p,q}(\Sigma_g) = H_q(\mathcal{D}^\theta(\Sigma_g)_p) \Longrightarrow H_{p+q}(\mathcal{M}^\theta(\Sigma_g)).$$

The induced map on $E^1$-pages of these spectral sequences may be studied via the map of Serre spectral sequences for the homotopy fibre sequences given in Proposition \ref{prop:FibrationDiscResolution},
\begin{diagram}
H_s(B^\theta_p; \underline{H_t}(\mathcal{M}^\theta(\Sigma_{g, p+2};\ell_{\partial \Sigma_{g,1}} \cup \xi))) & \rTo & H_s(B_p^\theta; \underline{H_t}(\mathcal{M}^\theta(\Sigma_{g, p+1};\xi)))\\
\dImplies & & \dImplies\\
H_{s+t}(\mathcal{D}^\theta(\Sigma_{g,1};\ell_{\partial \Sigma_{g,1}})_p) & \rTo & H_{s+t}(\mathcal{D}^\theta(\Sigma_g)_p).
\end{diagram}
By the case which has already been proved of Theorem \ref{thm:StabOrientableSurfaces} (\ref{it:OStabGamma}), the map of local coefficient systems over $B_p^\theta$
$$\underline{H_t}(\mathcal{M}^\theta(\Sigma_{g, p+2};\ell_{\partial \Sigma_{g,1}} \cup \xi)) \lra \underline{H_t}(\mathcal{M}^\theta(\Sigma_{g, p+1};\xi))$$
is an isomorphism for $q \leq G(g)$ and an epimorphism in all degrees. Thus the map on $E^\infty$-pages is an isomorphism in total degrees $* \leq G(g)$ and an epimorphism in total degrees $* \leq G(g)+1$.

This shows that the map $E^1_{s,t}(\Sigma_{g,1}) \to E^1_{s,t}(\Sigma_g)$ is an isomorphism for $t \leq G(g)$ and an epimorphism for $t \leq G(g)+1$. This implies that the map on $E^\infty$-pages is an isomorphism in total degree $* \leq G(g)$ and an epimorphism in total degree $* \leq G(g)+1$ as required.
\end{proof}

\begin{rem}\label{rem:Orient2}
Let us explain the necessity of the assumption that $\theta^*\gamma_2$ be orientable when showing that stabilisation maps which close off the last boundary exhibit homological stability.

First consider the trivial tangential structure $\theta = \mathrm{Id} : BO(2) \to BO(2)$, which does not satisfy this assumption. In this case $\Bun_\partial(TF, \gamma_2 ;\ell_{\partial F}) \simeq *$ for any surface and any boundary condition, and so $\mathcal{M}^\theta(F) \simeq B\Diff_\partial(F)$. Now consider the tangential structure $\theta^+: BSO(2) \to BO(2)$. In this case
$$\Bun_\partial(TF, (\theta^+)^*\gamma_2 ;\ell_{\partial F}) \simeq \begin{cases}
\emptyset & \text{$F$ is not orientable compatibly with $\ell_{\partial F}$,}\\
* & \text{$F$ is orientable compatibly with $\ell_{\partial F}$, and $\partial F \neq \emptyset$}\\
\bZ/2 & \text{$F$ is orientable and has empty boundary.}
\end{cases}$$
Thus for $F$ orientable with non-empty boundary, and $\ell_{\partial F}$ an orientation which extends to $F$, the map $\mathcal{M}^{\theta^+}(F;\ell_{\partial F}) \to \mathcal{M}^{\theta}(F)$ is an equivalence, but for $F$ orientable with empty boundary the map $\mathcal{M}^{\theta^+}(F) \to \mathcal{M}^{\theta}(F)$ is a double cover (up to homotopy). In particular, $\theta$ and $\theta^+$ cannot both have stability for closing the last boundary. In fact, $\theta^+$ satisfies the conditions of Theorem \ref{thm:LastBoundary} and does have stability, but $\theta$ does not.

More generally, any $\theta : \X \to BO(2)$ may be pulled back to $BSO(2)$ to give a new tangential structure $\theta^+$ which is orientable. On surfaces with boundary these yield homotopy equivalent moduli spaces, but on closed surface sthey do not. This shows that the assumption that $\theta^*\gamma_2$ be orientable may be omitted if we only consider surfaces with non-empty boundary.
\end{rem}

\section{Stable homology}\label{sec:StableHomology}

Once we have established homology stability for a tangential structure $\theta : \X \to BO(2)$, the methods of Galatius--Madsen--Tillmann--Weiss \cite{GMTW} identify the stable homology with the homology of certain path components of the infinite loop space of the Thom spectrum
$$\mathbf{MT\theta} := \mathbf{Th}(-\theta^*\gamma_2 \to \X).$$
For surfaces with boundary, we can form the direct limit of
$$H_*(\mathcal{M}^\theta(\Sigma_{g,b})) \lra H_*(\mathcal{M}^\theta(\Sigma_{g+1,b})) \lra H_*(\mathcal{M}^\theta(\Sigma_{g+2,b})) \lra \cdots$$
over gluing on elements of $\mathcal{M}^\theta(\Sigma_{1, 1+1})$, and in the case of non-orientable surfaces we can form the direct limit of
$$H_*(\mathcal{M}^\theta(S_{n,b})) \lra H_*(\mathcal{M}^\theta(S_{n+1,b})) \lra H_*(\mathcal{M}^\theta(S_{n+2,b})) \lra \cdots$$
over gluing on elements of $\mathcal{M}^\theta(S_{1, 1+1})$. Let us write $H_*^{\text{stab}}(\mathcal{M}^\theta)$ for this direct limit: it is a consequence of homological stability that the direct limit does not depend on precisely which maps we use to form it. For each surface $F$ with non-empty boundary there is a map
$$\mathscr{S}_F : H_*(\mathcal{M}^\theta(F;\ell_{\partial F})) \lra H_*^{\text{stab}}(\mathcal{M}^\theta)$$
given by stabilisation.

For closed surfaces we cannot stabilise in this way, but there are nontheless natural maps
$$\mathcal{M}^\theta(F) \lra \Omega^\infty \mathbf{MT\theta},$$
given by the Pontrjagin--Thom construction: see \cite{GMTW, MT} for details of the construction of this map. Let us denote by $\Omega_{[F]}^\infty \mathbf{MT\theta}$ the collection of path components this map hits. On homology we obtain a map
$$\mathscr{S}_F : H_*\left(\mathcal{M}^\theta(F)\right) \lra H_*\left(\Omega_{[F]}^\infty \mathbf{MT\theta}\right),$$
and $H_*(\Omega_{[F]}^\infty \mathbf{MT\theta}) \cong H_*^{\text{stab}}(\mathcal{M}^\theta)$ by \cite{GMTW}. The following theorem describes the range in which these maps are isomorphisms, given the assumptions and notation of Theorems \ref{thm:StabOrientableSurfaces} and \ref{thm:StabNonorientableSurfaces}.

\begin{thm}
Suppose that $\theta$ is $k$-trivial and stabilises for orientable surfaces at genus $h$, and let $F, G : \bZ \to \bZ$ be given by Definition \ref{defn:Orientable:Range}. Then the map
$$\mathscr{S}_{\Sigma_{g,b}} : H_*(\mathcal{M}^\theta(\Sigma_{g,b};\ell_{\partial{\Sigma_{g,b}}})) \lra H_*^{\text{stab}}(\mathcal{M}^\theta)$$
is an isomorphism in degrees $* \leq \min(G(g), F(g)-1)$ (for $b=0$ we must in addition assume that $\theta^*\gamma_2$ is orientable). 

Similarly, suppose that $\theta$ is stabilises for non-orientable surfaces at genus $h$, and is $k'$-trivial for projective planes and stabilises for projective planes at genus $h'$. Let $H' : \bZ \to \bZ$ be given by Definition \ref{defn:Nonorientable:Range}. Then the map
$$\mathscr{S}_{S_{g,b}} : H_*(\mathcal{M}^\theta(S_{g,b};\ell_{\partial{S_{g,b}}})) \lra H_*^{\text{stab}}(\mathcal{M}^\theta)$$
is an isomorphism in degrees $* \leq H'(g)-1$.
\end{thm}
\begin{proof}
In the orientable case, if $b > 0$ we can form $H_*^{\text{stab}}(\mathcal{M}^\theta)$ as the direct limit of
$$H_*(\mathcal{M}^\theta(\Sigma_{g,b})) \overset{\text{type }\beta}\lra H_*(\mathcal{M}^\theta(\Sigma_{g,b+1})) \overset{\text{type }\alpha}\lra H_*(\mathcal{M}^\theta(\Sigma_{g+1,b})) \overset{\text{type }\beta}\lra \cdots$$
and we can choose stabilisation maps of type $\beta$ which admit a right inverse, i.e.\ where one of the created boundary conditions bounds a disc. Thus all the type $\beta$ maps can be taken to be injective in all degrees, and so isomorphisms in degrees $* \leq G(g)$. The claimed range now follows. If $b=0$ we consider instead the commutative diagram
\begin{equation*}
\xymatrix{
H_*(\mathcal{M}^\theta(\Sigma_{g,1};\ell_{\partial \Sigma_{g,1}})) \ar[r]^-{\mathscr{S}_{\Sigma_{g,1}}} \ar[d]^-{\text{type }\gamma}&  H_*^{\text{stab}}(\mathcal{M}^\theta) \ar@{=}[d]\\
H_*(\mathcal{M}^\theta(\Sigma_{g})) \ar[r]^-{\mathscr{S}_{\Sigma_g}} & H_*(\Omega_{[F]}^\infty \mathbf{MT\theta})
}
\end{equation*}
and see that the same stability range holds, as the map of type $\gamma$ is an isomorphism in degrees $* \leq G(g)$ by Theorem \ref{thm:LastBoundary}. The argument for non-orientable surfaces is the same.
\end{proof}

\begin{appendices}
\section{On complexes of arcs in surfaces}\label{app:CxArcs}

In the body of this paper we have required information on the connectivities of certain simplicial complexes which are slight modifications of those discussed in the literature. The purpose of this appendix is to deduce information about the complexes we need from that available in the work of Harer \cite{H} and of Wahl \cite{Wahl}. The complexes we need are subcomplexes of those considered by Harer and Wahl, and we give elementary arguments deducing their high-connectivity from the high-connectivity of the complexes of \cite{H, Wahl}.

\subsection{Arcs in orientable surfaces}\label{sec:CxArcsOrientable}

Let $\Sigma$ be a connected orientable surface with boundary, and let $b_0$, $b_1$ be distinct points on $\partial \Sigma$. Let $BX(\Sigma, \{b_0, b_1\}, \{b_0\})$ be Harer's simplicial complex \cite{H}, whose vertices are isotopy classes of properly embedded arcs in $\Sigma$ from $b_0$ to $b_1$ which do not disconnect $\Sigma$, and a collection of such span a simplex if they have representatives which are disjoint and do not disconnect $\Sigma$. For any simplex $\sigma \subset BX(\Sigma, \{b_0, b_1\}, \{b_0\})$, one can order the arcs clockwise at $b_0$ and anticlockwise at $b_1$, and compare these orderings. Let $B_0(\Sigma)$ denote the subcomplex consisting of those simplices where these two orderings agree. If $b_0, b_1$ lie on the same boundary component, this is the complex of the same name defined by Ivanov \cite{Ivanov}, and we recover his theorem on its connectivity. We are grateful to Nathalie Wahl for suggesting the following line of argument.

\begin{thm}\label{thm:OrientableComplex}
If $\Sigma$ has genus $g$, then $B_0(\Sigma)$ is $(g-2)$-connected.
\end{thm}
\begin{proof}
Note that the theorem is clearly true for $g \leq 1$: if $g=0$ then we require the complex to be $(-2)$-connected, which is no condition, and if $g=1$ then we require it to be $(-1)$-connected, i.e.\ non-empty, which is the case. Thus we proceed by induction on $g$.

Harer has shown that $BX(\Sigma, \{b_0, b_1\}, \{b_0\})$ is homotopy equivalent to a wedge of copies of $S^{2g-2+\partial}$, where $\partial$ is the number of boundary components containing the $b_i$. (His proof was slightly incomplete, but has been corrected by Wahl in \cite{Wahl}.) For $k \leq g-2$ let $f: S^k \to B_0(\Sigma)$ be a continuous map, which we may assume to be simplicial for some PL triangulation of $S^k$, and $\hat{f} : D^{k+1} \to BX(\Sigma, \{b_0, b_1\}, \{b_0\})$ be a nullhomotopy in $BX(\Sigma, \{b_0, b_1\}, \{b_0\})$, which we may again suppose to be simplicial for some PL triangulation $\vert K \vert \approx D^{k+1}$. We will show that $\hat{f}$ can be rechosen to have image in $B_0(\Sigma)$.

The vertices of a simplex $\sigma \subset BX(\Sigma, \{b_0, b_1\}, \{b_0\})$ may be given two orders: the clockwise ordering of the arcs at $b_0$, or the anticlockwise ordering of the arcs at $b_1$. By definition, the simplex $\sigma$ lies in $B_0(\Sigma)$ if and only if these two orderings agree. Say that $\sigma$ is \emph{bad} if the first arc with respect to the clockwise ordering at $b_0$ is not the first arc with respect to the the anticlockwise ordering at $b_1$. Note that bad simplices must have dimension at least 1. Bad simplices are not in $B_0(\Sigma)$, and conversely any simplex which is not in $B_0(\Sigma)$ has a face which is bad. Hence if $\hat{f}$ is such that for every simplex $\sigma < K$ the simplex $\hat{f}(\sigma) < BX(\Sigma, \{b_0, b_1\}, \{b_0\})$ is not bad, then $\hat{f}$ has image in $B_0(\Sigma)$. 

\vspace{2ex}

We begin with some preliminary calculations. Let $\sigma < BX(\Sigma, \{b_0, b_1\}, \{b_0\})$ be bad, and let $\Sigma \setminus \sigma$ be the surface obtained by cutting along the arcs in $\sigma$. Writing $g(X)$ for the genus of a connected orientable surface $X$, $b(X)$ for its number of boundary components, and $\vert \sigma\vert$ for the number of vertices of $\sigma$, we estimate
$$g(\Sigma) > g(\Sigma\setminus \sigma) > g(\Sigma) - |\sigma|$$
as follows: we have $\chi(\Sigma\setminus \sigma) = \chi(\Sigma)+|\sigma|$, and so using the identity $\chi(X) = 2-2g(X)-b(X)$ for a connected orientable surface $X$, we obtain
$$2(g(\Sigma)-g(\Sigma\setminus \sigma) - |\sigma|) = b(\Sigma\setminus \sigma) - b(\Sigma) -|\sigma|,$$
so the inequalities claimed are equivalent to the inequalities
$$b(\Sigma) - \vert \sigma \vert < b(\Sigma\setminus \sigma) < b(\Sigma) +|\sigma|.$$
We prove these inequalities by cases.
\begin{enumerate}[(i)]
	\item If $b_0$ and $b_1$ lie on different boundary components, then given $|\sigma|$ arcs, cutting along the first arc reduces the number of boundary components by one, and cutting along each subsequent arc at most increases the number of boundary components by one, which proves $b(\Sigma\setminus \sigma) < b(\Sigma) +|\sigma|$. On the other hand, once the first arc is cut out, the second arc has both ends on the same boundary component, so cutting it out increases the number of boundary components by one. Cutting out subsequent arcs reduces the number of boundary components by at most one each, so $b(\Sigma) - (\vert \sigma \vert-2) \leq b(\Sigma\setminus \sigma)$. (This did not require $\sigma$ to be bad.)

	\item If $b_0$ and $b_1$ lie on a single boundary component, let $a_0$ be the first arc in the clockwise ordering at $b_0$, and $a_1$ be the first arc in the anticlockwise ordering at $b_1$. As $\sigma$ is bad, $a_0 \neq a_1$. If we cut the arc $a_0$ out then we replace the boundary containing $b_0$ and $b_1$ by two boundary components, and as $a_1$ is the first arc in the anticlockwise ordering at $b_1$, on the cut surface $a_1$ gives an arc between these boundary components. Thus cutting $a_1$ out reduces the number of boundary components by one, so we have the same number of boundary components as when we started. Cutting each subsequent arc out creates or removes at most one boundary component, so $b(\Sigma) -(|\sigma|-2) \leq b(\Sigma\setminus \sigma) \leq b(\Sigma) +(|\sigma|-2)$.
\end{enumerate}

On the cut surface $\Sigma \setminus \sigma$ there are multiple copies of $b_0$ and $b_1$, but we can single out a preferred copy of each, $\tilde{b}_0$ and $\tilde{b}_1$, as follows: $\tilde{b}_0$ is the copy lying on the first (in the clockwise ordering at $b_0$) of the two edges formed by cutting along $a_0$, and $\tilde{b}_1$ is the copy lying on the first (in the anticlockwise ordering at $b_1$) of the two edges formed by cutting along $a_1$. The map $\Sigma \setminus \sigma \to \Sigma$ that glues the arcs together induces a simplicial map
$$BX(\Sigma \setminus \sigma, \{\tilde{b}_0, \tilde{b}_1\}, \{\tilde{b}_0\}) \lra BX(\Sigma, \{b_0, b_1\}, \{b_0\}),$$
and it is easy to see that this is the inclusion of a subcomplex.

\vspace{2ex}

This finishes the preliminary calculations, and we now begin the argument showing that $\hat{f}$ can be rechosen to have image in $B_0(\Sigma)$. Let $\sigma < K$ be a maximal dimensional simplex such that $\hat{f}(\sigma)$ is bad. We claim that the map
$$\hat{f}\vert_{\Link(\sigma)} : \Link(\sigma) \lra BX(\Sigma, \{b_0, b_1\}, \{b_0\})$$
in fact lands in the subcomplex
$$B_0(\Sigma\setminus \hat{f}(\sigma^b)) \subset BX(\Sigma\setminus \hat{f}(\sigma), \{\tilde{b}_0, \tilde{b}_1\}, \{\tilde{b}_0\}) \subset BX(\Sigma, \{b_0, b_1\}, \{b_0\}).$$

To see that it lands in the subcomplex $BX(\Sigma\setminus \hat{f}(\sigma), \{\tilde{b}_0, \tilde{b}_1\}, \{\tilde{b}_0\})$ we must show that when considered to lie in $\Sigma \setminus \hat{f}(\sigma)$, the arcs of $\hat{f}(\tau)$ for $\tau \in \Link(\sigma)$ start at $\tilde{b}_0$ and end at $\tilde{b}_1$. If this were not the case, then some vertex of $\tau$ can be added to $\sigma$ to give a larger simplex which is still fully bad, violating the assumed maximality of $\sigma$. Similarly, if $\hat{f}(\tau)$ does not lie in $B_0(\Sigma\setminus \hat{f}(\sigma))$ then $\hat{f}(\tau)$ has a face $\hat{f}(\tau') \leq \hat{f}(\tau)$ which is bad, but then $\hat{f}(\tau' * \sigma)$ is bad too, which again violates the maximality of $\sigma$. This proves the claim.

We have shown above that $g(\Sigma \setminus \hat{f}(\sigma)) < g(\Sigma)$, so by inductive hypothesis the complex $B_0(\Sigma \setminus \hat{f}(\sigma))$ is $(g(\Sigma \setminus \hat{f}(\sigma))-2)$-connected, and we may compute
$$g(\Sigma\setminus \hat{f}(\sigma)) > g(\Sigma) - |\hat{f}(\sigma)| \geq  k-|\hat{f}(\sigma)|+2 \geq k-|\sigma|+2$$
and so $k-|\sigma|+1 \leq  g(\Sigma\setminus \hat{f}(\sigma))-2$. As $\sigma$ is a simplex of a PL triangulation of $D^{k+1}$ which does not lie completely in the boundary (as the boundary maps to $B_0(\Sigma)$, so has no bad simplices), we have $\Link(\sigma) \approx S^{k-\vert\sigma\vert+1}$. It follows that the map $\hat{f}\vert_{\Link(\sigma)} : \Link(\sigma) \to B_0(\Sigma\setminus \hat{f}(\sigma))$ is nullhomotopic. Let us write $F : C\Link(\sigma) \to B_0(\Sigma\setminus \hat{f}(\sigma))$ for a choice of nullhomotopy, which we may suppose is simplicial with respect to some triangulation of $C\Link(\sigma)$ extending that of $\Link(\sigma)$.

We now define
$$F * \hat{f} :  C\Link(\sigma) * (\partial \sigma) \approx \mathrm{St}(\sigma) \lra BX(\Sigma, \{b_0, b_1\}, \{b_0\}),$$
a modification of $\hat{f}$ on the star of $\sigma < K$. This gives a new triangulation $\vert K' \vert \approx D^{k+1}$ and map $\hat{f}'$. Furthermore, as $\sigma$ does not lie entirely in the boundary of $D^{k+1}$ we have $\mathrm{St}(\sigma) \cap \partial D^{k+1} \subset \Link(\sigma) * (\partial \sigma)$ where the new map agrees with the old, so $\hat{f}'$ is still a nullhomotopy of $f$. The new simplices of $K'$ are of the form $\alpha * \beta$ for $\beta < \partial\sigma$ and $\alpha$ mapping through $B_0(\Sigma\setminus \sigma) \to BX(\Sigma, \{b_0, b_1\}, \{b_0\})$.

As long as $\alpha \neq \emptyset$ the first arc of $\alpha*\beta$ in the clockwise ordering at $b_0$ or the anticlockwise ordering at $b_1$ is the first arc of $\alpha$ in either of these orderings, so they are equal and hence $\alpha * \beta$ is not bad. Alternatively, if $\alpha = \emptyset$ then $\alpha * \beta = \beta < \partial \sigma$ is of strictly smaller dimension to $\sigma$. In either case, we have replaced $(K, \hat{f})$ by similar data with strictly fewer bad simplices of maximal dimension: iterating, we find a $(K', \hat{f}')$ having no bad simplices, as required.
\end{proof}

\subsection{Arcs in non-orientable surfaces}\label{sec:CxArcsNonorientable}

Firstly, recall that we say a connected non-orientable surface $S$ has genus $g$ if it is diffeomorphic to a surface obtained from $\#^g \bR\bP^2$ by removing a finite number of disjoint open discs. In this case we write $g(S)=g$.

Let $S$ be a non-orientable surface and $\vec{b}_0$, $\vec{b}_1$ be oriented points on $\partial S$ i.e.\ points with a chosen orientation of their tangent space in $\partial S$. Following Wahl \cite{Wahl}, we define a \emph{1-sided arc} from $\vec{b}_0$ to $\vec{b}_1$ to be an embedded arc from $b_0$ to $b_1$ which admits a normal orientation compatible with those of $\vec{b}_0$ and $\vec{b}_1$, and has connected non-orientable complement. Let $C(S, \vec{b}_0, \vec{b}_1)$ denote the simplicial complex with vertices the isotopy classes of 1-sided arcs from $\vec{b}_0$ to $\vec{b}_1$ which do not disconnect $S$ and have non-orientable complement, and where a collection of vertices span a simplex if they can be made disjoint, and have connected non-orientable complement. This is related to the complexes $\mathcal{G}(S, \vec{\Delta})$ of Wahl \cite{Wahl}. In Wahl's notation, $\vec{\Delta}$ denotes a set of oriented points in $\partial S$ and $\mathcal{G}(S, \vec{\Delta})$ denotes the simplicial complex whose vertices are the isotopy classes of 1-sided arcs in $S$ with ends in $\vec{\Delta}$, and where a collection of these span a simplex if they may be represented by disjoint arcs with connected non-orientable complement. Specifically, $C(S, \vec{b}_0, \vec{b}_1)$ is the full subcomplex of $\mathcal{G}(S, \{\vec{b}_0,\vec{b}_1\})$ on those arcs which go from $\vec{b}_0$ to $\vec{b}_1$.

Using the orientation of the tangent space given by $\vec{b}_i$ and the inwards normal vector, we can order the arcs clockwise or anticlockwise at each of the points $b_i$.

Let $\vec{b}_0$ and $\vec{b}_1$ lie on the same boundary component, and have coherent orientations. Let $C_0(S)$ denote the subcomplex of $C(S, \vec{b}_0, \vec{b}_1)$ where the clockwise ordering at $b_0$ coincides with the anticlockwise ordering at $b_1$. 

\begin{thm}\label{thm:MobiusComplex}
If $S$ has genus $g$, then $C_0(S)$ is $(\lfloor \frac{g-1}{3} \rfloor -1)$-connected.
\end{thm}
\begin{proof}
Consider the subcomplex $\mathcal{G}_0(S, \vec{b}_0)$ of Wahl's $\mathcal{G}(S, \vec{b}_0)$ consisting of those simplices which are ordered \textit{palindromically}: the $k$th arc in the clockwise order is the $k$th arc in the anticlockwise order, for all $k$. Recall that $\mathcal{G}(S, \vec{b}_0)$ is $(g-3)$-connected \cite[Theorem 3.3]{Wahl}.

The complex $C_0(S, \vec{b}_0, \vec{b}_1)$ is homeomorphic to $\mathcal{G}_0(S, \vec{b}_0)$ as follows. Choose a path in the boundary from $b_1$ to $b_0$, then composing arcs with this path defines a map $C_0(S, \vec{b}_0, \vec{b}_1) \to \mathcal{G}_0(S, \vec{b}_0)$ which is easily seen to be simplicial and a bijection on sets of simplices.

Note first that the theorem is trivially true for $g \leq 3$, so suppose for an induction that it holds for all genera below $g$. Let $k \leq \lfloor \frac{g-1}{3} \rfloor - 1$ and take a continuous map $f : S^k \to \mathcal{G}_0(S, \vec{b}_0)$, which we may suppose is simplicial for some PL triangulation of the $k$-sphere. The composition $S^k \to \mathcal{G}_0(S, \vec{b}_0) \to \mathcal{G}(S, \vec{b}_0)$ is nullhomotopic, by the discussion above and the inequality
$$\lfloor \tfrac{g-1}{3} \rfloor -1 \leq g-3,$$
which holds as long as $g \geq 2$, so we may choose a nullhomotopy $\hat{f}: D^{k+1} \to \mathcal{G}(S, \vec{b}_0)$, which we may again suppose to be simplicial with respect to some PL triangulation $\vert K \vert \approx D^{k+1}$. We will modify this map relative to $\partial D^{k+1}$ to have image in $\mathcal{G}_0(S, \vec{b}_0)$. 

The technique for doing so is the same as that of Theorem \ref{thm:OrientableComplex}, so we describe it is less detail, only pointing out the places where the arguments differ. Call a simplex $\sigma < \mathcal{G}(S, \vec{b}_0)$ \emph{bad} if the first arc in the clockwise order is not the first arc in the anticlockwise order. As before, it will be enough to change $K$ and $\hat{f}$ so that for every simplex $\sigma < K$, $\hat{f}(\sigma)$ is not bad.

To do this, let $\sigma < K$ be a simplex of maximal dimension so that $\hat{f}(\sigma)$ is bad, and let $a_0$ be the first arc of $\hat{f}(\sigma)$ in the clockwise order and $a_1$ be the first arc of $\hat{f}(\sigma)$ in the anticlockwise order. As $\hat{f}(\sigma)$ is bad, $a_0 \neq a_1$. Let $\vec{c}_0$ be the oriented point on the boundary of $S \setminus \hat{f}(\sigma)$ which lifts $\vec{b}_0$ and lies on the first (in the clockwise ordering) of the two edges formed by cutting along $a_0$, and let $\vec{c}_1$ be the oriented point on the boundary of $S \setminus \hat{f}(\sigma)$ which lifts $\vec{b}_0$ and lies on the first (in the anticlockwise ordering) of the two edges formed by cutting along $a_1$. The map $S \setminus \hat{f}(\sigma) \to S$ that glues the arcs together  induces a simplicial map
$$C_0(S \setminus \hat{f}(\sigma), \vec{c}_0, \vec{c_1}) \lra \mathcal{G}(S, \vec{b}_0),$$
and it is easy to see that this is the inclusion of a subcomplex.

If $\tau < \Link(\sigma)$ then $\hat{f}(\tau) < C_0(S \setminus \hat{f}(\sigma), \vec{c}_0, \vec{c_1}) < \mathcal{G}(S, \vec{b}_0)$, as in the proof of Theorem \ref{thm:OrientableComplex}, so we have $\hat{f}\vert_{\Link(\sigma)} : \Link(\sigma) \to C_0(S \setminus \hat{f}(\sigma), \vec{c}_0, \vec{c_1})$. Writing $g(X)$ for the genus of a connected non-orientable surface $X$, the required estimates in this case are
$$g(S) > g(S \setminus \hat{f}(\sigma)) \geq g(S)-2|\hat{f}(\sigma)|+1 \geq g(S)-2|\sigma|+1,$$
which follow as removing the first arc loses a single genus, and removing subsequent arcs loses at most two genera per arc. As $g(S \setminus \hat{f}(\sigma)) < g(S)=g$, by inductive hypothesis it follows that $C_0(S \setminus \hat{f}(\sigma), \vec{c}_0, \vec{c_1}) \cong \mathcal{G}_0(S \setminus \hat{f}(\sigma), \vec{c}_0)$ is $(\lfloor \frac{g(S \setminus \hat{f}(\sigma))-1}{3} \rfloor-1)$-connected. We calculate
$$3(k+1-|\sigma|) \leq g(S)-1-3|\sigma| \leq g(S \setminus \hat{f}(\sigma)) - |\sigma|-2 \leq g(S \setminus \hat{f}(\sigma)) - 4,$$
as we must have $|\sigma| \geq 2$ for $\hat{f}(\sigma)$ to be bad. As $\Link(\sigma) \approx S^{k+1-\vert\sigma\vert}$, it follows that the map $\hat{f}\vert_{\Link(\sigma)}$ is nullhomotopic. We then finish as in the proof of Theorem \ref{thm:OrientableComplex}.
\end{proof}

If $\vec{b}_0$ and $\vec{b}_1$ lie on the same boundary component, and have opposite orientations, let $D_0(S, \vec{b}_0, \vec{b}_1)$ denote the subcomplex of $C(S, \vec{b}_0, \vec{b}_1)$ where the clockwise ordering at $\vec{b}_0$ coincides with the clockwise ordering at $\vec{b}_1$. 

Let $c : [0,1] \hookrightarrow \partial S$ be an embedded interval in the same boundary component as the $b_i$. Forming the boundary connect sum with $S_{h,1}$ inside this interval gives a direct system of simplicial complexes
$$D_0(S, \vec{b}_0, \vec{b}_1) \lra D_0(S \natural S_{1,1}, \vec{b}_0, \vec{b}_1) \lra D_0(S \natural S_{2,1}, \vec{b}_0, \vec{b}_1) \lra \cdots.$$

\begin{thm}\label{thm:NonorientableHandleComplex}
The space $\underset{h\to\infty}\hocolim D_0(S \natural S_{h,1}, \vec{b}_0, \vec{b}_1)$ is contractible.
\end{thm}
\begin{proof}
We claim that the double stabilisation map
$$D_0(S, \vec{b}_0, \vec{b}_1) \lra D_0(S \natural S_{2,1}, \vec{b}_0, \vec{b}_1)$$
is nullhomotopic: in fact, we claim that it has image in the link of a particular vertex. To construct this vertex, consider the submanifold $W \subset S \natural S_{2,1}$ given by $S_{2,1}$ and a regular neighbourhood of the component of $\partial S$ containing the $b_i$. This is abstractly diffeomorphic to $S_{2,2}$, and contains a 1-sided arc $v$ from $\vec{b}_0$ to $\vec{b}_1$ as shown in Figure \ref{fig:ThmA3}, which gives a vertex $v \in D_0(S \natural S_{2,1}, \vec{b}_0, \vec{b}_1)$. (Strictly speaking $v$ is not a 1-sided arc in $W$, as its complement is orientable, but in $S \natural S_{2,1}$ it is a 1-sided arc.)

\begin{figure}[h]
\centering
\includegraphics[bb = 0 0 224 106]{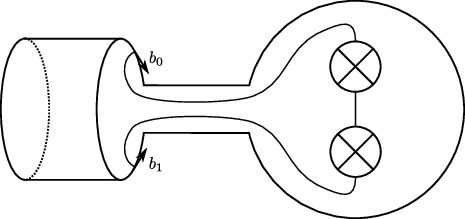}
\caption{}\label{fig:ThmA3}
\end{figure}

If $\sigma = \{v_0, v_1, \ldots, v_p\} \subset D_0(S, \vec{b}_0, \vec{b}_1)$ is a simplex then we may choose representatives of the arcs $v$ and  $v_i$ which are disjoint in $S \natural S_{2,1}$, as the $v_i$ may be pushed out of the $S_{2,1}$ part. It is clear that cutting out the arcs $\{v, v_0, v_1, \ldots, v_p\}$ leaves a connected non-orientable complement, as cutting out the $v_i$ from $S$ does, and cutting out $v$ from $W$ leaves a connected complement. Furthermore, $v$ either comes first in the clockwise order at $\vec{b}_0$ and $\vec{b}_1$, or comes last in both of them (this depends on whether the interval $c([0,1]) \subset \partial S$ has the orientated points $\vec{b}_i$ pointing towards it, as in Figure \ref{fig:ThmA3}, or away from it). In any case, $\sigma * v \in D_0(S \natural S_{2,1}, b_0, b_1)$, as claimed.
\end{proof}

Let $\vec{b}_0$ and $\vec{b}_1$ lie on different boundary components. Let $E_0(S, \vec{b}_0, \vec{b}_1)$ denote the subcomplex of $C(S, \vec{b}_0, \vec{b}_1)$ where the clockwise ordering at $\vec{b}_0$ coincides with the clockwise ordering at $\vec{b}_1$. 

Let $c : [0,1] \hookrightarrow \partial S$ be an embedded interval in the same boundary component as $b_0$. Forming the boundary connect sum with $S_{h,1}$ inside this interval gives a direct system of simplicial complexes
$$E_0(S, \vec{b}_0, \vec{b}_1) \lra E_0(S \natural S_{1,1}, \vec{b}_0, \vec{b}_1) \lra E_0(S \natural S_{2,1}, \vec{b}_0, \vec{b}_1) \lra \cdots.$$

\begin{thm}\label{thm:NonorientableBoundaryComplex}
The space $\underset{h \to \infty}\hocolim E_0(S \natural S_{h,1}, \vec{b}_0, \vec{b}_1)$ is contractible.
\end{thm}
\begin{proof}
The simplicial complex $E_0(S, \vec{b}_0, \vec{b}_1)$ is the subcomplex of Wahl's complex $\mathcal{G}(S, \{\vec{b}_0, \vec{b}_1\})$ of those arcs which go from $b_0$ to $b_1$ and satisfy the ordering condition. By \cite[Theorem 3.5]{Wahl} the space $\underset{h\to\infty}\hocolim \mathcal{G}(S \natural S_{h,1}, \{\vec{b}_0, \vec{b}_1\})$ is contractible.

We will show that $\underset{h\to\infty}\hocolim E_0(S \natural S_{h,1}, \vec{b}_0, \vec{b}_1)$ is $k$-connected by induction on $k$: certainly it is $(-1)$-connected (i.e.\ nonempty), which begins the induction. Any element of $\pi_k(\underset{h\to\infty}\hocolim E_0(S \natural S_{h,1}, \vec{b}_0, \vec{b}_1))$ may be represented by a map 
$$f : S^k \lra E_0(S \natural S_{h,1}, \vec{b}_0, \vec{b}_1)$$
which is simplicial with respect to some PL triangulation of the sphere, and by Wahl's theorem we may extend this to a map $\hat{f} : D^{k+1} \to \mathcal{G}(S \natural S_{h,1}, \{\vec{b}_0, \vec{b}_1\})$ after perhaps increasing $h$, which may also suppose to be simplicial with respect to a PL triangulation $\vert K \vert \approx D^{k+1}$. 

Call a simplex $\sigma < \mathcal{G}(S \natural S_{h,1}, \{\vec{b}_0, \vec{b}_1\})$ \emph{bad} if the first arc, $a_0$, in the clockwise ordering at $\vec{b}_0$ is not the first arc, $a_1$, in the clockwise ordering at $\vec{b}_1$. (Note that this can happen for two reasons: some $a_i$ might have both ends at $b_i$, or else the $a_i$ both go from $b_0$ to $b_1$.) As in the previous arguments, it is enough to ensure that for each $\sigma < K$, $\hat{f}(\sigma)$ is not bad, and we proceed in the same manner by giving a technique for reducing the number of maximal dimensional bad simplices of $K$ by changing $K$ and the map $\hat{f}$. When changing these data, we are also allowed to increase $h$.

We proceed as in the proof of Theorem \ref{thm:OrientableComplex} and Theorem \ref{thm:MobiusComplex}, with the following extra fundamental observation: if $\hat{f}(\sigma) < \mathcal{G}(S\natural S_{h,1}, \{\vec{b}_0, \vec{b}_1\})$ is a bad simplex then it must consist of at least 2 arcs, so $\sigma < K$ is a simplex of dimension at least 1 which does not lie entirely in the boundary, so $\Link(\sigma)$ is homeomorphic to a sphere of dimension at most $(k-1)$. Thus, after perhaps increasing $h$, we may suppose that the map
$$\hat{f}\vert_{\Link(\sigma)} : \Link(\sigma) \lra E_0((S \natural S_{h,1}) \setminus \hat{f}(\sigma), \vec{b}_0, \vec{b}_1) < \mathcal{G}(S\natural S_{h,1}, \{\vec{b}_0, \vec{b}_1\})$$
is nullhomotopic.
\end{proof}

\end{appendices}

\bibliographystyle{plain}
\bibliography{biblio}  

\end{document}